\documentclass[a4paper,12pt]{article}

\usepackage{authblk}

\usepackage[T1]{fontenc}
\usepackage[latin1]{inputenc}
\usepackage[english]{babel}
\usepackage{geometry}
\usepackage{indentfirst} 
\title{Clifford Tori and the singularly perturbed Cahn-Hilliard equation}
\date{September 03,2015}
\author{Matteo Rizzi\thanks{mrizzi@sissa.it}}
\affil{SISSA, via bonomea 265, 34136, Trieste, Italy.}

\usepackage{amsmath,amsfonts,amssymb,amsthm}
\usepackage{latexsym}

\newcommand{\N}{\mathbb{N}}

\newcommand{\R}{\mathbb{R}}

\newcommand{\tr}{\text{tr}}

\newtheorem{theorem}{Theorem}
\newtheorem{proposition}[theorem]{Proposition}
\newtheorem{lemma}[theorem]{Lemma}

\newtheorem{remark}[theorem]{Remark}

\begin{document}

\maketitle

\begin{abstract}
In this paper we construct entire solutions $u_{\varepsilon}$ to the Cahn-Hilliard equation
$-\varepsilon^{2}\Delta(-\varepsilon^{2}\Delta u+W^{'}(u))+W^{''}(u)(-\varepsilon^{2}\Delta u+W^{'}(u))=0$, 
under the volume constraint $\int_{\R^{3}}(1-u_{\varepsilon})dx=4\sqrt{2}\pi^{2}$, whose nodal set approaches the Clifford Torus, that is the Torus with radii of ratio $1/\sqrt{2}$ embedded in $\R^{3}$, as $\varepsilon\to 0$. What is crucial is that the Clifford Torus is a Willmore hypersurface and it is non-degenerate, up to conformal transformations. The proof is based on the Lyapunov-Schmidt reduction and on careful geometric expansions of the laplacian.
\end{abstract}
\tableofcontents
\section{Introduction}
The Allen-Cahn equation 
\begin{eqnarray}
-\varepsilon^{2}\Delta u=u-u^{3},
\label{allen-cahn} 
\end{eqnarray}
arises in several physical contexts, such as the study of the stable configurations of two different fluids confined in a bounded container $\Omega$. If $u(x)$ is the density of one of the two fluids at a point $x\in\Omega$ and the energy per unit volume is given by a function $W$ of $u$, it looks reasonable to obtain stable configurations by minimizing the energy functional
\begin{eqnarray}\notag
E(u)=\int_{\Omega}W(u)dx
\end{eqnarray}
among all distributions fulfilling the volume constraint
\begin{eqnarray}
\int_{\Omega}udx=m.
\label{constraint}
\end{eqnarray}
If, for instance, $W(u)=(1-u^{2})^{2}$, and $m\in(-1,1)$, any piecewise constant function taking only the values $\pm 1$ and satisfying  (\ref{constraint}) is a minimizer, irrespectively of the shape of the interface. Therefore this model is unsatisfactory, since it is very far from the reasonable physical assumption that the interfaces are area minimizers, so one replaces the energy by
\begin{eqnarray}\notag
E_{\varepsilon}(u)=\int_{\Omega}\bigg(\frac{\varepsilon}{2}|\nabla u|^{2}+\frac{(1-u^{2})^{2}}{4\varepsilon}\bigg)dx.
\end{eqnarray}
We can see that there is a competition between the potential energy, that forces $u$ to be close to $\pm 1$, and the gradient term that penalizes the phase transition. By minimizing this functional, we are looking for the physical interfaces in which the phase transition can occur. 

The minimizers $u_{\varepsilon}$ of $E_{\varepsilon}$ are solutions to the Euler Lagrange equation, that is (\ref{allen-cahn}). In order to see if the interfaces are actually minimal surfaces, it is interesting to study the asymptotic behaviour of the level sets $\{u_{\varepsilon}=c\}$ as the parameter $\varepsilon\to 0$. It is useful to exploit the variational structure of the problem.
It was shown by Modica and Mortola that the energy $E_{\varepsilon}$, seen as a functional on $L^{1}(\Omega)$ and extended to be $+\infty$ when the integrand is not an $L^{1}$ function, $\Gamma-$converges to the functional
\begin{eqnarray}\notag
E(u)=
\begin{cases}
cPer_{\Omega}(\{u=1\})&\text{if $u=\pm 1$ a.e. in }\Omega\\\notag
+\infty &\text{otherwise in }L^{1}(\Omega)
\end{cases}
\end{eqnarray}
in the strong topology of $L^{1}(\Omega)$ (see \cite{MM}), where $c>0$ is a suitable constant. 

Moreover, Modica showed that, if $u_{\varepsilon}$ are minimizers of $F_{\varepsilon}$ under the volume constraint
\begin{eqnarray}\notag
\int_{\Omega}u_{\varepsilon}dx=m,
\end{eqnarray} 
for some $m\in(-1,1)$, then there exists a sequence $\varepsilon_{k}\to 0$ such that $u_{\varepsilon_{k}}$ converges to some function $u$ in $L^{1}(\Omega)$ (see proposition $3$ of \cite{Mo}). Furthermore, Theorem $1$ of \cite{Mo} asserts that $u=\pm 1$ a. e. in $\Omega$, and the set
\begin{eqnarray}\notag
E=\{x\in\Omega:u(x)=1\}
\end{eqnarray}
is actually a perimeter minimizer between all the subsets $F\subset\Omega$ satisfying the volume constraint
\begin{eqnarray}\notag
|F|=\frac{|\Omega|+m}{2}.
\end{eqnarray}
Further results about the relation between the minimizers of $E_{\varepsilon}$ and the minimizers of the perimeter can be found in \cite{Mo} and in \cite{CS}, where Choksi and Sternberg also described the relation between phase transition theory and the study of a certain kind of polymers.

Conversely, it is an interesting problem to understand if any minimal hypersurface can be achieved as the limit of nodal sets of minimizers of the Ginzburg-Landau energy $E_{\varepsilon}$. 

The first result in this direction is due to Kohn and Sternberg (see \cite{KS}). They considered a smooth bounded domain $\Omega\subset\R^{2}$ and, as an interface, a disjoint union of segments $l_{i}$ meeting the boundary $\partial\Omega$ orthogonally. They defined $u_{0}$ to be locally constant on $\Omega\backslash\cup_{i}l_{i}$, taking the values $\pm 1$, and constructed a sequence of minimizers $u_{\varepsilon}$ converging to $u_{0}$ in $L^{1}(\Omega)$.

In \cite{PR}, Pacard and Ritoré proved a more general result, that holds true for a larger class of interfaces. They started from a minimal hypersurface $\Sigma$ in a compact Riemannian manifold $M$ and, under suitable assumptions, they showed that it can be achieved as the limit as $\varepsilon\to 0$ of nodal sets (that is $0$-level sets) of solutions $u_{\varepsilon}$ of the rescaled Allen-Cahn equation (\ref{allen-cahn}). These solutions $u_{\varepsilon}$ were constructed with techniques such as fixed point theorems and the Lyapunov-Schmidt reduction, and are not necessarily minimizers. 

As regards the hypersurface $\Sigma$, they imposed some restrictions. They required it to be \textit{admissible}, that is the nodal set of a smooth function $f:M\to\R$. In the sequel, we will set
\begin{eqnarray}\notag
M^{+}(\Sigma)=\{p\in M:f(p)>0\}&\text{and }M^{-}(\Sigma)=\{p\in M:f(p)<0\}.
\end{eqnarray}
Moreover, $\Sigma$ has to be \textit{non-degenerate}. 
In order to explain the notion of non-degeneracy, let us give the variational characterization of minimal hypersurfaces. A hypersurface $\Sigma$ in a compact Riemannian manifold $M$ is said to be minimal if it is a minimizer for the area functional, whose critical points are characterized by the Euler equation $H=0$, where $H$ denotes the mean curvature of $\Sigma$. In the sequel, the mean curvature $H$ of a hypersurface $\Sigma$ embedded in $\R^{N}$ will always be
\begin{eqnarray}\notag
H=k_{1}+\dots+k_{N-1},
\end{eqnarray}
where the $k_{j}$'s are the principal curvatures.

The second variation of the area functional is given by
\begin{eqnarray}\notag
A^{''}(\Sigma)[\phi,\psi]=\int_{\Sigma} L_{0}\phi(y)\psi(y)d\sigma(y),
\end{eqnarray} 
where the self-adjoint operator
\begin{eqnarray}\notag
L_{0}\phi=-\Delta_{\Sigma}\phi-|A|^{2}\phi
\end{eqnarray}
is called the Jacobi operator of $\Sigma$ and 
\begin{eqnarray}\notag
|A|^{2}=k_{1}^{2}+\dots+k_{N-1}^{2}
\end{eqnarray}
is the squared norm of its second fundamental form. By definition, a minimal hypersurface $\Sigma$ is said to be non-degenerate if its Jacobi operator
\begin{eqnarray}\notag
L_{0}:C^{2,\alpha}(\Sigma)\to C^{0,\alpha}(\Sigma)
\end{eqnarray}
is an isomorphism. For an introduction to these topics, see also \cite{dPW}.

Moreover, the results in \cite{PR} hold even if the potential $W(t)=(1-t^{2})^{2}/4$ is replaced by a more general double-well potential, that is a smooth function $W$ such that
\begin{eqnarray}
\begin{cases}
W(t)\geq 0 &\text{for any }t,\\
W(t)=0 &\text{if and only if }t=\pm 1,\\
W^{''}(\pm 1)>0.
\end{cases}
\label{double-well}
\end{eqnarray}
To sum up, they proved the following Theorem.
\begin{theorem}[\cite{PR}]
Let $W$ be as in (\ref{double-well}). Let $\Sigma$ be an admissible non-degenerate minimal hypersurface in a compact Riemennian manifold $M$. Then there exists $\varepsilon_{0}>0$ such that for any $0<\varepsilon<\varepsilon_{0}$ there exists a solution $u_{\varepsilon}$ to the rescaled Allen-Cahn equation 
\begin{eqnarray}\notag
-\varepsilon^{2}\Delta u_{\varepsilon}+W^{'}(u_{\varepsilon})=0
\end{eqnarray}
such that $u_{\varepsilon}\to\pm 1$ on compact subsets of $M^{\pm}(\Sigma)$.
\label{thpr}
\end{theorem}
Anyway, despite several results lead to think that, in some sense, the nodal sets of the solutions to the Allen-Cahn equation resemble minimal surfaces, there are also solutions for which the nodal set is far from being minimal. For instance, Agudelo, Del Pino and Wei constructed axially symmetric solutions $u=u(|x^{'}|,x_{3})$ in $\R^{3}$ such that the components of the nodal set, for $|x^{'}|$ large enough, look like a catenoid (see \cite{ADpW}).

The Lyapunov-Schmidt reduction was also applied to the non compact case, to construct entire solutions to the Allen-Cahn equation in $\R^{9}$ that are monotone in one variable but not one-dimensional, since their nodal set resembles the Bombieri-De Giorgi-Giusti graph, that is a minimal graph over $\R^{8}$ that is not affine (see \cite{BDG},\cite{dPKW}). This solutions are related to a famous conjecture of De Giorgi, that asserts that, at least for $N\leq 8$, any entire bounded solution $|u|<1$ to the Allen-Cahn equation
\begin{eqnarray}\notag
-\Delta u=u-u^{3}
\end{eqnarray}
satisfying $\partial_{N}u>0$ in the whole $\R^{N}$ must be one-dimensional, that is it must depend just on one euclidean variable, in other words $u(x)=u(<a,x>)$, for some unit vector $a\in S^{N-1}$. The result by Del Pino, Kowalczyk and Wei shows that de Giorgi's conjecture is sharp about the upper bound on the dimension. Up to now it is known that the conjecture is true in dimension $N=2$ (see \cite{GG},\cite{FSV}) and $N=3$ (see \cite{AC},\cite{FSV}). The conjecture is still open in dimension $4\leq N\leq 8$, although notable progress was made by Savin (see \cite{S}), that proved that the conjecture is true in dimension $4\leq N\leq 8$ under the reasonable assumption that, for any $x^{'}\in\R^{N-1}$,
\begin{eqnarray}\notag
\lim_{x_{N}\to\pm\infty}u(x^{'},x_{N})=\pm 1,
\end{eqnarray}
that yields that these solutions are minimizers of the energy
\begin{eqnarray}\notag
E(u)=\int_{\R^{3}}\big(\frac{1}{2}|\nabla u|^{2}+\frac{1}{4}(1-u^{2})^{2}\big)dx.
\end{eqnarray}

We are interested here in analogues of these results for the Cahn-Hilliard equation
\begin{eqnarray}
-\varepsilon^{2}\Delta(-\varepsilon^{2}\Delta u+W^{'}(u))+W^{''}(u)
(-\varepsilon^{2}\Delta u+W^{'}(u))=0,
\label{Cahn-Hilliard}
\end{eqnarray}
with $W$ satisfying (\ref{double-well}). Note that, as in the case of Allen-Cahn, we rescale the equation in order to treat $\Gamma$-convergence. If, for instance, we study the equation in a bounded domain $\Omega\subset\R^{N}$, it is possible to see that it is the Euler equation of the functional
\begin{eqnarray}\notag
\mathcal{W}_{\varepsilon}(u)=
\begin{cases}
\frac{1}{2\varepsilon}\int_{\Omega}\big(
\varepsilon\Delta u-\frac{W^{'}(u)}{\varepsilon}\big)^{2}dx &\text{ if $u\in L^{1}(\Omega)\cap H^{2}(\Omega)$}\\\notag
+\infty &\text{otherwise in $L^{1}(\Omega)$.}
\end{cases}
\end{eqnarray}
As in the case of the functionals $E_{\varepsilon}$ related to the Allen-Cahn equation, some $\Gamma-$convergence results are known about $\mathcal{W}_{\varepsilon}$. More precisely, the asymptotic behaviour of $\mathcal{W}_{\varepsilon}$ as $\varepsilon\to 0$ is related to the Willmore functional
\begin{eqnarray}\notag
\mathcal{W}(u)=
c\int_{\partial E\cap\Omega}H_{\partial E}^{2}(y)d\mathcal{H}^{N-1},
\end{eqnarray}
where $E=\{u=1\}$, if $u=\pm 1$ a. e., defined when the interface $\partial E$ is smooth enough. The nodal sets of the critical points $u$ of $\mathcal{W}$ are called \textit{Willmore hypersurfaces}. The Euler equation satisfied by this kind of hypersurfaces is
\begin{eqnarray}\notag
-\Delta_{\Sigma}H=\frac{1}{2}H^{3}-2HK,
\label{willmoreeq}
\end{eqnarray}
where $H$ is the mean curvature and $K$ is the Gauss curvature of $\Sigma=\partial E$. In the sequel, the Gauss curvature $K$ of hypersurface $\Sigma$ embedded in $\R^{N}$ will always be
\begin{eqnarray}\notag
K=k_{1}\dots k_{N-1}.
\end{eqnarray} 
An equivalent form of the Willmore equation is
\begin{eqnarray}
-\Delta_{\Sigma}H+\frac{1}{2}H(H^{2}-2|A|^{2})=0.
\end{eqnarray}
The Willmore functional arises naturally in general relativity, since it is related to the Hawking mass, that is
\begin{eqnarray}\notag
m_{H}(\Sigma)=\sqrt{\frac{Area(\Sigma)}{16\pi}}(1-\frac{1}{16\pi}\mathcal{W}(\Sigma)).
\end{eqnarray}
Here $\Sigma$ can be interpreted as the surface of a body whose mass has to be measured. Furthermore, this functional is also appears in biology, under the name of \textit{Helfrich energy}, and it is used to describe the behaviour of some lipid bilayer cell membranes. For further details and references, we suggest to see \cite{LMS,IMM1,IMM2}. 
\\

In \cite{BP} Bellettini and Paolini proved the $\Gamma-\limsup$ inequality for smooth Willmore hypersurfaces, while the $\Gamma-\liminf$ inequality is much harder to prove. Up to now it has been proved in dimension $N=2,3$ by Röger and Schätzle in \cite{RS}, and, independently, in dimension $N=2$, by Nagase and Tonegawa in \cite{NT}. The problem is still open in higher dimension, while it is known that the approximation does not hold, in general, for non smooth sets, even in dimension $N=2$.
\\

In view of these $\Gamma-$convergence results that establish a link between the Cahn-Hilliard functional and the Willmore functional, it is interesting to see if also the above counter-part is true. In other words, we try to answer the following question: given a Willmore hypersurface $\Sigma$, is it possible to construct a sequence of solutions $u_{\varepsilon}$ of the Cahn-Hilliard equation (\ref{Cahn-Hilliard}) whose nodal sets approach $\Sigma$ as $\varepsilon\to 0$? In the paper, we show that this result holds true if, for instance, $\Sigma$ is the standard Clifford Torus, that is the zero level set of the function
\begin{eqnarray}
f(x)=\Big(\sqrt{2}+\sqrt{x_{1}^{2}+x_{2}^{2}}\Big)^{2}+x_{3}^{2}-1.
\label{def_f}
\end{eqnarray}
It has been recently proved in \cite{MN} that the Clifford Torus is the unique minimizer of the Willmore energy (up to confromal transformations) among surfaces of genus greater or equal than $1$.

It is interesting to see that it is possible to construct these solutions in such a way that they respect the symmetries of the Torus, that is the symmetry with respect to the $x_{1}x_{2}$-plane and with respect to any rotation that fixes the $x_{3}$-axis. 
\begin{theorem}
Let $W$ be an even double-well potential satisfying (\ref{double-well}). Let $\Sigma$ be the Clifford Torus. Then there exists $\varepsilon_{0}$ such that for any $0<\varepsilon<\varepsilon_{0}$ there exists a solution $u_{\varepsilon}$ to (\ref{Cahn-Hilliard}) satisfying the volume constraint 
\begin{eqnarray}
\int_{\R^{3}}(1-u_{\varepsilon})dx=4\sqrt{2}\pi^{2},
\label{constraint_vol}
\end{eqnarray} 
with $u_{\varepsilon}\to\pm 1$ and $\partial_{k}u_{\varepsilon}\to 0$ uniformly on compact subsets of $\Sigma^{\pm}$, for $1\leq k\leq 4$. Moreover, $u_{\varepsilon}(x_{1},x_{2},x_{3})=u_{\varepsilon}(x_{1},x_{2},-x_{3})$ and $u_{\varepsilon}(x)=u_{\varepsilon}(Rx)$, for any $x=(x_{1},x_{2},x_{3})\in\R^{3}$ and for any rotation $R\in SO(3)$ such that $R(0,0,1)=(0,0,1)$.
\label{mainth}
\end{theorem}
In the statement of the Theorem, we denoted 
\begin{eqnarray}\notag
\Sigma^{+}=\{x\in\R^{3}:f(x)>0\} &\text{and }\Sigma^{-}=\{x\in\R^{3}:f(x)<0\} 
\end{eqnarray}
This result is a fourth order analogue of Theorem \ref{thpr} by Pacard and Ritoré (see \cite{PR}). The proof is based on the Lyapunov-Schmidt reduction, that is we split equation (\ref{Cahn-Hilliard}) into a system of two equations. The auxiliary equation will be solved by using the spectral decomposition of the linearized Allen-Cahn operator and the bifurcation equation will be solved thanks to the $nondegeneracy$ of the Clifford Torus, up to conformal maps. For a more detailed introduction to the techniques developed in the proof, see section $2$.

In order to explain what we mean by nondegeneracy, we go back to the variational definition of Willmore hypersurface and we consider the second variation of the Willmore functional, that is
\begin{eqnarray}\notag
\mathcal{W}^{''}(\Sigma)[\phi,\psi]=\int_{\Sigma}\tilde{L}_{0}\phi\psi d\sigma,
\end{eqnarray}
where $\tilde{L}_{0}$ is the self-adjoint operator given by
\begin{eqnarray}
\tilde{L}_{0}\phi=L_{0}^{2}\phi+\frac{3}{2}H^{2}L_{0}\phi-H(\nabla_{\Sigma}\phi,\nabla_{\Sigma}H)+2(A\nabla_{\Sigma}\phi,\nabla_{\Sigma}H)+\\\notag
2H<A,\nabla^{2}\phi>+\phi(2<A,\nabla^{2}H>+|\nabla_{\Sigma}H|^{2}+2H\tr A^{3}).
\label{willmore_lin}
\end{eqnarray}
Here we have denoted by $(\cdotp,\cdotp)$ the scalar product induced by the metric $g$ on $\Sigma$, indeed, for instance $(\nabla\phi,\nabla H)=g^{ij}H_{i}\phi_{j}$, and by $<\cdotp,\cdotp>$ the trace of the product of two matrices, so for instance $<A,\nabla^{2}\phi>=A^{ij}\nabla^{2}_{ij}\phi$, and $A^{ij}=g^{ik}g^{jl}A_{kl}$. It is possible to find the explicit computation of the first and the second variation of the Willmore functional $\mathcal{W}$ in \cite{LMS}, section $3$. 
This is the analogue of the Jacobi operator in the case of minimal hypersurfaces. In view of a result by White \cite{Wh}, the Willmore functional is invariant under conformal transformations of the Euclidean space, that is homotheties, isometries and M\"{o}bius transformations, i.e. inversions with respect to spheres. On the other hand, by Corollary $2$, page $34$, of \cite{W}, we know that its second variation is positive definite on the orthogonal complement of the space of conformal transformations, hence the kernel of $\tilde{L}_{0}$ exactly consists of these transformations. 
\begin{remark}
In view of the above discussion, $\tilde{L}_{0}$ is injective if restricted to the space of functions with zero average and fulfilling the symmetries of the Torus, that is the symmetry with respect to the $x_{1}x_{2}$-plane and with respect to all rotations of $\R^{3}$ that fix the $x_{3}$ axis.
\label{remtilde_L_0}
\end{remark}
In fact, by considering just functions with zero average we exclude non trivial homothethies. This constraint is equivalent to prescribe the integral of $1-u_{\varepsilon}$, that is to impose
\begin{eqnarray}\notag
\int_{\R^{3}}(1-u_{\varepsilon})dx=4\sqrt{2}\pi^{2}=2|\Sigma^{+}|_{3},
\end{eqnarray}
where $|\Sigma^{+}|_{3}=2\sqrt{2}\pi^{2}$ is the volume of the interior of the Clifford Torus, that is its $3$-dimensional Lebesgue measure. In principle, a Lagrange multiplier $\lambda_{\varepsilon}$ should appear in our equation: Anyway this will turn out to be $0$ (see Section $7$). By imposing rotational symmetry and symmetry with respect to the plane $x_{1}x_{2}$ we exclude non trivial isometries and M\"{o}bius transformations.\\

\textbf{Acknowledgments} The author is supported by the PRIN project \textit{Variational and perturbative aspects of nonlinear differential problems.} The author is also particularly grateful to F. Mahmoudi, M. Del Pino, M. Kowalckyk and M. Saez for their kind hospitality and for their precious collaboration.

\section{Some useful facts in differential geometry}
For $0<\varepsilon\leq 1$, we define the rescaled Clifford Torus as $\Sigma_{\varepsilon}:=\{\varepsilon^{-1}\zeta:\zeta\in\Sigma\}$. In other words, $\Sigma_{\varepsilon}=\{y\in\R^{3}:f_{\varepsilon}(y)=0\}$, where $f_{\varepsilon}(y):=\varepsilon^{-2}f(\varepsilon y)$ and $f$ is defined in (\ref{def_f}).

For $0<\tau<\sqrt{2}-1$ and $0<\varepsilon\leq 1$, we define the tubular neighbourhood of width $\tau/\varepsilon$ of $\Sigma_{\varepsilon}$ as 
\begin{eqnarray}\notag
V_{\tau/\varepsilon}=\{x\in\R^{3}:dist(x,\Sigma_{\varepsilon})<\tau/\varepsilon\}.
\end{eqnarray}
On this neighbourhood of $\Sigma_{\varepsilon}$, we introduce a new system of coordinates, known as Fermi coordinates. First we define
\begin{eqnarray}\notag
Z_{\varepsilon}:\Sigma_{\varepsilon}\times(-\tau/\varepsilon,\tau/\varepsilon)\to V_{\tau/\varepsilon}
\end{eqnarray} 
by the relation
\begin{eqnarray}
Z_{\varepsilon}(y,z)=\exp_{y}(z\nu(\varepsilon y))=y+z\nu(\varepsilon y),
\label{Fermi_coord}
\end{eqnarray}
where $\nu(\varepsilon y)$ is the outward-pointing unit normal to the original Torus $\Sigma$ at $\varepsilon y$, that coincides with the the outward-pointing unit normal to $\Sigma_{\varepsilon}$ at $y$, and $\exp_{y}$ is the exponential map of $\R^{3}$ at $y$ seen as a point of $\R^{3}$. If $\tau$ is small enough, that is $0<\tau<\sqrt{2}-1$ in the case of the Clifford Torus, $Z_{\varepsilon}$ is a diffeomorphism. In other words, $Z_{\varepsilon}$ is a change of coordinates on $V_{\tau/\varepsilon}$, and 
the coordinates $(y,z)=Z_{\varepsilon}^{-1}(x)$ are known as Fermi coordinates of the rescaled torus $\Sigma_{\varepsilon}$, or stretched Fermi coordinates of the Torus. 
\begin{remark}
Any function $u:V_{\tau/\varepsilon}\to\R$ can be seen as a function of $(y,z)$. More precisely, we can consider the composition $u^{\star}(y,z)=u(Z_{\varepsilon}(y,z))$. In the sequel, with a slight abuse of notation, we will write $u=u(y,z)$.
\label{remFc}
\end{remark}
Let us fix a point $\zeta_{0}\in\Sigma$ and a parametrization onto a neighbourhood $V\subset\Sigma$ of $\zeta_{0}$, that is a smooth function
\begin{eqnarray}\notag
Y:U\to V
\end{eqnarray}
on an open set $U\subset\R^{2}$ such that $Y(\xi_{0})=\zeta_{0}$, for some $\xi_{0}\in U$. Then, setting $U_{\varepsilon}=\varepsilon^{-1}U$ and $V_{\varepsilon}=\varepsilon^{-1}V$, the function
\begin{eqnarray}\notag
Y_{\varepsilon}:U_{\varepsilon}\to V_{\varepsilon}
\end{eqnarray}
given by $Y_{\varepsilon}(\text{y}):=\varepsilon^{-1}Y(\varepsilon\text{y})$ is a parametrization of $\Sigma_{\varepsilon}$. In the sequel, we will denote by y the points in $U_{\varepsilon}$ and by $y=Y_{\varepsilon}(\text{y})$ the points in $V_{\varepsilon}$. For any $|z|<\tau/\varepsilon$, we consider the surface
\begin{eqnarray}
\Sigma_{\varepsilon,z}:=\{y+z\nu(\varepsilon y),y\in\Sigma_{\varepsilon}\}.
\end{eqnarray} 
On this surface, we consider the parametrization
\begin{eqnarray}
X_{\varepsilon}(\text{y},z):=Y_{\varepsilon}(\text{y})+z\nu(\varepsilon Y_{\varepsilon}(\text{y})).
\label{param}
\end{eqnarray}
In particular, $X:=X_{1}$ is a parametrization of $\Sigma_{z}:=\Sigma_{1,z}$, the omothetic surface to $\Sigma$ at distance $z$. It is known that the tangent vectors $\{\partial_{i}X_{\varepsilon}(\text{y},z)\}_{i=1,2}$ constitute a basis of the tangent space $T_{y+z\nu(\varepsilon y)}\Sigma_{\varepsilon,z}$, that will be referred to as the standard basis. We define the coefficients of the metric of $\Sigma_{\varepsilon,z}$ at $y+z\nu(\varepsilon y)$ as follows
\begin{eqnarray}
\tilde{g}_{\varepsilon,ij}(y,z):=<\partial_{i}X_{\varepsilon}(\text{y}),\partial_{j}X_{\varepsilon}(\text{y})>=
\tilde{g}_{ij}(\varepsilon y,\varepsilon z),
\label{def_tg}
\end{eqnarray}
where $<\cdotp,\cdotp>$ denotes the scalar product of $\R^{3}$ and $i,j=1,2$. The Laplacian on $\Sigma_{\varepsilon,z}$ is given by
\begin{eqnarray}
\Delta_{\Sigma_{\varepsilon,z}}=\frac{1}{\sqrt{\det{\tilde{g}_{\varepsilon}(y,z)}}}\partial_{j}
\big(\sqrt{\det{\tilde{g}_{\varepsilon}(y,z)}}\tilde{g}_{\varepsilon}^{ij}(y,z)\partial_{i}\big)=
\tilde{g}_{\varepsilon}^{ij}(y,z)\partial_{ij}+\tilde{b}_{\varepsilon}^{i}(y,z)\partial_{i},
\label{laplaSigmaz}
\end{eqnarray}
where
\begin{eqnarray}
\tilde{b}_{\varepsilon}^{i}(y,z):=\partial_{j}\tilde{g}_{\varepsilon}^{ij}(y,z)+\frac{1}{2}
\partial_{j}\big(\log\det\tilde{g}_{\varepsilon}(y,z)\big)\tilde{g}_{\varepsilon}^{ij}(y,z)
\label{def_ab}
\end{eqnarray}
and $\tilde{g}_{\varepsilon}^{ij}:=(\tilde{g}_{\varepsilon}^{-1})_{ij}$ are the elements of the inverse of the metric. 
These quantities are related to the ones of $\Sigma_{z}$ through the relations
\begin{eqnarray}\notag
\tilde{g}_{\varepsilon}^{ij}(y,z)=\tilde{g}^{ij}(\varepsilon y,\varepsilon z),\\\notag
\tilde{b}_{\varepsilon}^{i}(y,z)=\varepsilon\tilde{b}^{i}(\varepsilon y,\varepsilon z),
\end{eqnarray}
with $\tilde{g}^{ij}:=\tilde{g}_{1}^{ij}$ and $\tilde{b}^{i}:=\tilde{b}_{1}^{i}$. We define the second fundamental form at $y+z\nu(\varepsilon y)\in\Sigma_{\varepsilon,z}$ to be the linear application of the tangent space $T_{y+z\nu(\varepsilon y)}\Sigma_{\varepsilon,z}$ into itself that, in the standard basis $\{\partial_{i}X_{\varepsilon}(\text{y},z)\}_{i=1,2}$, is represented by the matrix
\begin{eqnarray}
\tilde{A}_{\varepsilon,ij}(y,z):=-<\partial_{i}\nu(\varepsilon y),\partial_{j}X_{\varepsilon}(\text{y},z)>.
\label{def_A}
\end{eqnarray}
We introduce the mean curvature $\tilde{H}_{\varepsilon}(y,z)$ of $\Sigma_{\varepsilon,z}$ at $y+z\nu(\varepsilon y)$ as follows
\begin{eqnarray}\notag
\tilde{H}_{\varepsilon}(y,z):=(\tilde{A}_{\varepsilon})^{i}_{i}(y,z)=\tilde{g}_{\varepsilon}^{ij}(y,z)\tilde{A}_{\varepsilon,ij}(y,z).
\end{eqnarray}
In other words
\begin{eqnarray}\notag
\tilde{H}_{\varepsilon}(y,z)=\tilde{k}_{\varepsilon,1}(y,z)+\tilde{k}_{\varepsilon,2}(y,z),
\end{eqnarray}
where $\tilde{k}_{\varepsilon,i}(y,z)$ are the \textit{principal curvatures} of $\Sigma_{\varepsilon,z}$, that is eigenvalues of the matrix $\tilde{g}_{\varepsilon}^{-1}(y,z)\tilde{A}_{\varepsilon}(y,z)$. Therefore we can see that the metric $\tilde{g}_{\varepsilon,ij}(y,z)$ and the matrix representing the second fundamental $\tilde{A}_{\varepsilon,ij}(y,z)$ form depend on the parametrization, while this is not the case for $\tilde{H}_{\varepsilon}(y,z)$. Setting, as above $\tilde{A}_{ij}:=\tilde{A}_{1,ij}$ and $\tilde{H}:=\tilde{H}_{1}$, we have $\tilde{A}_{\varepsilon,ij}(y,z)=\varepsilon\tilde{A}_{ij}(\varepsilon y,\varepsilon z)$ and $\tilde{H}_{\varepsilon}(y,z)=\varepsilon\tilde{H}(\varepsilon y,\varepsilon z)$.
\begin{lemma}
For a function $u:V_{\tau/\varepsilon}\to\R$ of class $C^{2}$, the Laplacian in Fermi coordinates is given by
\begin{eqnarray} 
\Delta u(y,z)=\Delta_{\Sigma_{\varepsilon,z}}u(y,z)-\varepsilon\tilde{H}(\varepsilon y,\varepsilon z)\partial_{z}u(y,z)+\partial_{zz}u(y,z).
\label{laplFermi}
\end{eqnarray}
For the notation, see Remark \ref{remFc}.
\end{lemma}
\begin{proof}
For any $y\in\Sigma_{\varepsilon}$ and $|z|<\tau/\varepsilon$, $\R^{3}$ splits into the direct sum of the tangent space to $\Sigma_{\varepsilon,z}$ and the one dimensional subspace generated by the unit normal $\nu(\varepsilon y)$, that is $\R^{3}=T_{y+z\nu(\varepsilon y)}\Sigma_{\varepsilon,z}+\R$. The vectors $\{\partial_{i}X_{\varepsilon}(y,z),\nu(\varepsilon y)\}_{i=1,2}$ constitute a basis of $\R^{3}=T_{y+z\nu(\varepsilon y)}\R^{3}$. The metric in this basis is given by
\begin{eqnarray}
G_{\varepsilon}(y,z)=\begin{bmatrix}
\tilde{g}_{\varepsilon}(y,z) & 0\\
0 & 1
\end{bmatrix}.
\end{eqnarray}
The inverse is
\begin{eqnarray}
G_{\varepsilon}^{-1}(y,z)=\begin{bmatrix}
\tilde{g}_{\varepsilon}^{-1}(y,z) & 0\\
0 & 1
\end{bmatrix}.
\end{eqnarray}
Here $1\leq I,J\leq 3$ and $1\leq i,j\leq 2$. The laplacian on $\R^{3}$ in the metric $G_{\varepsilon}$ is given by
\begin{eqnarray}\notag
\Delta u=\frac{1}{\sqrt{\det G_{\varepsilon}(y,z)}}\partial_{J}(\sqrt{\det G_{\varepsilon}(y,z)}G_{\varepsilon}^{IJ}(y,z)\partial_{I})=\\\notag
G^{IJ}_{\varepsilon}(y,z)\partial_{IJ}u(y,z)+\partial_{J}G_{\varepsilon}^{IJ}(y,z)\partial_{I}u(y,z)+\frac{1}{2}\partial_{J}(\log\det G_{\varepsilon}(y,z))G_{\varepsilon}^{IJ}(y,z)\partial_{I}u(y,z).
\end{eqnarray}
Thus
\begin{eqnarray}\notag
G_{\varepsilon}^{IJ}(y,z)\partial_{IJ}u(y,z)=\tilde{g}_{\varepsilon}^{ij}(y,z)\partial_{ij}u(y,z)+\partial_{zz}u(y,z)\\\notag
\partial_{J}G_{\varepsilon}^{IJ}(y,z)\partial_{I}u(y,z)=\partial_{j}\tilde{g}_{\varepsilon}^{ij}(y,z)\partial_{i}u(y,z)\\\notag
\frac{1}{2}\partial_{J}(\log\det G_{\varepsilon}(y,z))G_{\varepsilon}^{IJ}(y,z)\partial_{I}u(y,z)=\\\notag
\frac{1}{2}\partial_{j}(\log\det\tilde{g}_{\varepsilon}(y,z))\tilde{g}_{\varepsilon}^{ij}(y,z)\partial_{i}u(y,z)
+\frac{1}{2}\partial_{z}(\log\det\tilde{g}_{\varepsilon}(y,z))\partial_{z}u(y,z).
\end{eqnarray}
To conclude, we point out that
\begin{eqnarray}\notag
\frac{1}{2}\partial_{z}(\log\det\tilde{g}_{\varepsilon}(y,z))=-\tilde{H}_{\varepsilon}(y,z)=-\varepsilon\tilde{H}(\varepsilon y,\varepsilon z).
\end{eqnarray}
\end{proof}
Exploiting the Taylor expansion of $\tilde{H}$ of the mean curvature of a given hypersurface provided by Del Pino, Kowalczyk and Wei (see \cite{dPKW}), we have that
\begin{eqnarray}\
\tilde{H}(\varepsilon y,\varepsilon z)=\sum_{i=1}^{2}\frac{k_{i}(\varepsilon y)}{1-\varepsilon zk_{i}(\varepsilon y)}=\sum_{j\geq 1}(\varepsilon z)^{j-1}H_{j}(\varepsilon y), &\text{ }H_{j}(\varepsilon y):=\sum_{i=1}^{2}k^{j}_{i}(\varepsilon y)
\end{eqnarray}
Here $k_{i}(\varepsilon y):=\tilde{k}_{\varepsilon,i}(y,0)$ are the principal curvatures of the Clifford Torus $\Sigma$ at $\varepsilon y$. Therefore the Taylor expansions of the first and the second derivatives of $\tilde{H}$ are
\begin{eqnarray}
\begin{cases}
\tilde{H}_{z}(\varepsilon y,\varepsilon z)=\sum_{j\geq 1}j(\varepsilon z)^{j-1}H_{j+1}(\varepsilon y),\\
\tilde{H}_{zz}(\varepsilon y,\varepsilon z)=\sum_{j\geq 1}j(j+1)(\varepsilon z)^{j-1}H_{j+2}(\varepsilon y).
\end{cases}
\end{eqnarray}
In the sequel, we will set $H(\varepsilon y):=H_{1}(\varepsilon y)$, $|A(\varepsilon y)|^{2}:=H_{2}(\varepsilon y)$ and $\tr A^{3}(y):=H_{3}(\varepsilon y)$. 

Now we need the Taylor expansion in $\varepsilon z$ of $\Delta_{\Sigma_{\varepsilon,z}}$. For our purposes, it is enough to know the terms of order zero and one, while we also need the term of order two in the expansion of $\tilde{H}$. For this reason, we prefer not to expand the full Laplacian on $\R^{3}$. In fact, an expansion up to order one would not be enough, because we cannot neglect the terms involving $\tr A^{3}$, while an expansion up to order two would be a useless effort, in fact it would involve the terms of order two of $\Delta_{\Sigma_{\varepsilon,z}}$, that will always simplify in our forthcoming calculations. Before stating Next Lemma, we recall that
\begin{eqnarray}
\Delta_{\Sigma_{\varepsilon}}=\frac{1}{\sqrt{\det{g_{\varepsilon}(y)}}}\partial_{j}
\big(\sqrt{\det g_{\varepsilon}(y)}g_{\varepsilon}^{ij}(y)\partial_{i}\big)=
g_{\varepsilon}^{ij}(y)\partial_{ij}+b_{\varepsilon}^{i}(y)\partial_{i},
\end{eqnarray}
where
\begin{eqnarray}
g_{\varepsilon}^{ij}(y):=\tilde{g}_{\varepsilon}^{ij}(y,0)=\tilde{g}^{ij}(\varepsilon y,0)=g^{ij}(\varepsilon y)\\\notag
b_{\varepsilon}^{i}(y):=\tilde{b}_{\varepsilon}^{i}(y,0)=\varepsilon\tilde{b}^{i}(\varepsilon y,0)=\varepsilon b^{i}(\varepsilon y).
\end{eqnarray}
It is possible to find similar computations in \cite{MSY}, where Mahmoudi, Sànchez and Yao treat the more general case of a $k$ dimensional submanifold in an $N$ dimensional manifold.
\begin{lemma}
For a function $u:V_{\tau/\varepsilon}\to\R$ of class $C^{2}$, for any $y\in\Sigma_{\varepsilon}$, for any $|z|\leq\tau/\varepsilon$,
\begin{eqnarray}\notag
\Delta_{\Sigma_{\varepsilon,z}}u=\Delta_{\Sigma_{\varepsilon}}u+\varepsilon z(a_{1}^{ij}(\varepsilon y)\partial_{ij}+\varepsilon b_{1}^{i}(\varepsilon y)\partial_{i})\\\notag
+(\varepsilon z)^{2}(a_{2}^{ij}(\varepsilon y)\partial_{ij}+\varepsilon b_{2}^{i}(\varepsilon y)\partial_{i})
+\overline{a}^{ij}(\varepsilon y,\varepsilon z)\partial_{ij}+\varepsilon\overline{b}_{i}(\varepsilon y,\varepsilon z)\partial_{i},
\end{eqnarray}
where
\begin{eqnarray}\notag
a_{1}^{ij}:=2A^{ij}, b_{1}^{i}:=2\partial_{j}A^{ij}+2\Gamma^{k}_{kj}A^{ij}-g^{ij}H_{j},\\\notag
a_{2}^{ij}:=\frac{1}{2}\partial_{zz}\tilde{g}^{ij}(\varepsilon y,0), b_{2}^{i}:=\frac{1}{2}\partial_{zz}\tilde{b}^{i}(\varepsilon y,0),
\end{eqnarray}
everything evaluated at $\varepsilon y$, and the remainders satisfy $|\overline{a}^{ij}(\varepsilon y,\varepsilon z)|,|\overline{b}^{i}(\varepsilon y,\varepsilon z)|\leq c\varepsilon^{3}|z|^{3}$, for some constant $c>0$ depending on $\Sigma$.
\end{lemma}
Let $\phi,\psi:\Sigma\to\R$ be $C^{2}$ functions. Let us set $\phi_{i}:=\partial_{i}\phi$. We recall that, by the properties of the covariant derivative,
\begin{eqnarray}\notag
\nabla_{k}A^{ij}=\partial_{k}A^{ij}+\Gamma^{i}_{kl}A^{lj}+\Gamma^{j}_{kl}A^{li},\\\notag
\nabla^{2}_{ij}\phi=\phi_{ij}-\Gamma^{k}_{ij}\phi_{k},
\end{eqnarray}
where everything is evaluated at $\varepsilon y$. Moreover, by Codazzi's equation, $\nabla_{j}A^{ij}=g^{ik}\nabla_{k}A^{j}_{j}$, so in particular, 
\begin{eqnarray}\notag
a_{1}^{ij}\phi_{i}\psi_{j}=2(A\nabla\phi,\nabla\psi)\\
a_{1}^{ij}\psi_{ij}+b_{1}^{i}\psi_{i}=2A^{ij}\psi_{ij}-2\Gamma^{k}_{ji}A^{ij}\psi_{k}+2\nabla_{j}A^{ij}\psi_{i}\label{geom_rel}\\\notag
-(\nabla_{\Sigma}H,\nabla_{\Sigma}\psi)=2<A,\nabla^{2}\psi>+(\nabla_{\Sigma}\psi,\nabla_{\Sigma}H),
\end{eqnarray}
where we have set
\begin{eqnarray}\notag
<A,\nabla^{2}\psi>:=A^{ij}\nabla^{2}_{ij}\psi=A^{ij}\psi_{ij}+\Gamma^{k}_{ij}\psi_{k}.
\end{eqnarray}
\begin{proof}
By (\ref{param}) and (\ref{def_tg}), we can see that
\begin{eqnarray}\notag
\tilde{g}_{\varepsilon,ij}(y,z)=g_{ij}+\varepsilon z(<\partial_{i}Y,\partial_{j}\nu>+
<\partial_{j}Y,\partial_{i}\nu>)+(\varepsilon z)^{2}<\partial_{i}\nu,\partial_{j}\nu>.
\end{eqnarray}
In the proof, it is understood that the geometric quantities of $\Sigma$ are evaluated at $\varepsilon y$. In view of (\ref{def_A}) with $z=0$, we have 
\begin{eqnarray}\notag
\partial_{i}\nu=-A^{k}_{i}\partial_{k}Y
\end{eqnarray}
therefore
\begin{eqnarray}
\tilde{g}_{\varepsilon,ij}(y,z)=g_{ij}-\varepsilon z(g_{ik}A^{k}_{j}+g_{jk}A^{k}_{i})+(\varepsilon z)^{2}<\partial_{i}\nu,\partial_{j}\nu>=\\\notag
g_{ij}-2\varepsilon zA_{ij}+(\varepsilon z)^{2}<\partial_{i}\nu,\partial_{j}\nu>.
\end{eqnarray}
In order to expand the Laplacian, we need the expansion of the inverse of the metric. It is useful to write it as 
\begin{eqnarray}\notag
\tilde{g}_{\varepsilon}=L+M,
\end{eqnarray}
with $L_{ij}=g_{ij}$ and $M=-2\varepsilon zA_{ij}+(\varepsilon z)^{2}<\partial_{i}\nu,\partial_{j}\nu>$. Equivalently, $\tilde{g}_{\varepsilon}=L(I+L^{-1}M)$, hence
\begin{eqnarray}\notag
\tilde{g}_{\varepsilon}^{-1}=(I+L^{-1}M)^{-1}L^{-1}=(I-L^{-1}M+O((\varepsilon z)^{2}))L^{-1}=L^{-1}-L^{-1}ML^{-1}+O((\varepsilon z)^{2}),
\end{eqnarray}
thus 
\begin{eqnarray}\notag
\tilde{g}_{\varepsilon}^{ij}(y,z)=g^{ij}+2\varepsilon zA^{ij}+O((\varepsilon z)^{2}).
\end{eqnarray}
where $A^{ij}=g^{ik}g^{jl}A_{kl}$. Moreover
\begin{eqnarray}\notag
\log\det\tilde{g}_{\varepsilon}(y,z)=\log\det g_{\varepsilon}(y)+\tr(L^{-1}M)+O((\varepsilon z)^{2})=\log\det g_{\varepsilon}-2\varepsilon zH+O((\varepsilon z)^{2}),
\end{eqnarray}
so, since $\frac{1}{2}\partial_{j}(\log\det g)A^{ij}=\Gamma^{k}_{kj}A^{ij}$,
\begin{eqnarray}\notag
\Delta_{\Sigma_{\varepsilon,z}}=(g^{ij}+2\varepsilon zA^{ij})\partial_{ij}+
\varepsilon(\partial_{j}g^{ij}+2\varepsilon z\partial_{j}A^{ij})\partial_{i}\\\notag
+\varepsilon\big(\frac{1}{2}\partial_{j}(\log\det g)-\varepsilon zH_{j}\big)(g^{ij}+2\varepsilon zA^{ij})\partial_{i}+O((\varepsilon z)^{2})=\\\notag
\Delta_{\Sigma_{\varepsilon}}+\varepsilon z\bigg\{2A^{ij}\partial_{ij}+\varepsilon(2\partial_{j}A^{ij}+2\Gamma^{k}_{kj}A^{ij}-g^{ij}H_{j})\partial_{i}\bigg\}+O((\varepsilon z)^{2}).
\end{eqnarray}
\end{proof}
As a consequence, we have the following expansion of the Laplacian 
\begin{eqnarray}
\Delta=\partial_{zz}-\varepsilon\tilde{H}(\varepsilon y,\varepsilon z)\partial_{z}+\Delta_{\Sigma_{\varepsilon}}
+\varepsilon z(a_{1}^{ij}(\varepsilon y)\partial_{ij}+\varepsilon b_{1}^{i}(\varepsilon y)\partial_{i})\label{exp_lapl}\\\notag
+(\varepsilon z)^{2}(a_{2}^{ij}(\varepsilon y)\partial_{ij}+\varepsilon b_{2}^{i}(\varepsilon y)\partial_{i})+\overline{a}^{ij}(\varepsilon y,\varepsilon z)\partial_{ij}+\varepsilon\overline{b}^{i}(y,z)\partial_{i}.
\end{eqnarray}
Although (\ref{exp_lapl}) looks nice, we prefer to look for the expression of the Laplacian in a slightly different system of coordinates. We fix a $C^{2}$ function $\phi:\Sigma\to\R$ whose $L^{\infty}(\Sigma)$ is less than $1/4$ and we introduce a new change of variables, that is we put 
\begin{eqnarray}
t=z-\phi(\varepsilon y).
\label{deft}
\end{eqnarray}
The expression of the Laplacian will be more complicated than (\ref{exp_lapl}), but more appropriate for our purposes. The reason is that we know the kernel of the operator $-(\Delta_{\Sigma_{\varepsilon}}+\partial_{tt})+W^{''}(v_{\star}(t))$, that is the one dimensional space generated by $v_{\star}^{'}(t)$, while we do not know exactly the kernel (if any) of $-(\Delta_{\Sigma_{\varepsilon}}+\partial_{zz})+W^{''}(v_{\star}(z-\phi(\varepsilon y)))$. 

Given a function
\begin{eqnarray}\notag
f:\Sigma_{\varepsilon}\times\R\to\R
\end{eqnarray}
of class $C^{2}$, it is possible to define
\begin{eqnarray}\notag
\text{f}:\Sigma_{\varepsilon}\times\R\to\R
\end{eqnarray}
by setting $\text{f}(y,t):=f(y,z-\phi(\varepsilon y))$. A computation shows that
\begin{eqnarray}\notag
\text{f}_{t}(y,t)=f_{z}(y,z-\phi)\\\notag
\text{f}_{i}(y,t)=f_{i}(y,z-\phi)-\varepsilon\phi_{i}f_{z}(y,z-\phi)\\\notag
\text{f}_{ij}(y,t)=f_{ij}(y,z-\phi)-\varepsilon\phi_{i}f_{zj}(y,z-\phi)-\varepsilon\phi_{i}f_{zj}(y,z-\phi)\\\notag
+\varepsilon^{2}\phi_{ij}f_{z}(y,z-\phi)+\varepsilon\phi_{i}\phi_{j}f_{zz}(y,z-\phi),
\end{eqnarray}
where $\phi$ and its derivatives are evaluated at $\varepsilon y$, thus, in these coordinates, the expression of the Laplacian of a function $u$ defined in $V_{\tau/\varepsilon}$ of class $C^{2}$ is given by
\begin{eqnarray}
\Delta=\partial_{tt}+g^{ij}\partial_{ij}+\varepsilon b^{i}\partial_{i}+\text{D}=\partial_{tt}+\Delta_{\Sigma_{\varepsilon}}+\text{D},
\label{lapl_yt}
\end{eqnarray}
where the operator D is given by
\begin{eqnarray}
\text{D}:=-\varepsilon\hat{H}(\varepsilon y,\varepsilon(t+\phi))\partial_{t}-\varepsilon^{2}\Delta_{\Sigma}\phi\partial_{t}-2\varepsilon g^{ij}\phi_{i}\partial_{tj}+\varepsilon^{2}|\nabla_{\Sigma}\phi|^{2}\partial_{tt}\label{exp_laplt}\\\notag
+\varepsilon(t+\phi)\big\{a_{1}^{ij}\partial_{ij}+\varepsilon b_{1}^{i}\partial_{i}-\varepsilon^{2}(a_{1}^{ij}\phi_{ij}+ b_{1}^{i}\phi_{i})\partial_{t}-2\varepsilon a_{1}^{ij}\phi_{i}\partial_{tj}+\varepsilon^{2}a_{1}^{ij}\phi_{i}\phi_{j}\partial_{tt}\big\}\\\notag
+\varepsilon^{2}(t+\phi)^{2}\big\{a_{2}^{ij}\partial_{ij}+\varepsilon b_{2}^{i}\partial_{i}-\varepsilon^{2}(a_{2}^{ij}\phi_{ij}+ b_{2}^{i}\phi_{i})\partial_{t}-2\varepsilon a_{2}^{ij}\phi_{i}\partial_{tj}+\varepsilon^{2}a_{2}^{ij}\phi_{i}\phi_{j}\partial_{tt}\big\}\\\notag
+\hat{a}^{ij}\partial_{ij}+\varepsilon\hat{b}^{i}\partial_{i}-\varepsilon^{2}(\hat{a}^{ij}\phi_{ij}+ \hat{b}^{i}\phi_{i})\partial_{t}-2\varepsilon\hat{a}^{ij}\phi_{i}\partial_{tj}+\varepsilon^{2}\hat{a}^{ij}\phi_{i}\phi_{j}\partial_{tt}.
\end{eqnarray}
Here we have set $\hat{H}(\varepsilon y,\varepsilon(t+\phi)):=\tilde{H}(\varepsilon y,\varepsilon z)$, $\hat{a}^{ij}(\varepsilon y,\varepsilon(t+\phi))=\overline{a}^{ij}(\varepsilon y,\varepsilon z)$, $\hat{b}^{i}(\varepsilon y,\varepsilon(t+\phi))=\overline{b}^{i}(\varepsilon y,\varepsilon z)$ and all the geometric quantities of $\Sigma$ are evaluated at $\varepsilon y$.

\section{Functional setting}
\subsection{Functions on $\Sigma_{\varepsilon}$}
As first we define, for $0<\alpha<1$, the space $C^{k,\alpha}(\Sigma)$ as the set of functions $\phi:\Sigma\to\R$ that are $k$ times differentiable and whose $k-$th partial derivatives are H\"{o}lder continuous with exponent $\alpha$. We endow these spaces with the norms
\begin{eqnarray}\notag
|\phi|_{C^{k,\alpha}(\Sigma)}:=\sum_{j=0}^{k}||\nabla^{j}\phi||_{\infty}+\varepsilon^{\alpha}\sup_{p\neq q}\frac{|\nabla^{k}\phi(p)-\nabla^{k}\phi(q)|}{d(p,q)^{\alpha}}.
\end{eqnarray}
We note that these norms depend on $\varepsilon$, since this is the right scaling in order to obtain our estimates. 
Moreover, in order to treat $\tilde{L}_{0}$, we define the spaces of functions that respect the symmetries of the Torus, that is the symmetry with respect to the $x_{1}x_{2}$-plane and with respect to any rotation that keeps the $x_{3}$-axis fixed. To be precise, we set $T(x_{1},x_{2},x_{3}):=(x_{1},x_{2},-x_{3})$ and 
\begin{eqnarray}\notag
SO_{x_{3}}(3):=\{R\in SO(3):Re_{3}=e_{3}\},
\end{eqnarray}
where $e_{3}=(0,0,1)$, and we define
\begin{eqnarray}\notag
C^{k,\alpha}(\Sigma)_{s}:=\{\phi\in C^{k,\alpha}(\Sigma):\phi(\zeta)=\phi(T\zeta)\text{ for any $\zeta\in\Sigma$, }\\\notag
\phi(\zeta)=\phi(R\zeta)\text{ for any $R\in SO_{x_{3}}(3)$}\}.
\end{eqnarray}
By the symmetries of the Laplacian, the gradient and the geometric quantities of $\Sigma$, one can show that $\tilde{L}_{0}$ preserves the symmetries of functions $\phi\in C^{4,\alpha}(\Sigma)_{s}$, that is it maps $C^{4,\alpha}(\Sigma)_{s}$ into $C^{0,\alpha}(\Sigma)_{s}$. 

We note that $SO_{x_{3}}(2)\simeq SO(2)$, in the sense that any matrix $R\in SO_{x_{3}}(3)$ has the form
\begin{eqnarray}\notag
R=
\begin{bmatrix}
\tilde{R} & 0\\
0 & 1
\end{bmatrix},
\end{eqnarray} 
for some rotation of the $x_{1}x_{2}$-plane $\tilde{R}\in SO(2)$.

Moreover, Remark \ref{remtilde_L_0} can be rephrased by saying that the operator
\begin{eqnarray}\notag
\mathcal{L}:C^{4,\alpha}(\Sigma)_{s}\times\R\to C^{0,\alpha}(\Sigma)_{s}\times\R
\end{eqnarray}
defined by 
\begin{eqnarray}\notag
\mathcal{L}(\phi,\lambda):=\bigg(\tilde{L}_{0}\phi+\lambda,\int_{\Sigma}\phi(\zeta)d\sigma(\zeta)\bigg)
\end{eqnarray}
is injective. In fact, if $\mathcal{L}(\phi,\lambda)=0$, multiplying by $\phi$ and integrating over $\Sigma$ we get
\begin{eqnarray}\notag
\int_{\Sigma}\tilde{L}_{0}\phi(\zeta)\phi(\zeta)d\sigma(\zeta)=-\lambda\int_{\Sigma}\phi(\zeta)d\sigma(\zeta)=0,
\end{eqnarray}
and hence, since $\tilde{L}_{0}$ is positive definite on 
\begin{eqnarray}
X:=\bigg\{\phi\in C^{4,\alpha}(\Sigma)_{s}:\int_{\Sigma}\phi(\zeta)d\sigma(\zeta)=0\bigg\},
\end{eqnarray}
we conclude that $\phi=0$, so $\lambda=0$.

Being $\mathcal{L}$ also elliptic and self-adjoint with respect to the scalar product
\begin{eqnarray}\notag
<(\phi,\lambda),(\psi,\mu)>:=\int_{\Sigma}\phi(\zeta)\psi(\zeta)d\sigma(\zeta)+\lambda\mu,
\end{eqnarray}
it is actually an isomorphism with bounded inverse.

In the sequel ,we will often use the notation
\begin{eqnarray}\notag
B_{k}(1/4):=\{\phi\in C^{k,\alpha}(\Sigma)_{s}:|\phi|_{C^{k,\alpha}(\Sigma)}\leq 1/4\}.
\end{eqnarray}

\subsection{Exponentially decaying functions on $\R^{3}$}

For any $\delta>0$ and for any $x\in\R^{N}$, we define
\begin{eqnarray}\notag
\varphi_{\delta}(x)=:\zeta(|x|)+(1-\zeta(|x|))e^{\delta|x|},
\end{eqnarray}
where $\chi:\R\to\R$ is a $C^{\infty}$ cutoff function such that
\begin{eqnarray}\notag
\zeta(t)=
\begin{cases}
1 &\text{for $t<1$}\\\notag
0 &\text{for $t>2$.}
\end{cases}
\end{eqnarray}
Moreover, we introduce the weighted spaces
\begin{eqnarray}\notag
C^{k,\alpha}_{\delta}(\R^{3}):=\{u\in C^{k,\alpha}(\R^{3}):||\tilde{u}_{\delta}||_{C^{k,\alpha}(\R^{3})}<\infty\},
\end{eqnarray} 
where $\tilde{u}_{\delta}:=u\varphi_{\delta}$ and $C^{k,\alpha}(\R^{3})$ is the space of $C^{k}(\R^{3})$ functions whose forth derivatives are H\"{o}lder continuous with exponent $\alpha$. We point out that functions belonging $C^{k,\alpha}_{\delta}(\R^{3})$ decay exponentially with rate $\delta$, and the same is true for their derivatives.

This spaces are endowed with the norm $||u||_{C^{k,\alpha}_{\delta}(\R^{3})}=||\tilde{u}_{\delta}||_{C^{k,\alpha}(\R^{3})}$, where
\begin{eqnarray}\notag
||u||_{C^{k,\alpha}(\R^{3})}:=\sum_{j=0}^{k}||\nabla^{j}u||_{\infty}+
[\nabla^{k}u]_{\alpha}.
\end{eqnarray}
In order to construct solutions to (\ref{Cahn-Hilliard}) that respect the symmetry of the Torus, we need to introduce the spaces of functions 
fulfilling these symmetries, that is
\begin{eqnarray}\notag
C^{k,\alpha}_{\delta,s}(\R^{3}):=\{u\in C^{k,\alpha}_{\delta}(\R^{3}):u(Tx)=u(x)&\text{, }u(Rx)=u(x)&\text{ for any }R\in SO_{x_{3}}(3)\}.
\end{eqnarray}
\begin{remark}
We note that, for instance, if $u\in C^{2,\alpha}_{\delta,s}(\R^{3})$, then $\Delta u\in C^{0,\alpha}_{\delta,s}(\R^{3})$. In fact, by definition, any $u\in C^{2,\alpha}_{\delta,s}(\R^{3})$ satisfies $u(x)=u_{T}(x)$, where $u_{T}(x):=u(Tx)$. Taking the Laplacian, we can see that $\Delta u(x)=\Delta u_{T}(x)=\Delta u(Tx)$, and similarly, if $R\in SO_{x_{3}}(3)$ and we set $u_{R}(x)=u(Rx)$, then $\Delta u(x)=\Delta u_{R}(x)=\Delta u(Rx)$.
\label{rem_symm}
\end{remark}

\subsection{Functions on $\Sigma_{\varepsilon}\times\R$}
First we will show existence and uniqueness of the heteroclinic solution to the ODE $-v_{\star}^{''}+W^{'}(v_{\star})=0$. The result is known, but since the proof is quite short, we report it for completeness.
\begin{lemma}
Let $W$ be an even double well potential satisfying (\ref{double-well}). Then there exists a unique solution $v_{\star}$ to the problem
\begin{eqnarray}
\begin{cases}
-v_{\star}^{''}+W^{'}(v_{\star})=0\\
v_{\star}(0)=0\\
v_{\star}\to\pm 1 &\text{as $t\to\pm\infty$.}
\end{cases}
\label{ODE}
\end{eqnarray}
and this solution is odd.
\end{lemma}
It is known that, if $W(t)=\frac{1}{4}(1-t^{2})^{2}$ is the classical double-well potential, then $v_{\star}(t)=\tanh(t/\sqrt{2})$.
\begin{proof}
Let $v_{\star}$ be the unique solution to the Cauchy Problem
\begin{eqnarray}\notag
\begin{cases}
-v_{\star}^{''}+W^{'}(v_{\star})=0\\\notag
v_{\star}(0)=0\\\notag
v_{\star}^{'}(0)=\sqrt{2W(0)}.
\end{cases}
\end{eqnarray}
Let $(a,b)$ be its maximal interval of definition, with $a<0<b$. Since the function $w(t)=-v_{\star}(-t)$ is still a solution to the same Cauchy Problem, $v_{\star}$ is an odd function, so it is enough to study $v_{\star}$ in the positive half line and $a=-b$. Multiplying the ODE by $v_{\star}^{'}$ and integrating we have
\begin{eqnarray}
\frac{1}{2}(v_{\star}^{'})^{2}=W(v_{\star})+c.
\label{relint}
\end{eqnarray}
Evaluating at $t=0$, it is possible to see that $c=0$. As a consequence, $v_{\star}^{'}>0$ in $(0,b)$. In fact, if we assume by contradiction that there exists a first $t_{0}$ such that $v_{\star}^{'}(t_{0})=0$, then $W(v_{\star}(t_{0}))=0$, so in particular $v_{\star}(t_{0})=1$, but, by the uniqueness Cauchy Theorem, this implies that $v_{\star}\equiv 1$ in a neighbourhood of $t_{0}$, a contradiction. As a consequence, it is possible to define
\begin{eqnarray}\notag
l:=\lim_{t\to b}v_{\star}(t).
\end{eqnarray}
By monotonicity, we know that $l>0$. Now we want to rule out the case $l=\infty$. In fact, it this were true, we would have $v_{\star}^{''}<0$ near $0$ and $v_{\star}^{''}>0$ near $b$, so there should exist $t_{1}>0$ such that $v_{\star}^{''}(t_{1})=0$. Therefore, using the equation and (\ref{relint}), we can see that $v_{\star}(t_{1})=1$ and $v_{\star}^{'}(t_{1})=0$, which is not possible.

Since $l<\infty$, we have $b=\infty$. Now, always by (\ref{relint}), we get that $v_{\star}^{'}\to\sqrt{2W(l)}$ as $t\to\infty$. Since $u$ is bounded, $W(l)=0$, hence $l=1$.

Uniqueness follows from the Cauchy Theorem.
\end{proof}
It is known that $v_{\star}$ converges exponentially to $\pm 1$ as $t\to\pm\infty$ at a rate which is given by $\sqrt{W^{''}(1)}=\sqrt{W^{''}(-1)}$, since $W$ is even. More precisely, for any $k\in\N$, there exists a constant $c_{k}$ such that
\begin{eqnarray}
|\partial^{k}_{t}(v_{\star}-1)|\leq c_{k}e^{-t\sqrt{W^{''}(1)}} &\text{for any $t\geq 0$}
\label{expdecdx}
\end{eqnarray} 
and
\begin{eqnarray}
|\partial^{k}_{t}(v_{\star}+1)|\leq c_{k}e^{t\sqrt{W^{''}(1)}} &\text{for any $t\leq 0$.}
\label{expdecsx}
\end{eqnarray}
For instance, in the classical case $W(t)=\frac{1}{4}(1-t^{2})^{2}$, we have $\sqrt{W^{''}(\pm 1)}=\sqrt{2}$.\\

For $0<\delta<\sqrt{W^{''}(1)}$, we define the function
\begin{eqnarray}\notag
\psi_{\delta}(t)=(1+e^{t})^{\delta}(1+e^{-t})^{\delta}.
\end{eqnarray}
For $0<\varepsilon\leq 1$ and $0<\alpha<1$, we define the space $C^{k,\alpha}_{\delta}(\Sigma_{\varepsilon}\times\R)$ as the set of functions $U:\Sigma_{\varepsilon}\times\R\to\R$ that are $k$ times differentiable and whose $k-$th partial derivatives are H\"{o}lder continuous with exponent $\alpha$. This space is endowed with the norm
\begin{eqnarray}\notag
||U||_{C^{k,\alpha}_{\delta}(\Sigma_{\varepsilon}\times\R)}=||U\psi_{\delta}||_{C^{k,\alpha}(\Sigma_{\varepsilon}\times\R)},
\end{eqnarray}
where
\begin{eqnarray}\notag
||U||_{C^{k,\alpha}(\Sigma_{\varepsilon}\times\R)}=\sum_{j=0}^{k}||\nabla^{j} U||_{L^{\infty}(\Sigma_{\varepsilon}\times\R)}+\sup_{x\neq y}\frac{|\nabla^{k}u(x)-\nabla^{k}u(y)|}{|x-y|^{\alpha}}.
\end{eqnarray} 
Given the heteroclinic solution $v_{\star}$, we can define the spaces
\begin{eqnarray}\notag
\mathcal{E}^{k,\alpha}_{\delta}(\Sigma_{\varepsilon}\times\R):=\bigg\{U\in C^{k,\alpha}_{\delta}(\Sigma_{\varepsilon}\times\R):\int_{-\infty}^{\infty}U(y,t)v_{\star}^{'}(t)dt=0 \text{ for any }y\in\Sigma_{\varepsilon}\bigg\}
\end{eqnarray}
of functions that orthogonal, for any $y\in\Sigma_{\varepsilon}$, to $v_{\star}^{'}$.

Moreover, as above, we will be interested in the spaces of functions that respect the symmetries of the Torus, thus we define
\begin{eqnarray}\notag
C^{k,\alpha}_{\delta,s}(\Sigma_{\varepsilon}\times\R):=\{U\in C^{k,\alpha}_{\delta}(\Sigma_{\varepsilon}\times\R):U_{T}=U,&\text{ }U_{R}=U&\text{for any }R\in SO_{x_{3}}(3)\},
\end{eqnarray}
where we have set $U_{T}(y,z):=U(Ty,z)$ and $U_{R}(y,z):=U(Ry,z)$. Furthermore, we set $\mathcal{E}^{k,\alpha}_{\delta,s}(\Sigma_{\varepsilon}\times\R):=\mathcal{E}^{k,\alpha}_{\delta}(\Sigma_{\varepsilon}\times\R)\cap
C^{k,\alpha}_{\delta,s}(\Sigma_{\varepsilon}\times\R)$. These spaces consist of functions that are both symmetric and orthogonal to $v_{\star}^{'}$.

In the sequel, we will often mention the operator
\begin{eqnarray}\notag
\mathcal{L}_{\varepsilon}U:=-(\Delta_{\Sigma_{\varepsilon}}+\partial_{tt})U(y,t)+W^{''}(v_{\star}(t))U(y,t),
\end{eqnarray}
defined for any $U\in\ C^{4,\alpha}_{\delta}(\Sigma_{\varepsilon}\times\R)$.

\section{Idea of the proof: Lyapunov-Schmidt reduction}
By a rescaling argument, it is enough to construct solutions to 
\begin{eqnarray}\notag
-\Delta(-\Delta u+W{'}(u))+W^{''}(u)(-\Delta u+W{'}(u))=0,
\end{eqnarray}
whose nodal set is close to $\Sigma_{\varepsilon}$, since we can obtain the required solutions to (\ref{Cahn-Hilliard}) by setting $\tilde{u}(x):=u(x/\varepsilon)$. Thus we set
\begin{eqnarray}
F(u)=-\Delta(-\Delta u+W^{'}(u))+W^{''}(u)
(-\Delta u+W^{'}(u)).
\label{defF}
\end{eqnarray}
A computation shows that
\begin{eqnarray}
F^{'}(u)v=-\Delta(-\Delta v+W^{''}(u)v)+W^{''}(u)(-\Delta v+W^{''}(u)v)\\\notag
+W^{'''}(u)(-\Delta u+W^{'}(u))v
\label{defF'}
\end{eqnarray}
and
\begin{eqnarray}
F^{''}(u)[v,w]=-\Delta(W^{'''}(u)vw)+(W^{'''}(u)W^{''}(u)+W^{(4)}(u)(-\Delta u+W^{'}(u)))vw+\\\notag
W^{'''}(u)[w(-\Delta v+W^{''}(u)v)+v(-\Delta w+W^{''}(u)w)].
\label{defF''}
\end{eqnarray}
In order to produce the required solutions we fix $\varepsilon>0$ small and a small function $\phi\in C^{4,\alpha}(\Sigma)_{s}$, in the sense that $|\phi|_{C^{4,\alpha}(\Sigma)}<1/4$, and we define the approximate solution $v_{\varepsilon,\phi}$ in such a way that its nodal is exactly
\begin{eqnarray}\notag
\Sigma_{\varepsilon,\phi}:=\{y+\phi(\varepsilon y)\nu(\varepsilon y):y\in\Sigma_{\varepsilon}\},
\end{eqnarray}
and $v_{\varepsilon,\phi}\equiv\pm 1$ outside a sufficiently small tubular neighbourhood of $\Sigma_{\varepsilon,\phi}$, that is a neighbourhood of width $\tau/2\varepsilon+6$. More precisely, we set
\begin{eqnarray}\notag
\mathbb{H}(x):=
\begin{cases}
1 &\text{ if $f_{\varepsilon}(x)>0$}\\\notag
0 &\text{ if $f_{\varepsilon}(x)=0$}\\\notag
-1 &\text{ if $f_{\varepsilon}(x)<0$}
\end{cases}
\end{eqnarray}
and, for any $\varepsilon>0$ and for any integer $m>0$,
\begin{eqnarray}\notag
\chi_{m}(x):=
\begin{cases}
\zeta(|t|-\frac{\tau}{2\varepsilon}-m)&\text{ if }x=Z_{\varepsilon}(y,t+\phi(\varepsilon y))\in V_{\tau/\varepsilon},\\\notag
0&\text{ if }x\in\R^{3}\backslash V_{\tau/\varepsilon},
\end{cases}
\end{eqnarray}
and we look for an approximate solution of the form
\begin{eqnarray}
v_{\varepsilon,\phi}(x)=\chi_{5}(x)\tilde{v}_{\varepsilon,\phi}(y,t)+(1-\chi_{5}(x))\mathbb{H}(x),
\label{defv}
\end{eqnarray}
where $t$ is defined in (\ref{deft}), and $v_{\varepsilon,\phi}$ is understood to coincide with $\mathbb{H}$ outside the support of $\chi$. Moreover $v_{\varepsilon,\phi}$ will vanish exactly on $\Sigma_{\varepsilon,\phi}$ and it will respect the symmetries of the Torus. We stress that these cutoff functions actually depend on $\phi$, but we prefer not put the subscript $\phi$ to simplify the notation. 
However, we will see that the error $F(v_{\varepsilon,\phi})$ is small, but not zero, therefore we have to add a correction $w=w_{\varepsilon,\phi}$ depending on $\varepsilon$ and $\phi$ in order to obtain a real solution, that is $F(v_{\varepsilon,\phi}+w)=0$. Rephrasing our problem in this way, the unknowns are $\phi$ and $w$, for any $\varepsilon>0$ small but fixed. Expanding $F$ in Taylor series, our equation becomes
\begin{eqnarray}
F(v_{\varepsilon,\phi})+F^{'}(v_{\varepsilon,\phi})w+Q_{\varepsilon,\phi}(w)=0,
\label{taylor}
\end{eqnarray}
where
\begin{eqnarray}
Q_{\varepsilon,\phi}(w)=\int_{0}^{1}dt\int_{0}^{t}F^{''}(v_{\varepsilon,\phi}+sw)[w,w]ds,
\label{quadratic_w}
\end{eqnarray}
However, we are not able to solve it directly, because of the lack of coercivity of $F^{'}(v_{\varepsilon,\phi})$.

\subsection{The auxiliary equation: a gluing procedure}

We look for a solution of the form
\begin{eqnarray}
w(x)=\chi_{2}(x)U(y,t)+V(x), 
\end{eqnarray}
where $V$ is defined in the whole $\R^{3}$, $U$ is defined in the entire $\Sigma_{\varepsilon}\times\R$. 
Since we want our solutions $u_{\varepsilon}$ to respect the symmetries of the Torus, we look for solutions $U$ and $V$ such that
\begin{eqnarray}\notag
U(y,t)=U(Ty,t), &\text{ }U(y,t)=U(Ry,t), &\text{ for any $R\in SO_{x_{3}}(3)$ and $(y,t)\in\Sigma_{\varepsilon}\times\R$}\\\notag
V(x)=V(Tx), &\text{ }V(x)=V(Rx), &\text{ for any $R\in SO_{x_{3}}(3)$ and $x\in\R^{3}$.}
\end{eqnarray}
Now we observe that the potential
\begin{eqnarray}
\Gamma_{\varepsilon,\phi}(x):=(1-\chi_{1}(x))W^{''}(v_{\varepsilon,\phi})+\chi_{1}(x)W^{''}(1)
\end{eqnarray}
is positive and bounded away from $0$ in the whole $\R^{3}$, that is, for any $0<\gamma<\sqrt{W^{''}(1)}$, $0<\gamma^{2}<\Gamma_{\varepsilon,\phi}(x)<W^{''}(1)+\tau_{0}$ provided $\varepsilon$ is small enough, the estimate is uniform in $\phi$. Moreover, using that $\chi_{2}\chi_{1}=\chi_{1}$, we compute 
\begin{eqnarray}
0=\chi_{2}\bigg\{F(\tilde{v}_{\varepsilon,\phi})+F^{'}(\tilde{v}_{\varepsilon,\phi})+\chi_{1}Q_{\varepsilon,\phi}(U+V)
+\chi_{1}\text{M}_{\varepsilon,\phi}(V)\bigg\}\\\notag+(-\Delta+\Gamma_{\varepsilon,\phi})^{2}V
+(1-\chi_{2})F(\tilde{v}_{\varepsilon,\phi})+(1-\chi_{1})Q_{\varepsilon,\phi}(\chi_{2}U+V)+\text{N}_{\varepsilon,\phi}(U)+\text{P}_{\varepsilon,\phi}(V),
\end{eqnarray}
where
\begin{eqnarray}
\text{M}_{\varepsilon,\phi}(V):=(W^{''}(\tilde{v}_{\varepsilon,\phi})-W^{''}(1))(-\Delta V+\Gamma_{\varepsilon,\phi} V)\label{defM}\\\notag
+(-\Delta+W^{''}(\tilde{v}_{\varepsilon,\phi}))\big[(W^{''}(\tilde{v}_{\varepsilon,\phi})-W^{''}(1))V\big]\\
\text{N}_{\varepsilon,\phi}(U):=-2<\nabla\chi_{2},\nabla(-\Delta U+W^{''}(\tilde{v}_{\varepsilon,\phi})U)>-\Delta\chi_{2}(-\Delta U+W^{''}(\tilde{v}_{\varepsilon,\phi})U)\label{defN}\\\notag
+(-\Delta+W^{''}(\tilde{v}_{\varepsilon,\phi}))(-2<\nabla\chi_{2},\nabla U>-\Delta\chi_{2}U)\\
\text{P}_{\varepsilon,\phi}(V):=-2<\nabla\chi_{1},\nabla((W^{''}(\tilde{v}_{\varepsilon,\phi})-W^{''}(1))V)>
-\Delta\chi_{1}(W^{''}(\tilde{v}_{\varepsilon,\phi})-W^{''}(1))V\label{defP}\\\notag
+W^{'''}(v_{\varepsilon,\phi})(-\Delta v_{\varepsilon,\phi}+W^{'}(v_{\varepsilon,\phi}))V.
\end{eqnarray}
Hence we have reduced our problem to finding a solution $(V,U)$ to the system
\begin{eqnarray}
(-\Delta+\Gamma_{\varepsilon,\phi})^{2}V+(1-\chi_{2})F(\tilde{v}_{\varepsilon,\phi})\label{eq_aux1}\\\notag
+(1-\chi_{1})Q_{\varepsilon,\phi}(\chi_{2}U+V)+\text{N}_{\varepsilon,\phi}(U)+\text{P}_{\varepsilon,\phi}(V)=0 &\text{ in $\R^{3}$}\\
F(\tilde{v}_{\varepsilon,\phi})+F^{'}(\tilde{v}_{\varepsilon,\phi})+\chi_{1}Q_{\varepsilon,\phi}(U+V)
+\chi_{1}\text{M}_{\varepsilon,\phi}(V)=0 &\text{for $|t|\leq\tau/2\varepsilon+3$}\label{eq_aux2}.
\end{eqnarray}
The system of equations (\ref{eq_aux1}) and (\ref{eq_aux2}) is known as \textit{auxiliary equation}. First we solve equation (\ref{eq_aux1}) for any fixed $U$, thanks to coercivity, due to the fact that $\Gamma_{\varepsilon,\phi}$ is bounded away from $0$ uniformly in $\varepsilon$ and $\phi$. We will see that our solution also depends on the data $U$ and $\varepsilon$ in a Lipschitz way.
\begin{proposition}
For any $\varepsilon>0$ small enough, for any $U\in C^{4,\alpha}_{\delta,s}(\Sigma_{\varepsilon}\times\R)$ satisfying $||U||_{C^{4,\alpha}_{\delta}(\Sigma_{\varepsilon}\times\R)}\leq 1$ and for any $\phi\in B_{4}(1/4)$, equation (\ref{eq_aux1}) admits a solution $V_{\varepsilon,\phi,U}\in C^{4,\alpha}_{\delta,s}(R^{3})$ satisfying
\begin{eqnarray}
\begin{cases}
||V_{\varepsilon,\phi,U}||_{C^{4,\alpha}_{\delta}(R^{3})}\leq C_{1}e^{-a/\varepsilon}\\
||V_{\varepsilon,\phi,U_{1}}-V_{\varepsilon,\phi,U_{2}}||_{C^{4,\alpha}_{\delta}(R^{3})}\leq C_{1}e^{-a/\varepsilon}||U_{1}-U_{2}||_{C^{4,\alpha}_{\delta}(\Sigma_{\varepsilon}\times\R)}\\
||V_{\varepsilon,\phi_{1},U}-V_{\varepsilon,\phi_{2},U}||_{C^{4,\alpha}_{\delta}(R^{3})}\leq C_{1}e^{-a/\varepsilon}|\phi_{1}-\phi_{2}|_{C^{4,\alpha}(\Sigma)},
\end{cases}
\end{eqnarray}
for any $U_{1},U_{2}$ satisfying $||U_{1}||_{C^{4,\alpha}_{\delta}(\Sigma_{\varepsilon}\times\R)},||U_{2}||_{C^{4,\alpha}_{\delta}(\Sigma_{\varepsilon}\times\R)}\leq 1$, for any $\phi_{1},\phi_{2}\in B_{4}(1/4)$, for some constants $a,C_{1}>0$ independent of $U$, $\varepsilon$ and $\phi$.
\label{propaux_1}
\end{proposition}
The proof of Proposition \ref{propaux_1} is based on a fixed point argument (see section $6$).\\
 
Now we consider equation (\ref{eq_aux2}). In order to solve it, we need to extend it to the whole $\Sigma_{\varepsilon}\times\R$. First we observe that 
\begin{eqnarray}\notag
F^{'}(v_{\varepsilon,\phi})U=\mathcal{L}^{2}_{\varepsilon}U+\text{R}_{\varepsilon,\phi}(U),
\end{eqnarray}
where
\begin{eqnarray}\notag
\text{R}_{\varepsilon,\phi}(U):=\mathcal{L}_{\varepsilon}(\text{D}+W^{''}(\tilde{v}_{\varepsilon,\phi})-W^{''}(v_{\star}))(U)
+(\text{D}+W^{''}(\tilde{v}_{\varepsilon,\phi})-W^{''}(v_{\star}))\mathcal{L}_{\varepsilon}(U)\\\notag
+(\text{D}+W^{''}(\tilde{v}_{\varepsilon,\phi})-W^{''}(v_{\star}))^{2}U+
W^{'''}(\tilde{v}_{\varepsilon,\phi})(-\Delta\tilde{v}_{\varepsilon,\phi}+W^{'}(\tilde{v}_{\varepsilon,\phi}))U,
\end{eqnarray}
D is defined in (\ref{exp_laplt}). Therefore we reduced ourselves to consider 
\begin{eqnarray}
\mathcal{L}_{\varepsilon}^{2}U=-\chi_{4}F(\tilde{v}_{\varepsilon,\phi})-\chi_{1}Q_{\varepsilon,\phi}(U+V)-\chi_{4}\text{R}_{\varepsilon,\phi}(U)
-\chi_{1}\text{M}_{\varepsilon,\phi}(V)
\end{eqnarray}
in the entire $\Sigma_{\varepsilon}\times\R$. We would like to solve this equation with a fixed point argument, but, in order to do so, the right-hand side must be orthogonal to the Kernel of $\mathcal{L}^{2}_{\varepsilon}$, that is the one dimensional space generated by $v_{\star}^{'}(t)$, hence we can solve the problem
\begin{eqnarray}
\mathcal{L}_{\varepsilon}^{2}U=-\chi_{4}F(\tilde{v}_{\varepsilon,\phi})-\text{T}(U,V_{\varepsilon,\phi,U},\phi)
+p(y)v_{\star}^{'}(t)\label{proj_prob}\\\notag
\int_{-\infty}^{\infty}U(y,t)v_{\star}^{'}(t)dt=0 \text{ for any }y\in\Sigma_{\varepsilon},
\end{eqnarray}
where we have set, for the sake of simplicity, 
\begin{eqnarray}\notag
\text{T}(U,V,\phi):=\chi_{1}Q_{\varepsilon,\phi}(U+V)-\chi_{4}\text{R}_{\varepsilon,\phi}(U)-\chi_{1}\text{M}_{\varepsilon,\phi}(V)\\\notag
p(y):=\frac{1}{c_{\star}}\int_{-\infty}^{\infty}\big(\chi_{4}F(\tilde{v}_{\varepsilon,\phi})
+\text{T}(U,V_{\varepsilon,\phi,U},\phi)\big)(y,t)v_{\star}^{'}(t)dt
\end{eqnarray}
and $c_{\star}:=\int_{-\infty}^{\infty}(v_{\star}^{'}(t))^{2}dt$.

Before stating the next proposition, let us observe that any function $U:\Sigma_{\varepsilon}\times\R\to\R$ can be written as the sum of an even part and an odd part, the even part being $U_{e}(y,t):=\frac{1}{2}(U(y,t)+U(y,-t))$ and the odd part being $U_{o}(y,t):=\frac{1}{2}(U(y,t)+U(y,-t))$.
\begin{proposition}
For any $\varepsilon>0$ small enough and for any $\phi\in B_{4}(1/4)$, we can find a solution $U_{\varepsilon,\phi}\in\mathcal{E}^{4,\alpha}_{\delta,s}(\Sigma_{\varepsilon}\times\R)$ to equation (\ref{proj_prob}) satisfying
\begin{eqnarray}
\begin{cases}
||U_{\varepsilon,\phi}||_{C^{4,\alpha}_{\delta}(\Sigma_{\varepsilon}\times\R)}\leq C_{2}\varepsilon^{3}\\
||(U_{\varepsilon,\phi})_{o}||_{C^{4,\alpha}_{\delta}(\Sigma_{\varepsilon}\times\R)}\leq C_{2}\varepsilon^{4}\\
||U_{\varepsilon,\phi_{1}}-U_{\varepsilon,\phi_{2}}||_{C^{4,\alpha}_{\delta}(\Sigma_{\varepsilon}\times\R)}\leq C_{2}\varepsilon^{3}|\phi_{1}-\phi_{2}|_{C^{4,\alpha}(\Sigma)},
\end{cases}
\end{eqnarray}
for any $\phi_{1},\phi_{2}\in B_{4}(1/4)$, for some constant $C_{2}>0$ independent of $\varepsilon$.
\label{propaux_2}
\end{proposition}
The proof of Proposition \ref{propaux_2} will be given in section $6$.\\

\subsection{The bifurcation equation}

In conclusion, we will show that it is possible to find $\phi$ that solves
\begin{eqnarray}
\int_{-\infty}^{\infty}\big(\chi_{4}F(\tilde{v}_{\varepsilon,\phi})+\text{T}(U,V_{\varepsilon,\phi,U},\phi)\big)(y,t)v_{\star}^{'}(t)dt=0
\label{eqbifo_1}
\end{eqnarray}
for any $y\in\Sigma_{\varepsilon}$ and such that the real solution $u_{\varepsilon}(x):=v_{\varepsilon,\phi}(x/\varepsilon)+w_{\varepsilon,\phi}(x/\varepsilon)$ satisfies the volume constraint (\ref{constraint_vol}). First we note that, by the change of variables $x^{'}=x/\varepsilon$,
\begin{eqnarray}\notag
4\pi^{2}\sqrt{2}=\int_{\R^{3}}(1-u_{\varepsilon}(x))dx=\varepsilon^{3}\int_{\R^{3}}1-(v_{\varepsilon,\phi}(x)+w_{\varepsilon,\phi}(x))dx,
\end{eqnarray}
the latter integral can be calculated exploiting the natural change of variables 
\begin{eqnarray}
\begin{cases}
x_{1}=\varepsilon^{-1}\cos(\varepsilon\text{y}_{2})\big((z+\varepsilon^{-1})\cos(\varepsilon\text{y}_{1})+\varepsilon^{-1}\sqrt{2}\big),\\
x_{2}=\varepsilon^{-1}\sin(\varepsilon\text{y}_{2})\big((z+\varepsilon^{-1})\cos(\varepsilon\text{y}_{1})+\varepsilon^{-1}\sqrt{2}\big),\\
x_{3}=\varepsilon^{-1}(z+\varepsilon^{-1})\sin(\varepsilon\text{y}_{1}).
\end{cases}
\label{toric_coord}
\end{eqnarray}
on $V_{\tau/\varepsilon}$, induced by the parametrization $Y_{\varepsilon}(\text{y})=\varepsilon^{-1}Y(\varepsilon\text{y})$, where
\begin{eqnarray}
Y(\vartheta_{1},\vartheta_{2}):=(\cos\vartheta_{2}(\cos\vartheta_{1}+\sqrt{2}),\sin\vartheta_{2}(\cos\vartheta_{1}+\sqrt{2}),\sin\vartheta_{2})
\label{param_Sigma}
\end{eqnarray}
and $(\vartheta_{1},\vartheta_{2})=\varepsilon(\text{y}_{1},\text{y}_{2})\in[0,2\pi)^{2}$.
\begin{proposition}
For any $\varepsilon>0$ small enough, $c>0$ and $\phi\in C^{4,\alpha}(\Sigma)_{s}$ satisfying $|\phi|_{C^{4,\alpha}(\Sigma)}\leq c\varepsilon$,
\begin{eqnarray}\notag
\int_{\R^{3}}1-(v_{\varepsilon,\phi}(x)+w_{\varepsilon,\phi}(x))dx=\varepsilon^{-3}4\pi^{2}\sqrt{2}
+2\varepsilon^{-2}\int_{\Sigma}\phi(\zeta)d\sigma(\zeta)\\\notag
+8\sqrt{2}\pi^{2}\varepsilon^{-1}\int_{0}^{6+\tau/2\varepsilon}t(1-v_{\star}(t))dt+2G_{\varepsilon}(\phi),
\end{eqnarray}
with $G_{\varepsilon}$ fulfilling  
\begin{eqnarray}
\begin{cases}
|G_{\varepsilon}(\phi)|\leq c,\\\notag
|G_{\varepsilon}(\phi_{1})-G_{\varepsilon}(\phi_{2})|\leq c|\phi_{1}-\phi_{2}|_{C^{4,\alpha}(\Sigma)},
\end{cases}
\label{small_G}
\end{eqnarray}
for any $\phi,\phi_{1},\phi_{2}\in C^{4,\alpha}(\Sigma)_{s}$ satisfying $|\phi|_{C^{4,\alpha}(\Sigma)},|\phi_{1}|_{C^{4,\alpha}(\Sigma)},|\phi_{2}|_{C^{4,\alpha}(\Sigma)}\leq c\varepsilon$.
\label{prop_int_v}
\end{proposition}
The proof of this Proposition will be given in Section $7$. Therefore, in terms of $\phi$, equation (\ref{constraint_vol}) is equivalent to equation
\begin{eqnarray}
\int_{\Sigma}\phi(\zeta)d\sigma(\zeta)=-4\sqrt{2}\pi^{2}\varepsilon\int_{0}^{6+\tau/2\varepsilon}t(1-v_{\star}(t))dt-\varepsilon^{2}G_{\varepsilon}(\phi).
\label{eqbifo_2}
\end{eqnarray}
The system of equations (\ref{eqbifo_1}) and (\ref{eqbifo_2}) is known as \textit{bifurcation equation}, and it will be solved by a fixed point argument, that will be explained in this Proposition, whose proof will be carried out in Section $7$. 
\begin{proposition}
For any $\varepsilon>0$ small enough, the system of equations (\ref{eqbifo_1}) and (\ref{eqbifo_2}) admits a solution $\phi\in C^{4,\alpha}(\Sigma)_{s}$ satisfying $|\phi|_{C^{4,\alpha}(\Sigma)}\leq C_{3}\varepsilon$, for some constant $C_{3}=C_{3}(W,\tau)>0$.
\label{prop_bifo}
\end{proposition}
\begin{remark}
As we will see in the proof of Proposition \ref{solve_eq_bifo} below, the Willmore equation will appear at order $\varepsilon^{3}$, while the linearized operator
\begin{eqnarray}
\tilde{L}_{0}\phi=L_{0}^{2}\phi+\frac{3}{2}H^{2}L_{0}\phi-H(\nabla_{\Sigma}\phi,\nabla_{\Sigma}H)+2(A\nabla_{\Sigma}\phi,\nabla_{\Sigma}H)+\\\notag
2H<A,\nabla^{2}\phi>+\phi(2<A,\nabla^{2}H>+|\nabla_{\Sigma}H|^{2}+2H\tr A^{3}).
\end{eqnarray}
will appear at order $\varepsilon^{4}$, thus it is crucial for the remainder to be smaller in order to apply a contraction mapping principle. This is actually the case thanks to the fact that the odd part of $U_{\varepsilon,\phi}$ is of order $\varepsilon^{4}$.
\end{remark}

\section{The approximate solution}
\subsection{Construction}
First one can try to take $v_{\star}(t)$ as an approximate solution. We recall that $t=z-\phi(\varepsilon y)$, where $\phi\in B_{4}(1/4)$ is some small function that respects the symmetries of the $\Sigma$. We will see that these symmetries will be inherited by the approximate solution (see Remark \ref{rem_symm_v} below). Since the Fermi coordinates are just defined in a neighbourhood of the Torus, our approximate solution is not defined everywhere. For our purposes, it is enough to consider it in the set
\begin{eqnarray}
B=\{x=Z_{\varepsilon}(y,t+\phi(\varepsilon y))\in\R^{3}:|t|<\tau/2\varepsilon+5\},
\end{eqnarray}
that is a tubular neighbourhood of 
\begin{eqnarray}\notag
\Sigma_{\varepsilon,\phi}=\{y+\phi(\varepsilon y)\nu(\varepsilon y):y\in\Sigma_{\varepsilon}\}
\end{eqnarray}
of width $\tau/4\varepsilon$. Then it will be extended to the whole $\R^{3}$ with the aid of a cutoff function. 

In the sequel, $v_{\star}$ and its derivatives will always be evaluated at $t$, the geometric quantities, $\phi$ and its derivatives will always be evaluated at $\varepsilon y$. By (\ref{exp_lapl}), 
\begin{eqnarray}
-\Delta v_{\star}+W^{'}(v_{\star})=-v_{\star}^{''}+W^{'}(v_{\star})+\varepsilon\hat{H}(\varepsilon y,\varepsilon (t+\phi))v_{\star}^{'}\label{first_time}\\\notag
+\varepsilon^{2}\Delta_{\Sigma}\phi v_{\star}^{'}-\varepsilon^{2}|\nabla\phi|^{2}v_{\star}^{''}
+\varepsilon^{3}(t+\phi)(a_{1}^{ij}\phi_{ij}+b_{1}^{i}\phi_{i})v_{\star}^{'}-\varepsilon^{3}(t+\phi)a_{1}^{ij}\phi_{i}\phi_{j}v_{\star}^{''}\\\notag
+\varepsilon^{4}(t+\phi)^{2}(a_{2}^{ij}\phi_{ij}+b_{2}^{i}\phi_{i})v_{\star}^{'}-\varepsilon^{4}(t+\phi)^{2}a_{2}^{ij}\phi_{i}\phi_{j}v_{\star}^{''}\\\notag
\varepsilon^{2}(\overline{a}^{ij}\phi_{ij}+\overline{b}^{i}\phi_{i})v_{\star}^{'}-\varepsilon^{2}\overline{a}^{ij}\phi_{i}\phi_{j}v_{\star}^{''}.
\end{eqnarray}
The term of order $0$ in $\varepsilon$ vanishes since $v_{\star}$ satisfies the ODE $-v_{\star}^{''}+W^{'}(v_{\star})=0$. Thus, in order to compute $F(v_{\star})$, we need to apply the linear operator $-\Delta+W^{''}(v_{\star})$ to the remaining terms. We will write down all terms of order less or equal than $4$, the other ones being lower order terms, in some sense that will be clear soon. 
Let us set, for any function $v\in C^{2}(\R)$, $L_{\star}v:=-v^{''}+W^{''}(v_{\star})v$. Differentiating the ODE satisfied by $v_{\star}$, we get  $L_{\star}v_{\star}^{'}=0$, thus using the Taylor expansion of $\tilde{H}$, the first term of (\ref{first_time}) gives
\begin{eqnarray}
T^{1}_{\varepsilon,\phi}(y,t)=\big(-\Delta+W^{''}(v_{\star})\big)(\varepsilon\hat{H}(\varepsilon y,\varepsilon(t+\phi))v_{\star}^{'})=\varepsilon^{2}(H^{2}-2|A|^{2})v_{\star}^{''}\\\notag
+\varepsilon^{3}\bigg\{(2H|A|^{2}-4\tr A^{3})(t+\phi)v_{\star}^{''}+(H|A|^{2}-2\tr A^{3})v_{\star}^{'}-\Delta_{\Sigma}Hv_{\star}^{'}\\\notag
+2(\nabla_{\Sigma}H,\nabla_{\Sigma}\phi)v_{\star}^{''}-H|\nabla_{\Sigma}\phi|^{2}v_{\star}^{'''}+H\Delta_{\Sigma}\phi v_{\star}^{''}\bigg\}\\\notag
+\varepsilon^{4}\bigg\{(|A|^{4}-6H+2H\tr A^{3})((t+\phi)^{2}v_{\star}^{''}+(t+\phi)v_{\star}^{'})-\Delta_{\Sigma}|A|^{2}(t+\phi)v_{\star}^{'}\\\notag
+2(\nabla_{\Sigma}|A|^{2},\nabla_{\Sigma}\phi)(t+\phi)v_{\star}^{''}
-|A|^{2}|\nabla_{\Sigma}\phi|^{2}(t+\phi)v_{\star}^{'''}+\Delta_{\Sigma}\phi|A|^{2}(t+\phi)v_{\star}^{''}\\\notag
-(a_{1}^{ij}H_{ij}+b_{1}^{i}H_{i})(t+\phi)v_{\star}^{'}
+2a_{1}^{ij}H_{i}\phi_{j}(t+\phi)v_{\star}^{''}\\\notag
+H(a_{1}^{ij}\phi_{ij}+b_{1}^{i}\phi_{i})(t+\phi)v_{\star}^{''}-Ha_{1}^{ij}\phi_{i}\phi_{j}v_{\star}^{'''}\bigg\}+\varepsilon^{5}F^{1}_{\varepsilon,\phi}(y,t),
\end{eqnarray}
with $F^{1}_{\varepsilon,\phi}$ small and Lipschitzian in $\phi$, in the sense that
\begin{eqnarray}
\begin{cases}
|S_{\varepsilon,\phi}F^{1}_{\varepsilon,\phi}|_{C^{0,\alpha}_{\gamma}(\R^{3})}\leq c\\
|S_{\varepsilon,\phi}F^{1}_{\varepsilon,\phi_{1}}-S_{\varepsilon,\phi_{2}}F^{1}_{\varepsilon,\psi}|_{C^{0,\alpha}(\Sigma)}\leq c|\phi_{1}-\phi_{2}|_{C^{4,\alpha}(\Sigma)},
\end{cases}
\label{lip_phi}
\end{eqnarray}
for any $\phi,\phi_{1},\phi_{2}\in B_{4}(\tau/4)$, for some constant $c=c(W,\tau)>0$ independent of $\varepsilon$ and $\phi$.

Similarly, the second term of (\ref{first_time}) gives
\begin{eqnarray}
T^{2}_{\varepsilon,\phi}(y,t)=\big(-\Delta+W^{''}(v_{\star})\big)(\varepsilon^{2}\Delta_{\Sigma}\phi v_{\star}^{'})
=\varepsilon^{3}H\Delta_{\Sigma}\phi v_{\star}^{''}\\\notag
+\varepsilon^{4}\bigg\{-(\Delta_{\Sigma})^{2}\phi v_{\star}^{'}
+|A|^{2}\Delta_{\Sigma}\phi(t+\phi)v_{\star}^{''}
+2(\nabla_{\Sigma}\Delta_{\Sigma}\phi,\nabla_{\Sigma}\phi)v_{\star}^{''}\\\notag
+(\Delta_{\Sigma}\phi)^{2}v_{\star}^{''}
-|\nabla_{\Sigma}\phi|^{2}\Delta_{\Sigma}\phi v_{\star}^{'''}\bigg\}+\varepsilon^{5}F^{2}_{\varepsilon,\phi}(y,t),
\end{eqnarray}
with $F^{2}_{\varepsilon,\phi}$ fulfilling (\ref{lip_phi}).

The third term of (\ref{first_time}) is already quadratic in $\phi$, but, for the sake of completeness, we prefer to write it down.
\begin{eqnarray}
T^{3}_{\varepsilon,\phi}(y,t)=\big(-\Delta+W^{''}(v_{\star})\big)(-\varepsilon^{2}|\nabla_{\Sigma}\phi|^{2}v_{\star}^{''})=
\varepsilon^{2}|\nabla_{\Sigma}\phi|^{2}(v_{\star}^{(4)}-W^{''}(v_{\star})v^{''}_{\star})\\\notag
-\varepsilon^{3}H|\nabla_{\Sigma}\phi|^{2}v_{\star}^{'''}+\varepsilon^{4}\bigg\{-|A|^{2}|\nabla_{\Sigma}\phi|^{2}(t+\phi)v_{\star}^{'''}
+\Delta_{\Sigma}|\nabla_{\Sigma}\phi|^{2}v_{\star}^{''}\\\notag
-2(\nabla_{\Sigma}|\nabla_{\Sigma}\phi|^{2},\nabla_{\Sigma}\phi)v_{\star}^{'''}+|\nabla_{\Sigma}\phi|^{4}v_{\star}^{(4)}
-|\nabla_{\Sigma}\phi|^{2}\Delta_{\Sigma}\phi v_{\star}^{'''}\bigg\}
+\varepsilon^{5}F^{3}_{\varepsilon,\phi}(y,t)
\end{eqnarray}
The fouth term of (\ref{first_time}) gives
\begin{eqnarray}
T^{4}_{\varepsilon,\phi}(y,t)=\big(-\Delta+W^{''}(v_{\star})\big)\big(\varepsilon^{3}(a_{1}^{ij}\phi_{ij}+b_{1}^{i}\phi_{i})(t+\phi)v_{\star}^{'}\big)=\\\notag
-2\varepsilon^{3}(a_{1}^{ij}\phi_{ij}+b_{1}^{i}\phi_{i})v_{\star}^{''}
+\varepsilon^{4}H(a_{1}^{ij}\phi_{ij}+b_{1}^{i}\phi_{i})(v_{\star}^{'}+(t+\phi)v_{\star}^{''})+\varepsilon^{5}F^{4}_{\varepsilon,\phi}(y,t).
\end{eqnarray}
The fifth term of (\ref{first_time}) gives
\begin{eqnarray}\notag
T^{5}_{\varepsilon,\phi}(y,t)=\big(-\Delta+W^{''}(v_{\star})\big)(-\varepsilon^{3}a_{1}^{ij}\phi_{i}\phi_{j}(t+\phi)v_{\star}^{''})=\\\notag
\varepsilon^{3}a_{1}^{ij}\phi_{i}\phi_{j}((t+\phi)v_{\star}^{(4)}-(t+\phi)W^{''}(v_{\star})v^{''}_{\star}+2v_{\star}^{'''})\\\notag
-\varepsilon^{4}Ha_{1}^{ij}\phi_{i}\phi_{j}(v_{\star}^{''}+(t+\phi)v_{\star}^{'''})+\varepsilon^{5}F^{5}_{\varepsilon,\phi}(y,t),
\end{eqnarray}
with $F^{3}_{\varepsilon,\phi},F^{4}_{\varepsilon,\phi},F^{5}_{\varepsilon,\phi}$ fulfilling (\ref{lip_phi}).

Now we consider the terms involving $a_{2}^{ij}$ and $b_{2}^{i}$. We will see that all the contributions of order $\varepsilon^{4}$ coming from these terms will simplify, therefore we do not need to know the explicit expression of $a_{2}^{ij}$ and $b_{2}^{i}$.
\begin{eqnarray}
T^{6}_{\varepsilon,\phi}(y,t)=\bigg\{\big(-\Delta+W^{''}(v_{\star})\big)\big(\varepsilon^{4}(a_{2}^{ij}\phi_{ij}+b_{2}^{i}\phi_{i})(t+\phi)^{2}v_{\star}^{'}
-\varepsilon^{4}a_{2}^{ij}\phi_{i}\phi_{j}(t+\phi)^{2}v_{\star}^{''}\big)\\\notag
+\varepsilon^{2}(\overline{a}^{ij}\phi_{ij}+\overline{b}^{i}\phi_{i})v_{\star}^{'}
-\varepsilon^{2}\overline{a}^{ij}\phi_{i}\phi_{j}v_{\star}^{''}\bigg\}\\\notag
=-\varepsilon^{4}(a_{2}^{ij}\phi_{ij}+b_{2}^{i}\phi_{i})(2v_{\star}^{'}+4(t+\phi)v_{\star}^{''})\\\notag
-\varepsilon^{4}a_{2}^{ij}\phi_{i}\phi_{j}(2v_{\star}^{''}+4(t+\phi)v_{\star}^{'''}+(t+\phi)^{2}v^{(4)}_{\star}+W^{''}(v_{\star})v_{\star}^{''})\\\notag
+\big(-\Delta+W^{''}(v_{\star})\big)\bigg\{\varepsilon^{2}(\overline{a}^{ij}\phi_{ij}+\overline{b}^{i}\phi_{i})v_{\star}^{'}
-\varepsilon^{2}\overline{a}^{ij}\phi_{i}\phi_{j}v_{\star}^{''}\bigg\}+\varepsilon^{5}F^{6}_{\varepsilon,\phi}(y,t),
\end{eqnarray}
with $F^{6}_{\varepsilon,\phi}$ fulfilling (\ref{lip_phi}).



It turns out that, in the expansion of $F(v_{\star}(t))$, the only term of order $\varepsilon^{2}$ is $\varepsilon^{2}(H^{2}-2|A|^{2})v_{\star}^{''}$. Since it is too large for our purposes, we add a correction to the approximate solution in order to cancel it.

We set
\begin{eqnarray}\notag
\eta(t)=-v_{\star}^{'}(t)\int_{0}^{t}(v_{\star}^{'}(s))^{-2}ds\int_{0}^{s}\frac{\tau (v_{\star}^{'}(\tau))^{2}}{2}d\tau.
\end{eqnarray}
This function is exponentially decaying, odd and solves 
\begin{eqnarray}\notag
L_{\star}\eta(t)=-\eta^{''}(t)+W^{''}(v_{\star}(t))\eta(t)=\frac{1}{2}tv_{\star}^{'}(t)\\\notag
\int_{-\infty}^{\infty}\eta(t)v_{\star}^{'}(t)dt=0.
\end{eqnarray}
Differentiating this relation once more, it is possible to see that $L_{\star}^{2}\eta(t)=-v_{\star}^{''}(t)$. Our new approximate solution will be
\begin{eqnarray}
\tilde{v}_{\varepsilon}(y,t)=v_{\star}(t)+\varepsilon^{2}(\psi(\varepsilon y)+\varepsilon L\phi(\varepsilon y))\eta(t),
\end{eqnarray}
with $\psi:\Sigma\to\R$ and $L$ linear in $\phi$ to be determined later. In the sequel, $\eta$ and its derivatives are evaluated at $t$, the geometric quantities, $\phi$ and its derivatives will be evaluated at $\varepsilon y$. Taking the Taylor expansion of $F_{\varepsilon}$, 
\begin{eqnarray}\notag
F(\tilde{v}_{\varepsilon,\phi}(y,t))=F(v_{\star})+F^{'}(v_{\star})\big(\varepsilon^{2}(\psi+\varepsilon L\phi)\eta\big)\\\notag
+F^{''}(v_{\star})\big[\varepsilon^{2}(\psi+\varepsilon L\phi)\eta,\varepsilon^{2}(\psi+\varepsilon L\phi)\eta\big]+C_{\varepsilon,\phi}[\varepsilon^{2}(\psi(\varepsilon y)+\varepsilon L\phi(\varepsilon y))\eta],
\end{eqnarray}
where
\begin{eqnarray}\notag
C_{\varepsilon,\phi}[w]=\int_{0}^{1}dt\int_{0}^{t}ds\int_{0}^{s}F^{'''}\big(v_{\star}+\tau w\big)[w,w,w]d\tau.
\end{eqnarray}
Now we have to compute $F^{'}(v_{\star})\big(\varepsilon^{2}(\psi(\varepsilon y)+\varepsilon L\phi(\varepsilon y))\eta\big)$. As first we note that
\begin{eqnarray}\notag
T^{7}_{\varepsilon,\phi}(y,z)=W^{'''}(v_{\star})(-\Delta v_{\star}+W^{'}(v_{\star}))\varepsilon^{2}(\psi(\varepsilon y)+\varepsilon L\phi(\varepsilon y))\eta
=\varepsilon^{3}H\psi W^{'''}(v_{\star})\eta v_{\star}^{'}\\\notag
+\varepsilon^{4}(\psi\Delta_{\Sigma}\phi+HL\phi+(t+\phi)\psi|A|^{2})W^{'''}(v_{\star})\eta v_{\star}^{'}+\varepsilon^{5}F^{7}_{\varepsilon,\phi}(y,t),
\end{eqnarray}
with $F^{7}_{\varepsilon,\phi}$ fulfilling (\ref{lip_phi}).

After that, we have to compute $\big(-\Delta+W^{''}(v_{\star})\big)^{2}(\varepsilon^{2}(\psi+\varepsilon L\phi)\eta)$. We obtain
\begin{eqnarray}\notag
\big(-\Delta+W^{''}(v_{\star})\big)(\varepsilon^{2}(\psi+\varepsilon L\phi)\eta)=
\varepsilon^{2}\psi L_{\star}\eta+\varepsilon^{3}(H\psi\eta^{'}+L\phi L_{\star}\eta)\\\notag
+\varepsilon^{4}\bigg\{-\Delta_{\Sigma}\psi\eta+\big(|A|^{2}\psi(t+\phi)+HL\phi+2(\nabla_{\Sigma}\psi,\nabla_{\Sigma}\phi)\\\notag
+\psi\Delta_{\Sigma}\phi\big)\eta^{'}-\psi|\nabla_{\Sigma}\phi|^{2}\eta^{''}\bigg\}
+\varepsilon^{5}\tilde{F}_{\varepsilon,\phi}(y,t),
\label{first_time_corr}
\end{eqnarray}
with $\tilde{F}_{\varepsilon,\phi}$ satisfying (\ref{lip_phi}).

Applying the operator once more, we obtain
\begin{eqnarray}
T^{8}_{\varepsilon,\phi}(y,t)=\big(-\Delta+W^{''}(v_{\star})\big)(\varepsilon^{2}(\psi+\varepsilon L\phi)L_{\star}\eta)=\varepsilon^{2}\psi L_{\star}^{2}\eta
+\varepsilon^{3}\bigg\{L\phi L_{\star}^{2}\eta+H\psi(L_{\star}\eta)^{'}\bigg\}\\\notag
+\varepsilon^{4}\bigg\{-\Delta_{\Sigma}\psi L_{\star}\eta+\big(|A|^{2}\psi (t+\phi)+HL\phi+2(\nabla_{\Sigma}\psi,\nabla_{\Sigma}\phi)+\psi\Delta_{\Sigma}\phi\big)(L_{\star}\eta)^{'}\\\notag
-\psi|\nabla_{\Sigma}\phi|^{2}(L_{\star}\eta)^{''}\bigg\}+\varepsilon^{5}F^{8}_{\varepsilon,\phi}(y,t),
\end{eqnarray}
with $F^{8}_{\varepsilon,\phi}$ satisfying (\ref{lip_phi}).

Moreover,
\begin{eqnarray}
T^{9}_{\varepsilon,\phi}(y,t)=\big(-\Delta+W^{''}(v_{\star})\big)(\varepsilon^{3}H\eta^{'})=\varepsilon^{3}\psi HL_{\star}(\eta^{'})
+\varepsilon^{4}H^{2}\psi\eta^{''}+\varepsilon^{5}F^{9}_{\varepsilon,\phi}(y,t),
\end{eqnarray}
with $F^{9}_{\varepsilon,\phi}$ satisfying (\ref{lip_phi}).

As regards the term of order $\varepsilon^{4}$ of (\ref{first_time_corr}), we note that
\begin{eqnarray}
T^{10}_{\varepsilon,\phi}(y,t)=\varepsilon^{4}\big(-\Delta+W^{''}(v_{\star})\big)\bigg\{-\Delta_{\Sigma}\psi\eta+\big(|A|^{2}\psi (t+\phi)+HL\phi+2(\nabla_{\Sigma}\psi,\nabla_{\Sigma}\phi)\\\notag
+\psi\Delta_{\Sigma}\phi\big)\eta^{'}-\psi|\nabla_{\Sigma}\phi|^{2}\eta^{''}\bigg\}
=\varepsilon^{4}\bigg\{-\Delta_{\Sigma}\psi L_{\star}\eta
+\big(HL\phi+2(\nabla_{\Sigma}\psi,\nabla_{\Sigma}\phi)+\psi\Delta_{\Sigma}\phi\big)L_{\star}(\eta^{'})\\\notag
+|A|^{2}\psi L_{\star}((t+\phi)\eta^{'})-\psi|\nabla_{\Sigma}\phi|^{2}L_{\star}(\eta^{''})\bigg\}+\varepsilon^{5}F^{10}_{\varepsilon,\phi}(y,t),
\end{eqnarray}
with $F^{10}_{\varepsilon,\phi}$ satisfying (\ref{lip_phi}). To conclude, also
\begin{eqnarray}\notag
F^{11}_{\varepsilon,\phi}(y,t)=\big(-\Delta+W^{''}(v_{\star})\big)\tilde{F}_{\varepsilon,\phi}(y,t)
\end{eqnarray}
is negligible, that is it satisfies (\ref{lip_phi}), since $\tilde{F}_{\varepsilon,\phi}$ does.

The only term of order $\varepsilon^{2}$ in $F^{'}_{\varepsilon}(v_{\star})\big(\varepsilon^{2}(\psi(\varepsilon y)+\varepsilon L\phi(\varepsilon y))\eta\big)$ is $\varepsilon^{2}\psi L_{\star}\eta=-\varepsilon^{2}\psi v^{''}_{\star}$. Since we want it to erase the term of order $\varepsilon^{2}$ of $F_{\varepsilon}(v_{\star})$, we could set $\psi:=H^{2}-2|A|^{2}$. However, some quadratic terms appear at order $\varepsilon^{3}$. The only one that gives rise to some problems is $-2H|\nabla_{\Sigma}\phi|^{2}v_{\star}^{'''}$, thus we set $\psi:=H^{2}-2|A|^{2}+d|\nabla_{\Sigma}\phi|^{2}$, for some constant $d$ to be determined after projection. In particular, $\nabla_{\Sigma}\psi=2H\nabla_{\Sigma} H-2\nabla_{\Sigma}|A|^{2}+d\nabla_{\Sigma}|\nabla_{\Sigma}\phi|^{2}$. $L$ will be determined after projection.

Now we have to considered the contribution of $F^{''}_{\varepsilon}(v_{\star})\big(\varepsilon^{2}(\psi+\varepsilon L\phi)\eta\big)$, since it gives rise to a term of order $\varepsilon^{4}$. However, we will see that this contribution will cancel after projection
\begin{eqnarray}
F^{''}(v_{\star})\big(\varepsilon^{2}(\psi+\varepsilon L\phi)\eta\big)=\varepsilon^{4}W^{'''}(v_{\star})W^{''}(v_{\star})\psi^{2}\eta^{2}+\varepsilon^{5}F^{12}_{\varepsilon,\phi}(y,t),
\end{eqnarray}
with $F^{12}_{\varepsilon,\phi}$ satisfying (\ref{lip_phi}).\\

We recall that $\tilde{v}_{\varepsilon,\phi}$ is just defined in $B$, while our global approximate solution is $v_{\varepsilon,\phi}(x)=\chi_{5}(x)\tilde{v}_{\varepsilon,\phi}(y,t)+(1-\chi_{5}(x))\mathbb{H}(x)$ (see (\ref{defv})).
\begin{remark}
It follows from the construction that our approximate solution respects the symmetries of the Torus, that is $v_{\varepsilon,\phi}(x)=v_{\varepsilon,\phi}(Tx)$ and $v_{\varepsilon,\phi}(x)=v_{\varepsilon,\phi}(Rx)$, for any $R\in SO_{x_{3}}(3)$.
\label{rem_symm_v}
\end{remark}

\subsection{Projection}
As we noticed in section $4,2$, we need to consider the projection of the error $F(\tilde{v}_{\varepsilon,\phi})$. In this subsection, we will explain how to do and we will see that this projection also enables us to choose $L$ and $d$.
\begin{proposition}
Let us set, for any $\phi\in B_{4}(\tau/4)$,
\begin{eqnarray}
L\phi:=-4<A,\nabla^{2}\phi>+2H\Delta_{\Sigma}\phi+\phi(2H|A|^{2}-4\tr A^{3}),&\text{ }d=-4b_{\star}/c_{\star},
\label{choiceL}
\end{eqnarray}
where $c_{\star}:=\int_{-\infty}^{\infty}(v_{\star}^{'}(t))^{2}dt$ and $b_{\star}:=\int_{-\infty}^{\infty}(v_{\star}^{''}(t))^{2}dt$. 
Then, for any $y\in\Sigma_{\varepsilon}$, the projection of $F_{\varepsilon}(\tilde{v}_{\varepsilon,\phi})$ satisfies
\begin{eqnarray}
\int_{-\infty}^{\infty}F(\tilde{v}_{\varepsilon,\phi})(y,t)v_{\star}^{'}(t)dt=
-\varepsilon^{4}c_{\star}\tilde{L_{0}}\phi(\varepsilon y)+\varepsilon^{5}\mathcal{F}_{\varepsilon,\phi}(\varepsilon y),
\end{eqnarray}
with $\mathcal{F}_{\varepsilon,\phi}$ uniformly bounded and Lipschitzian in $\phi\in B_{4}(\tau/4)$ and in $\varepsilon$, that is there exists a constant $c=c(W,\tau)>0$ such that 
\begin{eqnarray}
\begin{cases}
|\mathcal{F}_{\varepsilon,\phi}|_{C^{0,\alpha}(\Sigma)}\leq c,\\
|\mathcal{F}_{\varepsilon,\phi_{1}}-\mathcal{F}_{\varepsilon,\phi_{2}}|_{C^{0,\alpha}(\Sigma)}\leq c|\phi_{1}-\phi_{2}|_{C^{4,\alpha}(\Sigma)},
\end{cases}
\label{lip_ep_s}
\end{eqnarray}
for any $\phi,\phi_{1},\phi_{2}\in B_{4}(\tau/4)$ and for any $\varepsilon>0$ small enough.
\label{solve_eq_bifo}
\end{proposition}
\begin{proof}
Above we computed $F_{\varepsilon}(\tilde{v}_{\varepsilon,\phi})$ using (\ref{lapl_yt}), now we just project it term by term.

Integrating by parts we can show that
\begin{eqnarray}
\int_{-\infty}^{\infty}tv_{\star}^{''}(t)v_{\star}^{'}(t)dt=-\frac{1}{2}c_{\star}\label{int1}\\
\int_{-\infty}^{\infty}L_{\star}\eta(t)v_{\star}^{'}(t)dt=\frac{1}{4}c_{\star}\label{int2}\\
\int_{-\infty}^{\infty}L_{\star}(\eta^{'}(t))v_{\star}^{'}(t)dt=0,\label{int3}
\end{eqnarray}
so in particular
\begin{eqnarray}\notag
\int_{-\infty}^{\infty}W^{'''}(v_{\star}(t))\eta(t)(v_{\star}^{'}(t))^{2}dt
=\int_{-\infty}^{\infty}\big\{L_{\star}\eta(t)-L_{\star}(\eta^{'}(t))\big\}v_{\star}^{'}(t)dt=\frac{1}{4}c_{\star}.\label{int4}
\end{eqnarray}
Moreover, setting $b_{\star}:=\int_{-\infty}^{\infty}(v_{\star}^{''}(t))^{2}dt=-\int_{-\infty}^{\infty}v_{\star}^{'''}(t)v_{\star}^{'}(t)dt$, we can see that
\begin{eqnarray}
\int_{-\infty}^{\infty}\big\{tv_{\star}^{(4)}(t)-tW^{''}(v_{\star}(t))v_{\star}^{''}(t)
+2v_{\star}^{'''}(t)\big\}v_{\star}^{'}(t)dt=\label{int5}\\\notag
-\int_{-\infty}^{\infty}tL_{\star}(v_{\star}^{''}(t))v_{\star}^{'}(t)dt-2b_{\star}=-\int_{-\infty}^{\infty}tv_{\star}^{''}(t)
L_{\star}(v_{\star}^{'}(t))dt+2b_{\star}-2b_{\star}=0
\end{eqnarray}
because $L_{\star}(v_{\star}^{'})=0$.

In the forthcoming calculations, right-hand side will always be evaluated at $\varepsilon y$. By (\ref{int1}) and (\ref{geom_rel}),
\begin{eqnarray}\notag
\int_{-\infty}^{\infty}\big\{T^{1}_{\varepsilon,\phi}(y,t)-\varepsilon^{2}(H^{2}-2|A|^{2})v_{\star}^{''}(t)\big\}v_{\star}^{'}(t)dt
=\varepsilon^{3}\bigg\{-c_{\star}\Delta_{\Sigma}H+b_{\star}H|\nabla_{\Sigma}\phi|^{2}\bigg\}\\\notag
+\varepsilon^{4}c_{\star}\bigg\{-\phi\Delta_{\Sigma}|A|^{2}-(\nabla_{\Sigma}|A|^{2},\nabla_{\Sigma}\phi)-\frac{1}{2}|A|^{2}\Delta_{\Sigma}\phi
-\phi(2<A,\nabla^{2}H>+|\nabla_{\Sigma}H|^{2})\\\notag
-2(A\nabla_{\Sigma}H,\nabla_{\Sigma}\phi)-\frac{1}{2}H(2<A,\nabla^{2}\phi>+(\nabla_{\Sigma}H,\nabla_{\Sigma}\phi))\bigg\}
+\varepsilon^{5}\mathcal{F}^{1}_{\varepsilon,\phi},
\end{eqnarray}
\begin{eqnarray}\notag
\int_{-\infty}^{\infty}T^{2}_{\varepsilon,\phi}(y,t)v_{\star}^{'}(t)dt=\varepsilon^{4}c_{\star}\bigg\{-(\Delta_{\Sigma})^{2}\phi
-\frac{1}{2}|A|^{2}\Delta_{\Sigma}\phi\bigg\}
+\varepsilon^{5}\mathcal{F}^{2}_{\varepsilon,\phi},
\end{eqnarray}
\begin{eqnarray}\notag
\int_{-\infty}^{\infty}T^{3}_{\varepsilon,\phi}(y,t)v_{\star}^{'}(t)dt=\varepsilon^{3}b_{\star}H|\nabla_{\Sigma}\phi|^{2}
+\varepsilon^{5}\mathcal{F}^{3}_{\varepsilon,\phi},
\end{eqnarray}
\begin{eqnarray}\notag
\int_{-\infty}^{\infty}T^{4}_{\varepsilon,\phi}(y,t)v_{\star}^{'}(t)dt=\varepsilon^{4}\frac{1}{2}c_{\star}H(2<A,\nabla^{2}\phi>
+(\nabla_{\Sigma}H,\nabla_{\Sigma}\phi))+\varepsilon^{5}\mathcal{F}^{4}_{\varepsilon,\phi},
\end{eqnarray}
with $\mathcal{F}^{1}_{\varepsilon,\phi},\mathcal{F}^{2}_{\varepsilon,\phi},\mathcal{F}^{3}_{\varepsilon,\phi},\mathcal{F}^{4}_{\varepsilon,\phi}$ satisfying (\ref{lip_ep_s}).

By (\ref{int5}),
\begin{eqnarray}\notag
\int_{-\infty}^{\infty}T^{5}_{\varepsilon,\phi}(y,t)v_{\star}^{'}(t)dt=\\\notag
\varepsilon^{3}2(A\nabla_{\Sigma}\phi,\nabla_{\Sigma}\phi)\int_{-\infty}^{\infty}\big\{tv_{\star}^{(4)}(t)-tW^{''}(v_{\star}(t))v_{\star}^{''}(t)
+2v_{\star}^{'''}(t)\big\}v_{\star}^{'}(t)dt+\varepsilon^{5}\mathcal{F}^{5}_{\varepsilon,\phi}=\varepsilon^{5}\mathcal{F}^{5}_{\varepsilon,\phi},
\end{eqnarray}
with $\mathcal{F}^{5}_{\varepsilon,\phi}$ satisfying (\ref{lip_ep_s}). Once again by (\ref{int1}), we can see that
\begin{eqnarray}\notag
\int_{-\infty}^{\infty}T^{6}_{\varepsilon,\phi}(y,t)v_{\star}^{'}(t)dt=\varepsilon^{5}\mathcal{F}^{6}_{\varepsilon,\phi},
\end{eqnarray}
with $\mathcal{F}^{6}_{\varepsilon,\phi}$ satisfying (\ref{lip_ep_s}).

Now let us consider the terms coming from the correction.
\begin{eqnarray}\notag
\int_{-\infty}^{\infty}T^{7}_{\varepsilon,\phi}(y,t)v_{\star}^{'}(t)dt=\varepsilon^{3}c_{\star}\bigg\{\frac{1}{4}H(H^{2}-2|A|^{2})
+dH|\nabla_{\Sigma}\phi|^{2}\bigg\}\\\notag
+\varepsilon^{4}c_{\star}\frac{1}{4}\bigg\{(H^{2}-2|A|^{2})\Delta_{\Sigma}\phi+HL\phi+(H^{2}-2|A|^{2})|A|^{2}\phi\bigg\}
+\varepsilon^{5}\mathcal{F}^{7}_{\varepsilon,\phi},
\end{eqnarray}
\begin{eqnarray}\notag
\int_{-\infty}^{\infty}\big\{T^{8}_{\varepsilon,\phi}(y,t)-\varepsilon^{2}(H^{2}-2|A|^{2})L_{\star}\eta(t)\big\}v_{\star}^{'}(t)dt=
\varepsilon^{3}c_{\star}\frac{1}{4}\bigg\{H(H^{2}-2|A|^{2})\\\notag
+dH|\nabla_{\Sigma}\phi|^{2}\bigg\}+\varepsilon^{4}c_{\star}\bigg\{\frac{1}{4}(H^{2}-2|A|^{2})|A|^{2}\phi+\frac{1}{4}HL\phi
+H(\nabla_{\Sigma}H,\nabla_{\Sigma}\phi)\\\notag
-(\nabla_{\Sigma}|A|^{2},\nabla_{\Sigma}\phi)+\frac{1}{4}(H^{2}-2|A|^{2})\Delta_{\Sigma}\phi\bigg\}+\varepsilon^{5}\mathcal{F}^{8}_{\varepsilon,\phi},
\end{eqnarray}
with $\mathcal{F}^{7}_{\varepsilon,\phi},\mathcal{F}^{8}_{\varepsilon,\phi}$ satisfying (\ref{lip_ep_s}). To conclude, also
\begin{eqnarray}
\mathcal{F}^{9}_{\varepsilon,\phi}=\int_{-\infty}^{\infty}\big\{T^{9}_{\varepsilon,\phi}(y,t)+T^{10}_{\varepsilon,\phi}(y,t)\big\}
v_{\star}^{'}(t)dt
\end{eqnarray}
fulfills (\ref{lip_ep_s}). In conclusion, we choose $L$ and $d$ as in (\ref{choiceL}) in order to cancel the quadratic term appearing at order $\varepsilon^{3}$ and to obtain exactly $\tilde{L}_{0}$ as a linear term at order $\varepsilon^{4}$. Since $\Sigma$ is a Willmore surface, that is it satisfies the Euler equation $-\Delta_{\Sigma}H+\frac{1}{2}H(H^{2}-2|A|^{2})=0$, we have
\begin{eqnarray}\notag
\int_{-\infty}^{\infty}F(\tilde{v}_{\varepsilon,\phi})v_{\star}^{'}(t)dt=\varepsilon^{3}\bigg\{c_{\star}\big(-\Delta_{\Sigma}H+\frac{1}{2}H
(H^{2}-2|A|^{2})\big)(\varepsilon y)\bigg\}\\\notag
-\varepsilon^{4}c_{\star}\tilde{L_{0}}\phi(\varepsilon y)+\varepsilon^{5}\mathcal{H}_{\varepsilon,\phi}(\varepsilon y)=\\\notag
-\varepsilon^{4}c_{\star}\tilde{L_{0}}\phi(\varepsilon y)+\varepsilon^{5}\mathcal{H}_{\varepsilon,\phi}(\varepsilon y),
\end{eqnarray}
where $\mathcal{H}_{\varepsilon,\phi}:=\sum_{k=1}^{9}\mathcal{F}^{k}_{\varepsilon,\phi}$, thus the statement is true with $\mathcal{F}_{\varepsilon,\phi}:=\mathcal{H}_{\varepsilon,\phi}+\mathcal{G}_{\varepsilon,\phi}$.
\end{proof}

\section{Solving the auxiliary equation}
This Section will be devoted to the proofs of Propositions \ref{propaux_1} and \ref{propaux_2}. In both cases, we will first study the linear problem associated to our equation and then we will apply a contraction mapping principle.

\subsection{Solvabilty far away from $\Sigma_{\varepsilon}$: the linear problem}
We will prove the following Proposition.
\begin{proposition}
Let $0<\delta<\gamma<\sqrt{W^{''}(1)}$. Then, for any $\varepsilon>0$ small enough, for any $\phi\in B_{4}(\tau/4)$, and for any $f\in C^{0,\alpha}_{\gamma,s}(\R^{3})$, the equation  
\begin{eqnarray}
(-\Delta+\Gamma_{\varepsilon,\phi})^{2}V=f
\end{eqnarray}
admits a unique solution $V=\Psi_{\varepsilon,\phi}(f)$ in $C^{4,\alpha}_{\delta,s}(\R^{3})$ satisfying $||V||_{C^{4,\alpha}_{\delta}(\R^{3})}\leq c||f||_{C^{0,\alpha}_{\gamma}(\R^{3})}$, for some constant $c>0$ independent of $\varepsilon$ and $\phi$.
\label{inv4gamma}
\end{proposition}
\begin{remark}
The symmetries of the solution follow for free from the symmetries of the laplacian and of $\Gamma_{\varepsilon,\phi}$.  
In fact, if $f\in C^{0,\alpha}_{\gamma,s}(\R^{3})$, and $V$ is a solution to $(-\Delta+\Gamma_{\varepsilon,\phi})^{2}V=f$, then also $u_{T}(x):=u(Tx)$ is a solution, thus, by uniqueness, $u=u_{T}$. The same argument also shows that $u=u_{R}$, for any $R\in SO_{x_{3}}(3)$, hence $u\in C^{4,\alpha}_{\delta,s}(\R^{3})$.
\label{rem_symm_far}
\end{remark}
We stress that the assumption $\delta<\gamma$ is crucial. When we solve the equation $(-\Delta+\Gamma_{\varepsilon,\phi})^{2}u=f$ we lose some regularity, in the sense that the solution might decay slower than $f$. 

We split the proof into some lemmas and a proposition, with the aid of some remarks. First we reduce ourselves to consider a second order PDE, then, by a bootstrap argument, we will solve our forth order equation.
\begin{proposition}
Let $0<\delta<\gamma<\sqrt{W^{''}(1)}$. Then, for any $\varepsilon>0$ small enough, for any $\phi\in B_{4}(\tau/4)$, and for any $f\in C^{0,\alpha}_{\gamma}(\R^{3})$, the equation  
\begin{eqnarray}
-\Delta u+\Gamma_{\varepsilon,\phi}u=f
\end{eqnarray}
admits a unique solution $u=\tilde{\Psi}_{\varepsilon,\phi}(f)$ in $C^{2,\alpha}_{\delta}(\R^{3})$ satisfying $||u||_{C^{2,\alpha}_{\delta}(\R^{3})}\leq c||f||_{C^{0,\alpha}_{\gamma}(\R^{3})}$, for some constant $c>0$ independent of $\varepsilon$ and $\phi$.
\label{inv2gamma}
\end{proposition}
Before giving the proof, we state a technical Lemma.

\begin{lemma}
For any $1\leq p\leq\infty$, there exists a constant $C=C(p,\delta)>0$ such that, for any $u\in C^{0,\alpha}_{\delta}(\R^{3})$, we have
\begin{eqnarray}\notag
||u||_{L^{p}(\R^{3})}\leq C||u\varphi_{\delta}||_{\infty}.
\end{eqnarray}
\label{lemmaestp_infty}
\end{lemma}
\begin{proof}
The case $p=\infty$ is trivial, since $\varphi_{\delta}\geq 1$, so we can assume that $p<\infty$. We split the $L^{p}-$norm of $u$ into the sum of two terms, that is the integral over a ball of radius $R>1$ and its complement
\begin{eqnarray}\notag
\int_{\R^{N}}|u|^{p}dx=\int_{B_{R}}|u|^{p}dx+\int_{B_{R}^{c}}|u|^{p}dx.
\end{eqnarray}
The first term satisfies
\begin{eqnarray}\notag
\int_{B_{R}}|u|^{p}dx\leq|B_{R}| &\text{ }||u||_{\infty}^{p},
\end{eqnarray}
and the second one fulfills
\begin{eqnarray}\notag
\int_{B_{R}^{c}}|u|^{p}dx=\int_{B_{R}^{c}}(|u|\varphi_{\delta})^{p}\varphi_{-\delta}^{p}dx\leq
||u\varphi_{\delta}||_{\infty}^{p}\int_{B_{R}^{c}}\varphi_{-\delta}^{p}dx\leq
||u\varphi_{\delta}||_{\infty}^{p},
\end{eqnarray}
for some suitable $R=R(\delta,p)>1$, where we have set $\varphi_{-\delta}=1/\varphi_{\delta}$. 
\end{proof}

Now we are ready to prove Proposition \ref{inv2gamma}.
\begin{proof}
\textit{Step (i): existence, uniqueness and local H$\ddot{o}$lder regularity.}\\

Existence and uniqueness of the weak solution follow from the Riesz representation 
theorem. Since $f\in C^{0,\alpha}_{loc}(\R^{3})$, then $u\in C^{2,\alpha}_{loc}(\R^{3})$.\\ 

\textit{Step (ii): estimate for the $L^{\infty}$ norm.}\\

Now we will show that $u\varphi_{\delta}\in L^{\infty}(\R^{3})$ and 
\begin{eqnarray}
||u\varphi_{\delta}||_{\infty}\leq c||f\varphi_{\gamma}||_{\infty}.
\label{estnorminfty}
\end{eqnarray}
From now on, we will assume that $f$ is not identically $0$, and hence $||\tilde{u}_{\delta}||_{\infty}>0$, otherwise there is nothing to prove. As first we will prove that $\tilde{u}_{\delta}\to 0$ as $|x|\to\infty$. In order to do so, it is enough to show that $u\varphi_{\gamma}\in L^{\infty}(\R^{3})$. This will be done by using the function $e^{-\gamma|x|}$ as a barrier. More precisely, we fix $\rho>0$ and $|z|>\rho$. Then we fix $\sigma>0$ and $R>|z|$ so large that $u(x)<\sigma$ for $|x|\geq R$. Therefore $u$ fulfills
\begin{eqnarray}\notag
\begin{cases}
u<\max_{\partial B_{\rho}}u<\lambda e^{-\gamma\rho}<\lambda e^{-\gamma\rho}+\sigma &\text{for $|x|=\rho$}\\\notag
u<\sigma<\lambda e^{-\gamma R}+\sigma &\text{for $|x|=R$}\\\notag
(-\Delta+\Gamma_{\varepsilon,\phi})(u-(\lambda e^{-\gamma|x|}+\sigma))\leq \bigg(c-\lambda\frac{N-1}{R}\bigg)e^{-\gamma r}\leq 0 &\text{for $\rho<|x|<R$,}
\end{cases}
\end{eqnarray}
provided $\lambda\geq\lambda_{0}$, with $\lambda_{0}$ independent of $\sigma$. By the maximum principle we get that $u(z)<\lambda e^{-\gamma|z|}+\sigma$, for any $|z|\geq\rho$ and for any $\sigma>0$. Letting $\sigma\to 0$, we get that $u\varphi_{\gamma}\in L^{\infty}$.

Since $\tilde{u}_{\delta}\to 0$ as $|x|\to\infty$, the supremum is achieved at some point $y\in\R^{3}$, that is $||\tilde{u}_{\delta}||_{\infty}=|\tilde{u}_{\delta}(y)|$.
Now we consider two cases. If $|y|\leq 1$, we apply the elliptic estimates to control the $L^{\infty}$ norm of $u\varphi_{\delta}$ with $\tilde{f}_{\gamma}$, otherwise we use the equation for $\tilde{u}_{\delta}$. 

Let us consider the case $|y|\leq 1$. We observe that 
$\tilde{u}_{\delta}\equiv u$ in $B_{1}(0)$ and we apply elliptic estimates to get that
\begin{eqnarray}\notag
||u||_{L^{\infty}(B_{1}(0))}\leq C||u||_{W^{2,2}(B_{2}(0))}\leq C(||u||_{L^{2}(B_{2}(0))}+||f||_{L^{2}(B_{2}(0))}).
\end{eqnarray}
Now we multiply the equation $(-\Delta+\Gamma_{\varepsilon,\phi})u=f$ by $u$ and integrate by parts to obtain that 
\begin{eqnarray}\notag
||u||_{L^{2}(B_{2}(0))}\leq||u||_{H^{1}(\R^{3})}\leq C||f||_{L^{2}(\R^{3})}.
\end{eqnarray}
Moreover, by Lemma \ref{lemmaestp_infty} applied with $p=2$, we get
\begin{eqnarray}\notag
||f||_{L^{2}(\R^{3})}\leq c||\tilde{f}_{\gamma}||_{\infty}.
\end{eqnarray} 
Exactly in the same way, we can estimate the term $||f||_{L^{2}(B_{2})}$. We point out that all the constants are independent of $\varepsilon$ and $\phi$, because the potential is positive and bounded away from $0$ uniformly in $\varepsilon$ and $\phi$.

Now let us turn to the case in which the maximum point of $|\tilde{u}_{\delta}|$ is achieved outside $B_{1}$. The equation satisfied by $\tilde{u}_{\delta}$ is
\begin{eqnarray}
-\Delta\tilde{u}_{\delta}+\Gamma_{\varepsilon,\phi}\tilde{u}_{\delta}=\tilde{f}_{\delta}-2
<\nabla u,\nabla\varphi_{\delta}>-u\Delta
\varphi_{\delta}=\\\notag
\tilde{f}_{\delta}-2\varphi_{-\delta}<\nabla\tilde{u}_{\delta},\nabla\varphi_{\delta}>
-\tilde{u}_{\delta}
(2<\nabla\varphi_{\delta},\nabla\varphi_{-\delta}>-\varphi_{-\delta}\Delta
\varphi_{\delta}).
\end{eqnarray}
If $\tilde{u}_{\delta}(y)>0$, then $\tilde{u}_{\delta}$ has a global maximum point at $y$, and a computation shows that
\begin{eqnarray}\notag
0<\gamma^{2}\tilde{u}_{\delta}(y)<\Gamma_{\varepsilon,\phi}(y)\tilde{u}_{\delta}(y)\leq-\Delta\tilde{u}_{\delta}(y)+\gamma^{2}\tilde{u}_{\delta}(y)\leq
\tilde{f}_{\delta}(y)+\delta^{2}\tilde{u}(y),
\end{eqnarray}
the estimate being uniform in $\varepsilon$ and in $\phi$, hence
\begin{eqnarray}\notag
\tilde{u}_{\delta}(y)\leq\frac{1}{\gamma^{2}-\delta^{2}}\tilde{f}_{\delta}(y)\leq c\tilde{f}_{\gamma}(y),
\end{eqnarray}
so we have (\ref{estnorminfty}). If $\tilde{u}_{\delta}(y)<0$, then $\tilde{u}_{\delta}$ has a global minimum at $y$, and the conclusion follows from a similar argument.\\

\textit{Step (iii): estimates for higher order derivatives.}

Now we will show that
\begin{eqnarray}
||\tilde{u}_{\delta}||_{C^{2,\alpha}(\R^{3})}\leq c||\tilde{f}_{\gamma}||_{C^{0,\alpha}(\R^{3})},
\end{eqnarray}
for some constant $c>0$. In order to do so, we observe that, by elliptic estimates (see \cite{GT}), we have that, for any $x\in\R^{3}$
\begin{eqnarray}\notag
||\tilde{u}_{\delta}||_{C^{2,\alpha}(B_{1}(x))}\leq(||\tilde{f}_{\delta}||
_{C^{0,\alpha}(B_{2}(x))}+||\tilde{u}_{\delta}||_{L^{\infty}(B_{2}(x))})\leq c||\tilde{f}_{\gamma}||_{C^{0,\alpha}}(\R^{3}),
\end{eqnarray}
the constants being independent of $\varepsilon$ and $\phi$.
\end{proof}

Now we can conclude the proof of Proposition \ref{inv4gamma}.
\begin{proof}
Given $f\in C^{0,\alpha}_{\varepsilon,\gamma}(\R^{3})$, we have to find $V\in C^{4,\alpha}_{\varepsilon,\delta}(\R^{3})$ fulfilling
\begin{eqnarray}
\begin{cases}
(-\Delta+\Gamma_{\varepsilon,\phi})^{2}V=f\\\notag
||V||_{C^{4,\alpha}_{\delta}(\R^{3})}\leq c||f||_{C^{0,\alpha}_{\gamma}(\R^{3})}.
\end{cases}
\end{eqnarray}
In order to do so, we use proposition \ref{inv2gamma} twice to find $w\in C^{2,\alpha}_{\varepsilon,\delta^{'}}(\R^{3})$ and $u\in C^{2,\alpha}_{\varepsilon,\delta}(\R^{3})$, with $0<\delta<\delta^{'}<\gamma$, such that
\begin{eqnarray}
\begin{cases}\notag
(-\Delta+\Gamma_{\varepsilon,\phi})u=f\\
(-\Delta+\Gamma_{\varepsilon,\phi})V=u,
\end{cases}
\end{eqnarray}
and 
\begin{eqnarray}
\begin{cases}
||u||_{C^{2,\alpha}_{\delta^{'}}(\R^{3})}\leq c||f||_{C^{0,\alpha}_{\gamma}(\R^{3})}\\\notag
||V||_{C^{2,\alpha}_{\delta}(\R^{3})}\leq c||u||_{C^{0,\alpha}_{\delta^{'}}(\R^{3})}.
\end{cases}
\end{eqnarray}
Now it remains to estimate the higher order derivatives of $u$. For this purpose, we differentiate the equation satisfied by $u$ and we get $(-\Delta+\Gamma_{\varepsilon,\phi})V_{j}=u_{j}-(\Gamma_{\varepsilon,\phi})_{j}V$, for $j=1,\dots,3$, hence, applying the regularity estimates for $(-\Delta+\Gamma_{\varepsilon,\phi})$, 
\begin{eqnarray}\notag
||V_{j}||_{C^{2,\alpha}_{\delta}(\R^{3})}\leq c(||u_{j}||_{C^{0,\alpha}_{\delta^{'}}(\R^{3})}+||f||_{C^{0,\alpha}_{\gamma}(\R^{3})})\\\notag
\leq c(||u_{j}\varphi_{\delta^{'}}||_{C^{0,\alpha}(\R^{3})}+||f||_{C^{0,\alpha}_{\gamma}(\R^{3})})\leq c(||u_{j}\varphi_{\delta^{'}}||_{C^{1}(\R^{3})}+||f||_{C^{0,\alpha}_{\gamma}(\R^{3})}),
\end{eqnarray}
hence
\begin{eqnarray}\notag
||\nabla^{3}(V\varphi_{\delta})||_{\infty}\leq c(||u||_{C^{2,\alpha}_{\delta^{'}}(\R^{3})}+||f||_{C^{0,\alpha}_{\gamma}(\R^{3})})\leq c||f||_{C^{0,\alpha}_{\gamma}(\R^{3})}.
\end{eqnarray} 
Similarly, differentiating the equation once again, we see that 
\begin{eqnarray}\notag
(-\Delta+\Gamma_{\varepsilon,\phi})V_{ij}=u_{ij}-(\Gamma_{\varepsilon,\phi})_{i}V_{j}-(\Gamma_{\varepsilon,\phi})_{j}V_{i}
-(\Gamma_{\varepsilon,\phi})_{ij}V,
\end{eqnarray}
for $i,j=1,\dots,3$, so in particular
\begin{eqnarray}\notag
||V_{ij}||_{C^{2,\alpha}_{\delta}(\R^{3})}\leq c||u_{ij}||_{C^{0,\alpha}_{\delta^{'}}(\R^{3})}=
c(||u_{ij}\varphi_{\delta^{'}}||_{C^{0,\alpha}(\R^{3})}+||f||_{C^{0,\alpha}_{\gamma}(\R^{3})})\\\notag
\leq c(||u_{ij}\varphi_{\delta^{'}}||_{C^{1}(\R^{3})}+||f||_{C^{0,\alpha}_{\gamma}(\R^{3})}),
\end{eqnarray}
therefore
\begin{eqnarray}\notag
||\nabla^{4}(V\varphi_{\delta})||_{\infty}+[V\varphi_{\delta}]_{0,\alpha}\leq c(||u||_{C^{2,\alpha}_{\delta^{'}}(\R^{3})}+||f||_{C^{0,\alpha}_{\gamma}(\R^{3})})\leq c||f||_{C^{0,\alpha}_{\gamma}(\R^{3})},
\end{eqnarray}
all the constants being independent of $\varepsilon$ and $\phi$.
\end{proof}

\subsection{The proof of Proposition \ref{propaux_1}: solving equation (\ref{eq_aux1}) by a fixed point argument}
Equation (\ref{eq_aux1}) is equivalent to the fixed point problem
\begin{eqnarray}\notag
V=T_{1}(V):=\Psi_{\varepsilon,\phi}\bigg\{(1-\chi_{2})F(\tilde{v}_{\varepsilon,\phi})+(1-\chi_{1})Q_{\varepsilon,\phi}(\chi_{2}U+V)
+\text{N}_{\varepsilon,\phi}(U)+\text{P}_{\varepsilon,\phi}(V)\bigg\},
\end{eqnarray}
that we will solve by showing that $T_{1}$ is a contraction on the ball
\begin{eqnarray}\notag
\Lambda_{1}:=\{V\in C^{4,\alpha}_{\delta,s}(\R^{3}):||V||_{C^{4,\alpha}_{\delta}(\R^{3})}\leq C_{1}e^{-a/\varepsilon}\},
\end{eqnarray}
provided the constant $C_{1}$ is large enough. In fact, by the exponential decay of $U$ far from $\Sigma_{\varepsilon}$, we get that
\begin{eqnarray}\notag
||\text{N}_{\varepsilon,\phi}(U)||_{C^{4,\alpha}_{\delta}(\R^{3})}\leq\tilde{c}e^{-a/\varepsilon},
\end{eqnarray}
for some constants $a,\tilde{c}>0$ depending on $W,\tau,\delta$ but not of $\varepsilon$ and $\phi$. By (\ref{expdecdx}) and (\ref{expdecsx}), the same is true for $(1-\chi_{2})F(\tilde{v}_{\varepsilon,\phi})$. Moreover, by (\ref{first_time}), (\ref{expdecdx}) and (\ref{expdecsx}), 
\begin{eqnarray}\notag
||\text{P}_{\varepsilon,\phi}(V)||_{C^{4,\alpha}_{\delta}(\R^{3})}\leq ce^{-a/\varepsilon}||V||_{C^{4,\alpha}_{\delta}(\R^{3})}\leq ce^{-2a/\varepsilon},
\end{eqnarray}
with $c>0$ depending on $W,\tau,\delta$ but not of $\varepsilon$ and $\phi$. Moreover, using that
\begin{eqnarray}\notag
||(1-\chi_{1})V||_{C^{4,\alpha}_{\delta}(\R^{3})}\leq c||V||_{C^{4,\alpha}_{\delta}(\R^{3})}
\end{eqnarray}
and
\begin{eqnarray}\notag
||(1-\chi_{1})\chi_{2}U||_{C^{4,\alpha}_{\delta}(\R^{3})}\leq ce^{-a/\varepsilon},
\end{eqnarray}
where $(1-\chi_{1})\chi_{2}U$ is understood to be $0$ outside the support of $\chi_{2}$, and the definition of $Q_{\varepsilon,\phi}$ (see (\ref{quadratic_w})), we get
\begin{eqnarray}\notag
||(1-\chi_{1})Q_{\varepsilon,\phi}(\chi_{2}U+V)||_{C^{4,\alpha}_{\delta}(\R^{3})}\leq ce^{-2a/\varepsilon}.
\end{eqnarray}
Up to now, we have just proved that $T_{1}$ maps $\Lambda_{1}$ in itself. In order to show that it is actually a contraction, we need to estimate its Lipschitz constant. The only terms depending on $V$ are $P_{\varepsilon,\phi}$, that fulfills
\begin{eqnarray}\notag
||P_{\varepsilon,\phi}(V_{1})-P_{\varepsilon,\phi}(V_{2})||_{C^{4,\alpha}_{\delta}(\R^{3})}\leq c\varepsilon||V_{1}-V_{2}||_{C^{4,\alpha}_{\delta}(\R^{3})}
\end{eqnarray}
for some constant $c>0$ independent of $\varepsilon$ and $\phi$, and $(1-\chi_{1})Q_{\varepsilon,\phi}(\chi_{2}U+V)$, that fulfills
\begin{eqnarray}\notag
||(1-\chi_{1})(Q_{\varepsilon,\phi}(\chi_{2}U+V)-Q_{\varepsilon,\phi}(\chi_{2}U+V))||_{C^{4,\alpha}_{\delta}(\R^{3})}\leq ce^{-a/\varepsilon}||V_{1}-V_{2}||_{C^{4,\alpha}_{\delta}(\R^{3})}.
\end{eqnarray}

\textit{Lipschitz dependence on $U$ and $\phi$.}\\

Given $\phi\in B_{4}(\tau/4)$ and $U_{1},U_{2}\in C^{4,\alpha}(\Sigma_{\varepsilon}\times\R)$, the difference between the solutions $V_{\varepsilon,\phi, U_{1}}$ and $V_{\varepsilon,\phi, U_{1}}$ fulfills
\begin{eqnarray}\notag
(-\Delta+\Gamma_{\varepsilon,\phi})^{2}(V_{\varepsilon,\phi, U_{1}}-V_{\varepsilon,\phi, U_{2}})=(1-\chi_{1})(Q_{\varepsilon,\phi}(\chi_{2}U_{2}+V_{\varepsilon,\phi,U_{2}})-
Q_{\varepsilon,\phi}(\chi_{2}U_{1}+V_{\varepsilon,\phi,U_{1}}))\\\notag
+\text{N}_{\varepsilon,\phi}(U_{2})-\text{N}_{\varepsilon,\phi}(U_{1})+\text{P}_{\varepsilon,\phi}(V_{\varepsilon,\phi,U_{2}})
-\text{P}_{\varepsilon,\phi}(V_{\varepsilon,\phi,U_{1}}).
\label{lipU}
\end{eqnarray}
By (\ref{defN}), the terms involving $\text{N}_{\varepsilon,\phi}$ satisfy
\begin{eqnarray}\notag
||\text{N}_{\varepsilon,\phi}(U_{1})-\text{N}_{\varepsilon,\phi}(U_{2})||_{C^{0,\alpha}_{\gamma}(\R^{3})}\leq ce^{-a/\varepsilon}||U_{2}-U_{1}||_{C^{4,\alpha}_{\delta}(\Sigma_{\varepsilon}\times\R)}.
\end{eqnarray}
By (\ref{defP}), the terms involving $\text{N}_{\varepsilon,\phi}$ can be estimated with the difference between the solutions, that is
\begin{eqnarray}
||\text{P}_{\varepsilon,\phi}(V_{\varepsilon,\phi,U_{1}})-\text{P}_{\varepsilon,\phi}(V_{\varepsilon,\phi,U_{2}})||_{C^{0,\alpha}_{\gamma}(\R^{3})}\leq ce^{-a/\varepsilon}||V_{\varepsilon,\phi,U_{1}}-V_{\varepsilon,\phi,U_{2}}||_{C^{4,\alpha}_{\delta}(\R^{3})},
\end{eqnarray}
and 
\begin{eqnarray}\notag
||(1-\chi_{1})(Q_{\varepsilon,\phi}(\chi_{2}U_{1}+V_{\varepsilon,\phi,U_{1}})-Q_{\varepsilon,\phi}(\chi_{2}U_{2}+V_{\varepsilon,\phi,U_{2}}))||
_{C^{0,\alpha}_{\gamma}(\R^{3})}\leq\\\notag
ce^{-a/\varepsilon}(||V_{\varepsilon,\phi,U_{1}}-V_{\varepsilon,\phi,U_{2}}||_{C^{4,\alpha}_{\delta}(\R^{3})}+
||U_{1}-U_{2}||_{C^{4,\alpha}_{\delta}(\Sigma_{\varepsilon}\times\R)}).
\end{eqnarray}
Therefore, applying $\Psi_{\varepsilon,\phi}$ to the right-hand side of (\ref{lipU}), we obtain
\begin{eqnarray}\notag
||V_{\varepsilon,\phi,U_{1}}-V_{\varepsilon,\phi,U_{2}}||_{C^{4,\alpha}_{\delta}(\R^{3})}\leq\\\notag
ce^{-a/\varepsilon}(||V_{\varepsilon,\phi,U_{1}}-V_{\varepsilon,\phi,U_{2}}||_{C^{4,\alpha}_{\delta}(\R^{3})}+
||U_{1}-U_{2}||_{C^{4,\alpha}_{\delta}(\Sigma_{\varepsilon}\times\R)}),
\end{eqnarray}
thus, reabsorbing the norm of the difference between the solutions,
\begin{eqnarray}\notag
\frac{1}{2}||V_{\varepsilon,\phi,U_{1}}-V_{\varepsilon,\phi,U_{2}}||_{C^{4,\alpha}_{\delta}(\R^{3})}
\leq(1-ce^{-a/\varepsilon})||V_{\varepsilon,\phi,U_{1}}-V_{\varepsilon,\phi,U_{2}}||_{C^{4,\alpha}_{\delta}(\R^{3})}\\\notag
\leq ce^{-a/\varepsilon}||U_{1}-U_{2}||_{C^{4,\alpha}_{\delta}(\Sigma_{\varepsilon}\times\R)}.
\end{eqnarray}
The Lipschitz dependence on $\phi$ can be treated with a similar argument. It is worth to point out that also the potential $\Gamma_{\varepsilon,\phi}$ depends on $\phi$, through the approximate solution and the cutoff function. However, this dependence is mild enough for our purposes, in fact
the difference of the potentials $\Gamma_{\varepsilon,\phi_{1}}-\Gamma_{\varepsilon,\phi_{2}}$ is exponentially small in $\varepsilon$.

\subsection{Invertibility in a neighbourhood of $\Sigma_{\varepsilon}$: the linear problem}
Now we look for a solution to equation (\ref{eq_aux2}) respecting the symmetries of the Torus. First we study the linear operator $\mathcal{L}_{\varepsilon}^{2}$.
\begin{proposition}
Let $0<\delta<\sqrt{W^{''}(\pm 1)}$ and $\phi\in B_{4}(1/4)$. For any $f\in\mathcal{E}^{0,\alpha}_{\delta}(\Sigma_{\varepsilon}\times\R)$, there exists a unique solution $U=G_{\varepsilon}(f)$ in $\mathcal{E}^{4,\alpha}_{\delta}(\Sigma_{\varepsilon}\times\R)$ to $\mathcal{L}_{\varepsilon}^{2}U=f$ such that
\begin{eqnarray}\notag
||U||_{C^{4,\alpha}_{\delta}(\Sigma_{\varepsilon}\times\R)}\leq C||f||_{C^{0,\alpha}_{\delta}(\Sigma_{\varepsilon}\times\R)},
\end{eqnarray}
for some constant $C>0$ which is independent of $\varepsilon$. 
\label{propflat}
\end{proposition}
If $f$ respects the symmetries of the Torus, then also the solution $U=G_{\varepsilon}f$ does. In other words, $G_{\varepsilon}$ maps $\mathcal{E}^{4,\alpha}_{\delta,s}(\Sigma_{\varepsilon}\times\R)$ into $\mathcal{E}^{0,\alpha}_{\delta,s}(\Sigma_{\varepsilon}\times\R)$. This fact follows from uniqueness. 

It is useful to see that we can control the odd part of the solution with the odd part (in $t$) of $f$ and the same is true for the even parts.
\begin{lemma}
Let $0<\delta<\sqrt{W(1)}$ and $f\in C^{0,\alpha}_{\delta,s}(\Sigma_{\varepsilon}\times\R)$. Let $U\in C^{4,\alpha}_{\delta}(\Sigma_{\varepsilon}\times\R)$ be the solution to $\mathcal{L}_{\varepsilon}^{2}U=f$. Then
\begin{eqnarray}
\begin{cases}
||U_{o}||_{C^{4,\alpha}_{\delta}(\Sigma_{\varepsilon}\times\R)}\leq c||f_{o}||_{C^{0,\alpha}_{\delta}(\Sigma_{\varepsilon}\times\R)}\\\notag
||U_{e}||_{C^{4,\alpha}_{\delta}(\Sigma_{\varepsilon}\times\R)}\leq c||f_{e}||_{C^{0,\alpha}_{\delta}(\Sigma_{\varepsilon}\times\R)},
\end{cases}
\end{eqnarray}
where $c$ is the constant found in Proposition \ref{propflat}.
\label{Lemma_oe}
\end{lemma}
\begin{proof}
We set, for any $(y,t)\in\Sigma_{\varepsilon}\times\R$, $\tilde{U}(y,t):=U(y,-t)$ and $\tilde{f}(y,t):=f(y,-t)$. Using that $W^{''}$ is even and $v_{\star}$ is odd, we can see that $\mathcal{L}^{2}_{\varepsilon}\tilde{U}=\tilde{f}$. Therefore, subtracting and multiplying by $1/2$, we get
\begin{eqnarray}\notag
\mathcal{L}^{2}_{\varepsilon}\bigg(\frac{U(y,t)-\tilde{U}(y,t)}{2}\bigg)=\frac{f(y,t)-\tilde{f}(y,t)}{2},
\end{eqnarray}
that is $\mathcal{L}^{2}_{\varepsilon}U_{o}(y,t)=f_{o}$. In addition,
\begin{eqnarray}\notag
\int_{-\infty}^{\infty}U_{o}(y,t)v_{\star}^{'}(t)dt=\int_{-\infty}^{\infty}f_{o}(y,t)v_{\star}^{'}(t)dt=0,
\end{eqnarray}
for any $y\in\Sigma_{\varepsilon}$, hence $U_{o}=G_{\varepsilon}(f_{o})$, so in particular the first estimate holds true. The second one can be proved by a similar argument.
\end{proof}
Now we prove Proposition \ref{propflat}, with the aid of some Lemmas and Remarks. 

First we consider the spectral decomposition of $\mathcal{L}_{\varepsilon}$. We will denote by $(\lambda_{j},\phi_{j})_{j\geq 0}$ the eingendata of $-\Delta_{\Sigma}$. We observe that $\lambda_{0}=0$, $\lambda_{j}\geq\lambda_{1}>0$, $\phi_{0}$ is constant and, without loss of generality, we can assume that $||\phi_{j}||_{L^{2}(\Sigma)}=1$ (see \cite{PR}). Similarly, we will denote by $\{\mu_{k}\}_{k\geq 0}$ the eigenvalues of $L_{\star}=-\partial_{tt}+W^{''}(v_{\star}(t))$. In \cite{M}, M\"{u}ller proved that $\mu_{0}=0$, and the corresponding eigenspace, that is the Kernel, is generated by $v_{\star}^{'}(t)$, while $\mu_{k}\geq\mu_{1}>0$ (see also \cite{MW}). 
\begin{remark}
The eigenvalues of $\mathcal{L}_{\varepsilon}$ are $\{\mu_{k}+\varepsilon^{2}\lambda_{j}\}_{j,k\geq 0}$, thus al non-zero eigenvalues are positive and bounded away from $0$, indeed $\mu_{k}+\varepsilon^{2}\lambda_{j}\geq\varepsilon^{2}\lambda_{1}>0$.
\label{rem_eigenvalues}
\end{remark}
\begin{lemma}
Let 
\begin{eqnarray}\notag
\mathcal{L}_{\varepsilon}:H^{1}(\Sigma_{\varepsilon}\times\R)\to H^{-1}(\Sigma_{\varepsilon}\times\R)
\end{eqnarray}
be defined  by the duality relation
\begin{eqnarray}\notag
\bigg<\mathcal{L}_{\varepsilon}U_{1},U_{2}\bigg>=\int_{\Sigma_{\varepsilon}\times\R}\bigg\{(\nabla_{\Sigma_{\varepsilon}}U_{1},\nabla
_{\Sigma_{\varepsilon}}U_{2})+\partial_{t}U_{1}\partial_{t}U_{2}+W^{''}(v_{\star}(t))U_{1}U_{2}\bigg\}d\sigma(y)dt,
\end{eqnarray}
for any $U_{1},U_{2}\in C^{k,\alpha}_{\delta}(\Sigma_{\varepsilon}\times\R)$. Then 
\begin{eqnarray}\notag
Ker(\mathcal{L}_{\varepsilon})=span(v_{\star}^{'}(t)).
\end{eqnarray}
\label{lemmaker}
\end{lemma}
\begin{proof}
It is possible to see that $(\lambda_{\varepsilon,j},\phi_{\varepsilon,j})_{j\geq 0}:=(\varepsilon^{2}\lambda_{j},\varepsilon^{2}\phi_{j}(\varepsilon y))_{j\geq 0}$ are eigendata of $\Sigma_{\varepsilon}$ and $\phi_{\varepsilon,j}$ are orthonormal in $L^{2}(\Sigma_{\varepsilon})$. Any function $w\in H^{1}(\Sigma_{\varepsilon}\times\R)$ can be expanded in Fourier series as follows
\begin{eqnarray}\notag
U(y,t)=\sum_{j\geq 0}U_{j}(t)\phi_{\varepsilon,j}(y)
\end{eqnarray}
where
\begin{eqnarray}\notag
U_{j}(t)=\int_{\Sigma_{\varepsilon}}U(y,t)\phi_{\varepsilon,j}(y)d\sigma(y).
\end{eqnarray}
If $\mathcal{L}_{\varepsilon}w=0$, applying the operator to each term in the series, we get 
\begin{eqnarray}\notag
-\partial_{tt}U_{j}(t)+\lambda_{\varepsilon,j}U_{j}(t)+
W^{''}\big(v_{\star}(t))U_{j}(t)=0
\end{eqnarray}
for any $j\geq 0$, so $U_{0}(t)=cv_{\star}^{'}(t)$ and $w_{j}=0$ for $j\geq 1$. 
\end{proof}
Let 
\begin{eqnarray}\notag
\mathcal{O}:=\bigg\{U\in H^{1}(\Sigma_{\varepsilon}\times\R):\int_{\Sigma_{\varepsilon}\times\R}U(y,t)v_{\star}^{'}(t)d\sigma(y)dt=0\bigg\}.
\end{eqnarray}
be the orthogonal to $v_{\star}^{'}(t)$ in $H^{1}(\Sigma_{\varepsilon}\times\R)$.
\begin{lemma}
For any $f\in L^{2}(\Sigma_{\varepsilon}\times\R)$ satifying
\begin{eqnarray}\notag
\int_{-\infty}^{\infty} f(y,t)v_{\star}^{'}(t)dt=0 &\text{for any $y\in\Sigma_{\varepsilon}$},
\label{ortfally}
\end{eqnarray}
there exists a unique $U\in H^{1}(\Sigma_{\varepsilon}\times\R)$ such that 
\begin{eqnarray}\notag
\begin{cases}
\mathcal{L}_{\varepsilon}U=f\\\notag
\int_{-\infty}^{\infty}U(y,t)v_{\star}^{'}(t)dt=0 &\text{for any $y\in\Sigma_{\varepsilon}$.}
\end{cases}
\end{eqnarray}
\end{lemma}
\begin{proof}
At first we observe that 
\begin{eqnarray}
||U||=\int_{\Sigma_{\varepsilon}\times\R}|\nabla_{\Sigma_{\varepsilon}}U(y,z)|^{2}+(\partial_{tt}U(y,t))^{2}+W^{''}(v_{\star}^{'}(z))U^{2}(y,z)d\sigma(y)dt
\end{eqnarray}
is an equivalent norm on $\mathcal{O}$, that is, for any $U\in X$, we have
\begin{eqnarray}\notag
c_{\varepsilon,1}||U||_{H^{1}(\Sigma_{\varepsilon}\times\R)}\leq||U||\leq c_{\varepsilon,2}||U||_{H^{1}(\Sigma_{\varepsilon}\times\R)},
\end{eqnarray}
for some constants $c_{\varepsilon,1},c_{\varepsilon,2}>0$. In fact, by the spectral decomposition of $\mathcal{L}_{\varepsilon}$, (see Remark \ref{rem_eigenvalues}),
\begin{eqnarray}\notag
\int_{\Sigma_{\varepsilon}\times\R}\mathcal{L}_{\varepsilon}UUd\sigma(y)dt\geq\varepsilon^{2}\lambda_{1}
\int_{\Sigma_{\varepsilon}\times\R}U^{2}d\sigma(y)dt.
\end{eqnarray} 
Since $W^{''}(v_{\star}(t))$ is bounded, a pointwise estimate yields that
\begin{eqnarray}
\int_{\Sigma_{\varepsilon}\times\R}\mathcal{L}_{\varepsilon}UUd\sigma(y)dt\geq\int_{\Sigma_{\varepsilon}\times\R}
|\nabla_{\Sigma_{\varepsilon}}U|^{2}+(\partial_{tt}U)^{2}d\sigma(y)dt-c\int_{\Sigma_{\varepsilon}\times\R}U^{2}d\sigma(y)dt,
\end{eqnarray}
for some constant $c>0$. Now we point out that, for any $0<\lambda<1$, we have
\begin{eqnarray}\notag
\int_{\Sigma_{\varepsilon}\times\R}\mathcal{L}_{\varepsilon}UUd\sigma(y)dt=\lambda\int_{\Sigma_{\varepsilon}\times\R}
\mathcal{L}_{\varepsilon}UUd\sigma(y)dt+(1-\lambda)\int_{\Sigma_{\varepsilon}\times\R}\mathcal{L}_{\varepsilon}
UUd\sigma(y)dt\geq\\\notag
\lambda\bigg(\int_{\Sigma_{\varepsilon}\times\R}|\nabla_{\Sigma_{\varepsilon}}U|^{2}+(\partial_{tt}U)^{2}
d\sigma(y)dt-c\int_{\Sigma_{\varepsilon}\times\R}U^{2}d\sigma(y)dt\bigg)+
(1-\lambda)\varepsilon^{2}\lambda_{1}\int_{\Sigma_{\varepsilon}\times\R}U^{2}d\sigma(y)dt,
\end{eqnarray}
so, in order to prove the lower bound, it is enough to choose $\lambda<\varepsilon^{2}\lambda_{1}/(c+\varepsilon^{2}\lambda_{1})$. As a consequence, by the Riesz representation theorem, for any $f\in L^{2}(\Sigma_{\varepsilon}\times\R)$ such that 
\begin{eqnarray}
\int_{\Sigma_{\varepsilon}\times\R}f(y,t)v_{\star}^{'}(t)d\sigma(y)dt=0,
\label{condfort}
\end{eqnarray}
the equation $\mathcal{L}_{\varepsilon}U=f$ admits a unique solution $U\in\mathcal{O}$. We observe that orthogonality condition (\ref{condfort}) is necessary for solvability, since 
\begin{eqnarray}\notag
\int_{\Sigma_{\varepsilon}\times\R}f(y,t)v_{\star}^{'}(t)d\sigma(y)dt=\int_{\Sigma_{\varepsilon}\times\R}
\mathcal{L}_{\varepsilon}U(y,t)v_{\star}^{'}(t)d\sigma(y)dt=\\\notag
\int_{\Sigma_{\varepsilon}\times\R}
U(y,t)\mathcal{L}_{\varepsilon}v_{\star}^{'}(t)d\sigma(y)dt=0.
\end{eqnarray}
If in particular $f$ satisfies (\ref{ortfally}), then, by proposition $8,4$ of \cite{PR}, also $w$ satisfies (\ref{ortfally}).
\end{proof}
Now we are ready to conclude the proof of Proposition \ref{propflat}.
\begin{proof}
There are two more steps. As first we need some regularity theory to estimate the $C^{2,\alpha}_{\delta}(\Sigma_{\varepsilon}\times\R)$ norm of the solution $U$ if $f\in\mathcal{E}^{0,\alpha}_{\delta}(\Sigma_{\varepsilon}\times\R)$, then we have to iterate the estimates to deal with the operator $\mathcal{L}_{\varepsilon}^{2}$. For the first step, see Proposition $8,3$ of \cite{PR}. As regards the second one, we argue as follows.

If $f\in\mathcal{E}^{0,\alpha}_{\delta}(\Sigma_{\varepsilon}\times\R)$, the above discussion yields that we can find  $\tilde{U}\in\mathcal{E}^{2,\alpha}_{\delta}(\Sigma_{\varepsilon}\times\R)$ such that 
\begin{eqnarray}
\begin{cases}
\mathcal{L}_{\varepsilon}\tilde{U}=f\\\notag
||\tilde{U}||_{C^{2,\alpha}_{\delta}(\Sigma_{\varepsilon}\times\R)}\leq C||f||_{C^{0,\alpha}_{\delta}(\Sigma_{\varepsilon}\times\R)},
\end{cases}
\label{estPR}
\end{eqnarray}
for some constant $C>0$ independent of $\varepsilon$. Now, by the same argument, we can find $U\in\mathcal{E}^{2,\alpha}_{\delta}(\Sigma_{\varepsilon}\times\R)$ satisfying
\begin{eqnarray}
\begin{cases}
\mathcal{L}_{\varepsilon}U=\tilde{U}\\\notag
||U||_{C^{2,\alpha}_{\delta}(\Sigma_{\varepsilon}\times\R)}\leq C||\tilde{U}||_{C^{0,\alpha}_{\delta}(\Sigma_{\varepsilon}\times\R)}\leq
C||f||_{C^{0,\alpha}_{\delta}(\Sigma_{\varepsilon}\times\R)},
\end{cases}
\label{eqboot}
\end{eqnarray}
for some constant $C>0$ independent of $\varepsilon$. To conclude the proof, we have to show that $U\in C^{4,\alpha}_{\delta}(\Sigma_{\varepsilon}\times\R)$ and 
\begin{eqnarray}
||U||_{C^{4,\alpha}_{\delta}(\Sigma_{\varepsilon}\times\R)}\leq C||f||_{C^{0,\alpha}_{\delta}(\Sigma_{\varepsilon}\times\R)}.
\label{relnorm4}
\end{eqnarray}
In order to do so we apply a bootstrap argument. We differentiate (\ref{eqboot}) with respect to $y_{j}$ and we get 
\begin{eqnarray}\notag
\mathcal{L}_{\varepsilon}U_{j}=\tilde{U}_{j}.
\end{eqnarray}
By (\ref{estPR}), we get that $U_{j}\in C^{2,\alpha}_{\delta}(\Sigma_{\varepsilon}\times\R)$ and
\begin{eqnarray}\notag
||U_{j}||_{C^{2,\alpha}_{\delta}(\Sigma_{\varepsilon}\times\R)}\leq C||\tilde{U}_{j}||_{C^{0,\alpha}_{\delta}(\Sigma_{\varepsilon}\times\R)}\leq C||\tilde{U}||_{C^{2,\alpha}_{\delta}(\Sigma_{\varepsilon}\times\R)}\leq C||f||_{C^{0,\alpha}_{\delta}(\Sigma_{\varepsilon}\times\R)}.
\end{eqnarray}
In the same way, taking the derivative with respect to $t$, we get 
\begin{eqnarray}\notag
\mathcal{L}_{\varepsilon}U_{t}=\tilde{U}_{t}-\frac{1}{\varepsilon}W^{'''}
(v_{\star}(t))v_{\star}^{'}(t)U.
\end{eqnarray}
Exactly as before, we have
\begin{eqnarray}\notag
||U_{t}||_{C^{2,\alpha}_{\delta}(\Sigma_{\varepsilon}\times\R)}\leq C(||\tilde{U}_{t}||_{C^{0,\alpha}_{\delta}(\Sigma_{\varepsilon}\times\R)}+||W^{'''}
(v_{\star}(z))v_{\star}^{'}(t)U||
_{C^{0,\alpha}_{\delta}(\Sigma_{\varepsilon}\times\R)})\leq\\\notag
C(||\tilde{U}||_{C^{2,\alpha}_{\delta}(\Sigma_{\varepsilon}\times\R)}+
||U||_{C^{0,\alpha}_{\delta}(\Sigma_{\varepsilon}\times\R)})\leq C(||f||_{C^{0,\alpha}_{\delta}(\Sigma_{\varepsilon}\times\R)}+
||\tilde{U}||_{C^{2,\alpha}_{\delta}(\Sigma_{\varepsilon}\times\R)})\leq C||f||_{C^{0,\alpha}_{\delta}(\Sigma_{\varepsilon}\times\R)}.
\end{eqnarray}
Therefore we have
\begin{eqnarray}\notag
||\nabla^{3}(U\psi_{\delta})||_{\infty}\leq C||\nabla U||_{C^{2,\alpha}_{\delta}(\Sigma_{\varepsilon}\times\R)}\leq C||f||_{C^{0,\alpha}_{\varepsilon,\delta}(\Sigma\times\R)}.
\end{eqnarray}
Differentiating the equation once again, we get 
\begin{eqnarray}\notag
||\nabla^{4}(U\psi_{\delta})||_{\infty}+
[\nabla^{4}(U\psi_{\delta})]_{\alpha}\leq C||f||_{C^{0,\alpha}_{\delta}(\Sigma_{\varepsilon}\times\R)}.
\end{eqnarray}
In conclusion, we have (\ref{relnorm4}).
\end{proof}

\subsection{The proof of Proposition \ref{propaux_2}: solving equation (\ref{proj_prob}) by a fixed point argument}
Equation (\ref{proj_prob}) is equivalent to the fixed point problem
\begin{eqnarray}\notag
U=T_{2}(U):=G_{\varepsilon}\bigg\{-\chi_{4}F(\tilde{v}_{\varepsilon,\phi})-\text{T}(U,V_{\varepsilon,\phi,U},\phi)+p(y)v_{\star}^{'}(t)\bigg\}.
\end{eqnarray}
Once again, we will solve it by showing that $T_{2}$ is a contraction on the ball
\begin{eqnarray}\notag
\Lambda_{2}:=\{U\in\mathcal{E}^{4,\alpha}_{\delta,s}(\Sigma_{\varepsilon}\times\R):||U||_{C^{4,\alpha}_{\delta}(\Sigma_{\varepsilon}\times\R)}\leq C_{2}\varepsilon^{3}\},
\end{eqnarray}
provided $C_{2}>0$ is large enough. First we observe that, by definition of $p$, the right hand side is orthogonal to $v_{\star}^{'}(t)$ for any $y\in\Sigma_{\varepsilon}$, thus we can actually apply the operator $G_{\varepsilon}$. Moreover, if $U$ respects the symmetries of the Torus, then also the right-hand side does, thus, applying $G_{\varepsilon}$, we get once again something that respects these symmetries. Now we show that, if $||U||_{C^{4,\alpha}_{\delta}(\Sigma_{\varepsilon}\times\R)}\leq C_{2}\varepsilon^{2}$, then also $T_{2}(U)$ satisfies the same upper bound, for some large constant $C_{2}$.

We note that
\begin{eqnarray}\notag
||\chi_{4}F(\tilde{v}_{\varepsilon,\phi})||_{C^{0,\alpha}_{\delta}(\Sigma_{\varepsilon}\times\R)}\leq c\varepsilon^{3},
\end{eqnarray}
for some constant $\tilde{c}$ depending just on $W,\tau$ and the geometric quantities of $\Sigma$, and the same is true for $p(y)v_{\star}^{'}(t)$. The other terms are smaller, for instance, using (\ref{quadratic_w}) and the fact that $V$ is exponentially small,
\begin{eqnarray}\notag
||\chi_{1}Q_{\varepsilon,\phi}(U+V)||_{C^{0,\alpha}_{\delta}(\Sigma_{\varepsilon}\times\R)}\leq c\varepsilon^{6}.
\end{eqnarray}
Similarly, we can see that $||\text{M}_{\varepsilon,\phi}(V)||_{C^{0,\alpha}_{\delta}(\Sigma_{\varepsilon}\times\R)}\leq ce^{-a/\varepsilon}$. In addition, since all the coefficients of $\text{R}_{\varepsilon,\phi}$ are at least of order $\varepsilon$, we get that
\begin{eqnarray}\notag
||\chi_{4}\text{R}_{\varepsilon,\phi}(U)||_{C^{0,\alpha}_{\delta}(\Sigma_{\varepsilon}\times\R)}\leq c||U||_{C^{4,\alpha}_{\delta}(\Sigma_{\varepsilon}\times\R)}\leq c\varepsilon^{4}.
\end{eqnarray}
As regards the Lipschitz dependence on $U$, we observe that 
\begin{eqnarray}\notag
||\chi_{1}(Q_{\varepsilon,\phi}(U_{1}+V)-Q_{\varepsilon,\phi}(U_{2}+V))||_{C^{0,\alpha}_{\delta}(\Sigma_{\varepsilon}\times\R)}\leq c\varepsilon^{3}||U_{1}-U_{2}||_{C^{4,\alpha}_{\delta}(\Sigma_{\varepsilon}\times\R)}
\end{eqnarray}
and
\begin{eqnarray}\notag
||\chi_{4}(\text{R}_{\varepsilon,\phi}(U_{1})-\text{R}_{\varepsilon,\phi}(U_{2}))||_{C^{0,\alpha}_{\delta}(\Sigma_{\varepsilon}\times\R)}\leq c\varepsilon||U_{1}-U_{2}||_{C^{4,\alpha}_{\delta}(\Sigma_{\varepsilon}\times\R)}.
\end{eqnarray}
\textit{Estimate of the odd part of the solution $U_{\varepsilon,\phi}$.}\\

Up to now we have proved the existence of a solution $U_{\varepsilon,\phi}$ to equation (\ref{proj_prob}) satisfying $||U_{\varepsilon,\phi}||_{C^{4,\alpha}_{\delta}(\Sigma_{\varepsilon}\times\R)}\leq c\varepsilon^{3}$. However, we point out that the only terms of order $\varepsilon^{3}$ in the right-hand side come from $\chi_{4}F(\tilde{v}_{\varepsilon,\phi})$. In fact, as we observed above, $||\text{T}(U,V_{\varepsilon,\phi,U},\phi)||_{C^{0,\alpha}_{\delta}(\Sigma_{\varepsilon}\times\R)}\leq c\varepsilon^{4}$, so in particular the same is true for
\begin{eqnarray}\notag
\frac{1}{c_{\star}}\bigg(\int_{-\infty}^{\infty}\text{T}(U,V_{\varepsilon,\phi,U},\phi)v_{\star}^{'}(t)dt\bigg)v_{\star}^{'}(t).
\end{eqnarray}
Moreover, by Proposition \ref{solve_eq_bifo}, 
\begin{eqnarray}\notag
\int_{-\infty}^{\infty}\chi_{4}F(\tilde{v}_{\varepsilon,\phi})(y,t)v_{\star}^{'}(t)dt=\int_{-\infty}^{\infty}F(\tilde{v}_{\varepsilon,\phi})(y,t)v_{\star}^{'}(t)dt
+\int_{-\infty}^{\infty}(\chi_{4}-1)F(\tilde{v}_{\varepsilon,\phi})(y,t)v_{\star}^{'}(t)dt
\end{eqnarray}
is of order $\varepsilon^{4}$, since the second term is exponentially small. Going back to Section $5$, it is possible to see that the only terms of order $\varepsilon^{3}$ in $F(\tilde{v}_{\varepsilon,\phi})$ are even in $t$, thus the odd part of the right-hand side is of order $\varepsilon^{4}$, and therefore, by Lemma \ref{Lemma_oe}, the same is true for $U_{\varepsilon,\phi}$, namely $||U_{\varepsilon,\phi}||_{C^{4,\alpha}_{\delta}(\Sigma_{\varepsilon}\times\R)}\leq c\varepsilon^{4}$.\\

\textit{Lipschitz dependence on $\phi$.}\\

Let us fix $\phi_{1},\phi_{2}\in\Lambda_{2}$. To simplify the notation, we set, for $k=1,2$, $\tilde{V}_{k}:=\tilde{v}_{\varepsilon,\phi_{k}}$, $U_{k}:=U_{\varepsilon,\phi_{k}}$, $V_{k}:=V_{\varepsilon,\phi_{k},U_{k}}$ and so on. In this proof, $\varepsilon$ will always be small but fixed, and we will be interested in the dependence on $\phi$. 

First we note that
\begin{eqnarray}\notag
\chi_{4}(F(\tilde{v}_{1})-F(\tilde{v}_{2}))=F(\tilde{v}_{1})-F(\tilde{v}_{2})+(\chi_{4}-1)(F(\tilde{v}_{1})-F(\tilde{v}_{2}))
\end{eqnarray}
The first term satisfies
\begin{eqnarray}\notag
||F(\tilde{v}_{1})-F(\tilde{v}_{2})||_{C^{4,\alpha}_{\delta}(\Sigma_{\varepsilon}\times\R)}\leq c\varepsilon^{3}|\phi_{1}-\phi_{2}|_{C^{4,\alpha}(\Sigma)},
\end{eqnarray}
because, for instance, 
\begin{eqnarray}\notag
|\varepsilon^{2}(|\nabla_{\Sigma}\phi_{1}|^{2}-|\nabla_{\Sigma}\phi_{2}|^{2})v_{\star}^{(4)}|\leq c\varepsilon^{2}(|\nabla_{\Sigma}\phi_{1}|+|\nabla_{\Sigma}\phi_{2}|)|\phi_{1}-\phi_{2}|_{C^{4,\alpha}(\Sigma)}\leq c\varepsilon^{3}|\phi_{1}-\phi_{2}|_{C^{4,\alpha}(\Sigma)}.
\end{eqnarray}
The other terms are similar, or even easier to treat because there is already an $\varepsilon^{3}$ that multiplies everything (see section 
$5,1$). The Lipschitz constant of the second term is exponentially small in $\varepsilon$, namely 
\begin{eqnarray}\notag
||(\chi_{4}-1)(F(\tilde{v}_{1})-F(\tilde{v}_{2}))||_{C^{4,\alpha}_{\delta}(\Sigma_{\varepsilon}\times\R)}\leq ce^{-a/\varepsilon}|\phi_{1}-\phi_{2}|_{C^{4,\alpha}(\Sigma)}.
\end{eqnarray}
Using the Lipschitz dependence of $V$ on the data proved in Proposition \ref{propaux_1} and the definitions of $\text{M}_{\varepsilon,\phi}$, $Q_{\varepsilon,\phi}$ and $\text{R}_{\varepsilon,\phi}$, it is possible to see that
\begin{eqnarray}\notag
||\text{M}_{1}(V_{1})-\text{M}_{2}(V_{2})||_{C^{4,\alpha}_{\delta}(\Sigma_{\varepsilon}\times\R)}\leq ce^{-a/\varepsilon}(||U_{1}-U_{2}||_{C^{4,\alpha}_{\delta}(\Sigma_{\varepsilon}\times\R)}+|\phi_{1}-\phi_{2}|_{C^{4,\alpha}(\Sigma)}),\\\notag
||\chi_{4}(Q_{1}(U_{1}+V_{1})-Q_{2}(U_{2}+V_{2}))||_{C^{4,\alpha}_{\delta}(\Sigma_{\varepsilon}\times\R)}\leq c\varepsilon^{3}(||U_{1}-U_{2}||_{C^{4,\alpha}_{\delta}(\Sigma_{\varepsilon}\times\R)}+|\phi_{1}-\phi_{2}|_{C^{4,\alpha}(\Sigma)}),\\\notag
||\chi_{4}(\text{R}_{1}(U_{1})-\text{R}_{2}(U_{2}))||_{C^{4,\alpha}_{\delta}(\Sigma_{\varepsilon}\times\R)}\leq c\varepsilon||U_{1}-U_{2}||_{C^{4,\alpha}_{\delta}(\Sigma_{\varepsilon}\times\R)}+c\varepsilon^{4}|\phi_{1}-\phi_{2}|_{C^{4,\alpha}(\Sigma)}.
\end{eqnarray}
Now it remains to deal with $p(y)$, that also depends on $\varepsilon$ and $\phi$. We write, for any $y\in\Sigma_{\varepsilon}$ and $\phi\in B_{4}(\tau/4)$,
\begin{eqnarray}\notag
p(y)=\int_{-\infty}^{\infty}F(\tilde{v}_{\varepsilon,\phi})(y,t)v_{\star}^{'}(t)dt+\big\{p_{1}(\phi)(y)+p_{2}(\phi)(y)+p_{3}(\phi)(y)
+p_{4}(\phi)(y)\big\}v_{\star}^{'}(t),
\end{eqnarray}
where we have set
\begin{eqnarray}
p_{1}(\phi)(y):=\frac{1}{c_{\star}}\int_{-\infty}^{\infty}(1-\chi_{4})F(\tilde{v}_{\varepsilon,\phi})(y,t)v_{\star}^{'}(t)dt,\label{proj1}\\
p_{2}(\phi)(y):=\frac{1}{c_{\star}}\int_{-\infty}^{\infty}\chi_{1}Q_{\varepsilon,\phi}(U+V)(y,t)v_{\star}^{'}(t)dt,\label{proj2}\\
p_{3}(\phi)(y):=\frac{1}{c_{\star}}\int_{-\infty}^{\infty}\chi_{1}\text{M}_{\varepsilon,\phi}(V)(y,t)v_{\star}^{'}(t)dt\label{proj3}\\
p_{4}(\phi)(y):=\frac{1}{c_{\star}}\int_{-\infty}^{\infty}\chi_{4}\text{R}_{\varepsilon,\phi}(U_{\varepsilon,\phi})(y,t)v_{\star}^{'}(t)dt\label{proj_R}
\end{eqnarray}
and $U:=U_{\varepsilon,\phi}$, $V:=V_{\varepsilon,\phi,U}$. Since we want to deal with functions defined on $\Sigma$, we will set, for any $y\in\Sigma_{\varepsilon}$, $\tilde{p}_{i}(\phi)(\varepsilon y):=p_{i}(\phi)(y)$, for $i=1,\dots,4$. It follows from Proposition \ref{solve_eq_bifo} and that
\begin{eqnarray}
\bigg|\int_{-\infty}^{\infty}F(\tilde{v}_{\varepsilon,\phi_{1}})(y,t)v_{\star}^{'}(t)dt
-\int_{-\infty}^{\infty}F(\tilde{v}_{\varepsilon,\phi_{2}})(y,t)v_{\star}^{'}(t)dt\bigg|\leq c\varepsilon^{3}|\phi_{1}-\phi_{2}|_{C^{4,\alpha}(\Sigma)}.
\end{eqnarray}
In addition, by the previous discussion,
\begin{eqnarray}
\begin{cases}
|\tilde{p}_{1}(\phi)|_{C^{0,\alpha}(\Sigma)}\leq ce^{-a/\varepsilon}\\
|\tilde{p}_{1}(\phi_{1})-\tilde{p}_{1}(\phi_{2})|_{C^{0,\alpha}(\Sigma)}\leq ce^{-a/\varepsilon}|\phi_{1}-\phi_{2}|_{C^{4,\alpha}(\Sigma)}.
\end{cases}
\label{dec_proj1}
\end{eqnarray}
Furthermore, by the Lipschitz dependence of $V$ on the data, proved in Proposition \ref{propaux_1}, and by the fact that $||U||_{C^{4,\alpha}_{\delta}(\Sigma_{\varepsilon}\times\R)}\leq C_{2}\varepsilon^{3}$, we have
\begin{eqnarray}
\begin{cases}
|\tilde{p}_{2}(\phi)|_{C^{0,\alpha}(\Sigma)}\leq c\varepsilon^{6}\\
|\tilde{p}_{2}(\phi_{1})-\tilde{p}_{2}(\phi_{2})|_{C^{0,\alpha}(\Sigma)}\leq c\varepsilon^{3}(|\phi_{1}-\phi_{2}|_{C^{4,\alpha}(\Sigma)}+||U_{1}-U_{2}||_{C^{4,\alpha}_{\delta}(\Sigma_{\varepsilon}\times\R)}).
\end{cases}
\label{dec_proj2}
\end{eqnarray}
and, similarly
\begin{eqnarray}
\begin{cases}
|\tilde{p}_{3}(\phi)|_{C^{0,\alpha}(\Sigma)}\leq ce^{-a/\varepsilon}\\
|\tilde{p}_{3}(\phi_{1})-\tilde{p}_{3}(\phi_{2})|_{C^{0,\alpha}(\Sigma)}\leq ce^{-a/\varepsilon}(|\phi_{1}-\phi_{2}|_{C^{4,\alpha}(\Sigma)}+||U_{1}-U_{2}||_{C^{4,\alpha}_{\delta}(\Sigma_{\varepsilon}\times\R)}).
\end{cases}
\label{dec_proj3}
\end{eqnarray}
As regards $\tilde{p}_{4}$, we give a first, rough estimate that is enough to prove the Lipschitz dependence of $U$ on $\phi$. However, we will see later that this estimate is actually not enough to solve the bifurcation equation, thus we will improve it in Lemma \ref{lemmaR}, using the estimate of the odd part of $U$ (see section $7$). 
\begin{eqnarray}
\begin{cases}
|\tilde{p}_{4}(\phi)|_{C^{0,\alpha}(\Sigma)}\leq c\varepsilon^{4}\\
|\tilde{p}_{4}(\phi_{1})-\tilde{p}_{4}(\phi_{2})|_{C^{0,\alpha}(\Sigma)}\leq c\varepsilon||U_{1}-U_{2}||_{C^{4,\alpha}_{\delta}(\Sigma_{\varepsilon}\times\R)}.
\end{cases}
\label{dec_proj4}
\end{eqnarray}
In conclusion, the equation satisfied by the difference of the solutions $U_{1}-U_{2}$ is of the form
\begin{eqnarray}\notag
\mathcal{L}^{2}_{\varepsilon}(U_{1}-U_{2})=g(\phi_{1})(y,t)-g(\phi_{2})(y,t),
\end{eqnarray}
where $g(\phi_{i})$ and $U_{i}$ satisfy
\begin{eqnarray}\notag
\int_{-\infty}^{\infty}(g(\phi_{1})-g(\phi_{2}))(y,t)v_{\star}^{'}(t)dt=\int_{-\infty}^{\infty}(U_{1}-U_{2})(y,t)v_{\star}^{'}(t)dt=0,
\end{eqnarray}
thus, by Proposition \ref{propflat},
\begin{eqnarray}\notag
||U_{1}-U_{2}||_{C^{4,\alpha}_{\delta}(\Sigma_{\varepsilon}\times\R)}\leq c\varepsilon||U_{1}-U_{2}||_{C^{4,\alpha}_{\delta}(\Sigma_{\varepsilon}\times\R)}+c\varepsilon^{3}|\phi_{1}-\phi_{2}|_{C^{4,\alpha}(\Sigma)},
\end{eqnarray}
and hence, reabsorbing the first term of the right-hand side,
\begin{eqnarray}\notag
\frac{1}{2}||U_{1}-U_{2}||_{C^{4,\alpha}_{\delta}(\Sigma_{\varepsilon}\times\R)}\leq c\varepsilon^{3}|\phi_{1}-\phi_{2}|_{C^{4,\alpha}(\Sigma)}. 
\end{eqnarray}

\section{Solving the bifurcation equation} 
\subsection{The proof of Proposition \ref{prop_int_v}}
First let us fix some notation. For any $\phi\in C^{4,\alpha}(\Sigma)_{s}$ and $0<\varepsilon\leq 1$, 
$|\Sigma_{\varepsilon,\phi}|_{3}$ will be the volume of the interior of $\Sigma_{\varepsilon,\phi}$, that is its $3$-Lebesgue measure. Moreover, we set 
\begin{eqnarray}\notag
B_{1}:=\{x=Z_{\varepsilon}(y,t+\phi(\varepsilon y)):-5-\tau/2\varepsilon<t<0\}\\\notag
B_{2}:=\{x=Z_{\varepsilon}(y,z):0<t<5+\tau/2\varepsilon\},
\end{eqnarray}
$V_{i}$ will be the volume of $B_{i}$, for $i=1,2$, and $A:=\R^{3}\backslash B$. Now we note that
\begin{eqnarray}\notag
\int_{\R^{3}}(1-v_{\varepsilon,\phi}(x))dx=\int_{A}(1-v_{\varepsilon,\phi}(x))dx+\int_{B}(1-v_{\varepsilon,\phi}(x))dx
\end{eqnarray}
and
\begin{eqnarray}\notag
\int_{A}(1-v_{\varepsilon,\phi}(x))dx+\int_{B}1dx=2(|\Sigma_{\varepsilon,\phi}|_{3}-V_{1})+V_{1}+V_{2}=2|\Sigma_{\varepsilon\phi}|_{3}+V_{2}-V_{1}.
\end{eqnarray}
In the forthcoming integrals, we will use the natural change of variables induced on $V_{\tau/\varepsilon}$ by the parametrization $Y_{\varepsilon}(\text{y})=\varepsilon^{-1}Y(\varepsilon\text{y})$ (see (\ref{param_Sigma})). 
The absolute value of the Jacobian determinant is $\varepsilon^{2}\big\{(z+\varepsilon^{-1})^{2}\cos(\varepsilon\text{y}_{1})+(z+\varepsilon^{-1})\varepsilon^{-1}\sqrt{2}\big\}$, thus we can see that
\begin{eqnarray}
|\Sigma_{\varepsilon,\phi}|_{3}=2\pi\varepsilon^{-1}\int_{0}^{2\pi/\varepsilon}d\text{y}_{1}\int_{-1/\varepsilon-\phi(\varepsilon\text{y}_{1})}^{0}
\varepsilon^{2}\big\{(t+\phi(\varepsilon\text{y}_{1})+\varepsilon^{-1})^{2}\cos\text{y}_{1}\\\notag
+(t+\phi(\varepsilon\text{y}_{1})+\varepsilon^{-1})\varepsilon^{-1}\sqrt{2}\big\}dz=
\varepsilon^{-3}2\pi^{2}\sqrt{2}+\varepsilon^{-2}\int_{\Sigma}\phi(\zeta)d\sigma(\zeta)\\\notag
+2\pi\varepsilon^{-1}\int_{0}^{2\pi}\phi^{2}(\vartheta)(\cos\vartheta+\sqrt{2}/2)d\vartheta+
\frac{2\pi}{3}\int_{0}^{2\pi}\phi^{3}(\vartheta)\cos(\vartheta)d\vartheta,
\label{vol_Sigma}
\end{eqnarray}
since the surface integral over $\Sigma_{\varepsilon}$ of a function $\psi$ of the variable $\text{y}_{1}$ is given by
\begin{eqnarray}
\int_{\Sigma_{\varepsilon}}\psi(y)d\sigma(y)=2\pi\varepsilon^{-1}\int_{0}^{2\pi/\varepsilon}(\cos(\varepsilon\text{y}_{1})
+\sqrt{2})\psi(\text{y}_{1})d\text{y}_{1}.
\label{int_Sigma}
\end{eqnarray}
Similarly, we can show that
\begin{eqnarray}
V_{2}-V_{1}=\label{vol_tau}\\\notag
2\pi\varepsilon^{-1}\int_{0}^{2\pi/\varepsilon}d\text{y}_{1}\int_{0}^{6+\tau/2\varepsilon}
\varepsilon^{2}\big\{(t+\phi(\varepsilon\text{y}_{1})+\varepsilon^{-1})^{2}\cos(\varepsilon\text{y}_{1})\\\notag
+(t+\phi(\varepsilon\text{y}_{1})+\varepsilon^{-1})\varepsilon^{-1}\sqrt{2}\big\}dt\\\notag
-2\pi\varepsilon^{-1}\int_{0}^{2\pi/\varepsilon}d\text{y}_{1}\int_{-6-\tau/2\varepsilon}^{0}
\varepsilon^{2}\big\{(t+\phi(\varepsilon\text{y}_{1})+\varepsilon^{-1})^{2}\cos(\varepsilon\text{y}_{1})\\\notag
+(t+\phi(\varepsilon\text{y}_{1})+\varepsilon^{-1})\varepsilon^{-1}\sqrt{2}\big\}dt=\\\notag
2\pi\varepsilon^{-1}\int_{0}^{6+\tau/2\varepsilon}tdt\int_{0}^{2\pi}\big\{2\sqrt{2}+4\varepsilon\phi(\vartheta_{1})
\cos(\vartheta_{1})\big\}d\vartheta_{1}.
\end{eqnarray}
Observing that 
\begin{eqnarray}
v_{\varepsilon,\phi}(\varepsilon\text{y}_{1},t)=\tilde{v}_{\varepsilon,\phi}(\varepsilon\text{y}_{1},t)+(1-\chi_{5}(x))(\mathbb{H}(x)
-\tilde{v}_{\varepsilon,\phi}(\varepsilon\text{y}_{1},t))
\end{eqnarray}
we compute
\begin{eqnarray}
\int_{B}v_{\varepsilon,\phi}(x)dx\label{exp_term}\\\notag
=2\pi\varepsilon^{-1}\int_{0}^{2\pi/\varepsilon}d\text{y}_{1}\int_{-6-\tau/2\varepsilon}^{6+\tau/2\varepsilon}
\varepsilon^{2}\tilde{v}_{\varepsilon,\phi}(\varepsilon\text{y}_{1},t)\big\{(t+\phi(\varepsilon\text{y}_{1})
+\varepsilon^{-1})^{2}\cos(\varepsilon\text{y}_{1})\\\notag
+(t+\phi(\varepsilon\text{y}_{1})+\varepsilon^{-1})\varepsilon^{-1}\sqrt{2}\big\}dt\\\notag
+2\pi\varepsilon^{-1}\int_{0}^{2\pi/\varepsilon}d\text{y}_{1}\int_{-6-\tau/2\varepsilon}^{6+\tau/2\varepsilon}(1-\chi_{5})
(\mathbb{H}(x)-\tilde{v}_{\varepsilon,\phi}(\varepsilon\text{y}_{1},t))
\varepsilon^{2}\big\{(t+\phi(\varepsilon\text{y}_{1})+\varepsilon^{-1})^{2}\cos(\varepsilon\text{y}_{1})\\\notag
+(t+\phi(\varepsilon\text{y}_{1})+\varepsilon^{-1})\varepsilon^{-1}\sqrt{2}\big\}dt.
\end{eqnarray}
The second integral is exponentially decreasing in $\varepsilon$, and the same is true for its Lipschitz constant. As regards the second one, exploiting the symmetry of $v_{\star}$, $\eta$ and of the domain, we can see that
\begin{eqnarray}\notag
2\pi\varepsilon^{-1}\int_{0}^{2\pi/\varepsilon}d\text{y}_{1}\int_{-6-\tau/2\varepsilon}^{6+\tau/2\varepsilon}\varepsilon^{2}
\tilde{v}_{\varepsilon,\phi}(\varepsilon\text{y}_{1},t)\big\{(t+\phi(\varepsilon\text{y}_{1})+\varepsilon^{-1})^{2}\cos(\varepsilon\text{y}_{1})\\\notag
+(t+\phi(\varepsilon\text{y}_{1})+\varepsilon^{-1})\varepsilon^{-1}\sqrt{2}\big\}dt=\\\notag
2\pi\varepsilon^{-1}\int_{0}^{6+\tau/2\varepsilon}tv_{\star}(t)dt\int_{0}^{2\pi}\big\{4\varepsilon\phi(\vartheta_{1})
\cos(\vartheta)+2\sqrt{2}\big\}d\vartheta_{1}+G^{1}_{\varepsilon}(\phi).
\label{int_v_star}
\end{eqnarray}
with $G^{1}_{\varepsilon}$ satisfying (\ref{small_G}). Thus, taking the sum of (\ref{vol_Sigma}), (\ref{vol_tau}), (\ref{exp_term}) and (\ref{int_v_star}), 
\begin{eqnarray}
\int_{\R^{3}}(1-v_{\varepsilon,\phi}(x))dx=\varepsilon^{-3}4\pi^{2}\sqrt{2}+2\varepsilon^{-2}\int_{\Sigma}\phi(\zeta)d\zeta\\\notag
+2\pi\varepsilon^{-1}\int_{0}^{6+\tau/2\varepsilon}t(1-v_{\star}(t))dt\int_{0}^{2\pi}\big\{4\varepsilon\phi(\vartheta_{1})
\cos(\vartheta_{1})+2\sqrt{2}\big\}d\vartheta_{1}+G^{2}_{\varepsilon}(\phi),
\end{eqnarray}
with $G^{2}_{\varepsilon}$ satisfying (\ref{small_G}). 

It remains to deal with the term involving $w_{\varepsilon,\phi}$.
\begin{eqnarray}\notag
\int_{\R^{3}}|w_{\varepsilon,\phi}(x)|dx=\int_{\R^{3}}|w_{\varepsilon,\phi}(x)|\varphi_{\delta}(x)\varphi_{-\delta}(x)dx\\\notag
\leq c||w_{\varepsilon,\phi}||_{C^{4,\alpha}_{\delta}(\R^{3})}\int_{\R^{3}}\varphi_{-\delta}(x)dx\leq c\varepsilon^{3}
\end{eqnarray}
and, by Propositions \ref{propaux_1} and \ref{propaux_2},
\begin{eqnarray}\notag
\bigg|\int_{\R^{3}}(w_{\varepsilon,\phi_{1}}(x)-w_{\varepsilon,\phi_{2}}(x))dx\bigg|\\\notag
\leq c||w_{\varepsilon,\phi_{1}}-w_{\varepsilon,\phi_{2}}||_{C^{4,\alpha}_{\delta}(\R^{3})}\int_{\R^{3}}\varphi_{-\delta}(x)dx\leq c\varepsilon^{3}|\phi_{1}-\phi_{2}|_{C^{4,\alpha}(\Sigma)},
\end{eqnarray}
for any $\phi_{1},\phi_{2}\in C^{4,\alpha}(\Sigma)_{s}$ satisfying $|\phi_{1}|_{C^{4,\alpha}(\Sigma)},|\phi_{2}|_{C^{4,\alpha}(\Sigma)}\leq c\varepsilon$.

\subsection{The proof of Proposition \ref{prop_bifo}}
Before giving the proof, we state a technical Lemma, in which we prove that the term $\tilde{p}_{4}$ is small enough.
\begin{lemma}
For any $\varepsilon>0$ small enough, for any $c>0$ and for any $\phi,\phi_{1},\phi_{2}$ satisfying $|\phi|_{C^{4,\alpha}(\Sigma)},|\phi_{1}|_{C^{4,\alpha}(\Sigma)},|\phi_{2}|_{C^{4,\alpha}(\Sigma)}\leq c\varepsilon$, we have
\begin{eqnarray}\notag
\begin{cases}
|\tilde{p}_{4}(\phi)|_{C^{0,\alpha}(\Sigma)}\leq\tilde{c}\varepsilon^{5}\\\notag
|\tilde{p}_{4}(\phi_{1})-\tilde{p}_{4}(\phi_{2})|_{C^{0,\alpha}(\Sigma)}\leq\tilde{c}\varepsilon^{5}|\phi_{1}-\phi_{2}|_{C^{0,\alpha}(\Sigma)},
\end{cases}
\end{eqnarray}
for some constant $\tilde{c}>0$.
\label{lemmaR}
\end{lemma}
\begin{proof}
We write $U_{\varepsilon,\phi}=(U_{\varepsilon,\phi})_{o}+(U_{\varepsilon,\phi})_{e}$. By Proposition \ref{propaux_2}, we know that $||(U_{\varepsilon,\phi})_{o}||_{C^{4,\alpha}_{\delta}(\Sigma_{\varepsilon}\times\R)}\leq c\varepsilon^{4}$, therefore $||R_{\varepsilon,\phi}((U_{\varepsilon,\phi})_{o})||_{C^{4,\alpha}_{\delta}(\Sigma_{\varepsilon}\times\R)}\leq c\varepsilon^{5}$, since all the coefficients of $\text{R}_{\varepsilon,\phi}$ are at least of order $\varepsilon$. It remains to deal with the even part $U_{e}$. We will see that all the terms of order $\varepsilon^{4}$ in the expression of $\text{R}_{\varepsilon,\phi}(U_{\varepsilon,\phi})_{e}$ will vanish after projection. This can be seen by a direct computation
\begin{eqnarray}\notag
\chi_{4}R_{\varepsilon,\phi}((U_{\varepsilon,\phi})_{e})=\varepsilon\chi_{4}\big\{ HW^{'''}(v_{\star})v_{\star}^{'}(U_{\varepsilon,\phi})_{e}-\Delta_{\Sigma_{\varepsilon}}(\partial_{t}(U_{\varepsilon,\phi})_{e}
+a_{1}^{ij}\partial_{ij}(U_{\varepsilon,\phi})_{e}t)\\\notag
+H\partial_{ttt}(U_{\varepsilon,\phi})_{e}+W^{''}(v_{\star})(H\partial_{t}(U_{\varepsilon,\phi})_{e}
+a_{1}^{ij}\partial_{ij}(U_{\varepsilon,\phi})_{e}t)\\\notag
+(H\partial_{t}(U_{\varepsilon,\phi})_{e}+a_{1}^{ij}\partial_{ij}(U_{\varepsilon,\phi})_{e}t)\mathcal{L}_{\varepsilon}(U_{\varepsilon,\phi})_{e}
+\tilde{R}_{\varepsilon,\phi}((U_{\varepsilon,\phi})_{e})\big\},
\end{eqnarray}
where $\tilde{R}_{\varepsilon,\phi}((U_{\varepsilon,\phi})_{e})$ is some linear operator with coefficients of order at least $\varepsilon^{2}$. All the terms of order $\varepsilon$ are odd, thus they vanish when we multiply by $v_{\star}^{'}$ and integrate, the other ones give rise to terms of order $\varepsilon^{5}$, being $U_{\varepsilon,\phi}$ of order $\varepsilon^{3}$.
\end{proof}
Now we are ready to prove Proposition \ref{prop_bifo}.
\begin{proof}
In view of Proposition \ref{solve_eq_bifo}, the system of equations (\ref{eqbifo_1}) and (\ref{eqbifo_2}) is equivalent to the fixed point problem
\begin{eqnarray}\notag
\phi=T_{3}(\phi):=-P\bigg(\mathcal{L}^{-1}\bigg(\varepsilon c_{\star}^{-1}\mathcal{F}_{\varepsilon,\phi}
+\varepsilon^{-4}\big\{\tilde{p}_{1}(\phi)+\tilde{p}_{2}(\phi)+\tilde{p}_{3}(\phi)+\tilde{p}_{4}(\phi)\big\},\\\notag
4\sqrt{2}\pi^{2}\varepsilon\int_{0}^{6+\tau/2\varepsilon}t(1-v_{\star}(t))dt+\varepsilon^{2}G_{\varepsilon}(\phi)\bigg)\bigg),
\end{eqnarray}
where $P:C^{4,\alpha}(\Sigma)_{s}\times\R\to C^{4,\alpha}(\Sigma)_{s}$ is the projection onto the first component. We will show that $T_{3}$ is a contraction on the ball
\begin{eqnarray}\notag
\Lambda_{3}:=\{\phi\in C^{4,\alpha}(\Sigma)_{s}:|\phi|_{C^{4,\alpha}(\Sigma)}<C_{3}\varepsilon\},
\end{eqnarray}
provided $C_{3}$ is large enough.

Using once again the same estimates as in the proof of Proposition \ref{propaux_2} and the fact that Lipschitzianity of $U$ with respect to $\phi$, we can see that $\tilde{p}_{1}$ and $\tilde{p}_{3}$ are exponentially small in $\varepsilon$, that is they satisfy, for instance
\begin{eqnarray}\notag
\begin{cases}
|\tilde{p}_{1}(\phi)|_{C^{0,\alpha}(\Sigma)}\leq ce^{-a/\varepsilon}\\\notag
|\tilde{p}_{1}(\phi_{1})-\tilde{p}_{1}(\phi_{2})|_{C^{0,\alpha}(\Sigma)}\leq ce^{-a/\varepsilon}|\phi_{1}-\phi_{2}|_{C^{4,\alpha}(\Sigma)},
\end{cases}
\end{eqnarray}
for any $\phi,\phi_{1},\phi_{2}\in\Lambda_{3}$. Similarly, by (\ref{quadratic_w}), we can see that
\begin{eqnarray}\notag
\begin{cases}
|\tilde{p}_{2}(\phi)|_{C^{0,\alpha}(\Sigma)}\leq c\varepsilon^{6}\\\notag
|\tilde{p}_{2}(\phi_{1})-\tilde{p}_{2}(\phi_{2})|_{C^{0,\alpha}(\Sigma)}\leq c\varepsilon^{6}|\phi_{1}-\phi_{2}|_{C^{4,\alpha}(\Sigma)}.
\end{cases}
\end{eqnarray}
The term $\varepsilon c_{\star}^{-1}\mathcal{F}_{\varepsilon,\phi}$ is small according to Proposition \ref{solve_eq_bifo}. 
The most difficult term is the one involving $\text{R}_{\varepsilon,\phi}$, since there are some coefficients of order $\varepsilon$ and $U$ is just of order $\varepsilon^{3}$. However, we verified in Lemma \ref{lemmaR} that these terms do not give rise to terms of order $\varepsilon^{4}$ after projection, thanks to the symmetries.

The second component can be treated in a similar way. In fact
\begin{eqnarray}\notag
4\sqrt{2}\pi^{2}\varepsilon\int_{0}^{6+\tau/2\varepsilon}t(1-v_{\star}(t))dt\leq c\varepsilon
\end{eqnarray}
and it is independent of $\phi$. To conclude, $\varepsilon^{2}G_{\varepsilon,\phi}$ is small according to Proposition \ref{prop_int_v}. In conclusion, $T_{3}$ is a contraction of the ball $\lambda_{3}$, provided $C_{3}$ is large enough.
\end{proof}

\end{document}